\documentclass[]{amsart}
\usepackage[top=25mm,bottom=25mm,left=25mm,right=25mm]{geometry}
\usepackage{amsmath}
\usepackage[english]{babel}
\usepackage[utf8]{inputenc}
\usepackage{url}
\usepackage{amsfonts}
\usepackage{amsthm}
\usepackage{graphicx}
\usepackage{color}
\usepackage{booktabs}
\usepackage[table,xcdraw]{xcolor}
\numberwithin{equation}{section}

\newtheorem{thm}{Theorem}[section]
\newtheorem{defn}{Definition}[section]

\newtheorem{prop}[thm]{Proposition}
\newtheorem{cor}[thm]{Corollary}
\newtheorem{lem}[thm]{Lemma}

\theoremstyle{remark}
\newtheorem{rmk}[thm]{Remark}
\theoremstyle{definition}

\DeclareMathOperator{\E}{\mathbb{E}}

\DeclareMathOperator{\N}{\mathbb{N}}
\DeclareMathOperator{\cN}{\mathcal{N}}

\DeclareMathOperator{\R}{\mathbb{R}}
\DeclareMathOperator{\cF}{\mathcal{F}}

\DeclareMathOperator{\cZ}{\mathcal{Z}}

\DeclareMathOperator{\cL}{\mathcal{L}}
\DeclareMathOperator{\cI}{\mathcal{I}}
\DeclareMathOperator{\cB}{\mathcal{B}}
\DeclareMathOperator{\cS}{\mathcal{S}}

\DeclareMathOperator{\fS}{\mathfrak{S}}
\DeclareMathOperator{\bP}{\mathbb{P}}

\newcommand{\pd}[2]{\frac{\partial #1}{\partial #2}}
\newcommand{\pdsup}[3]{\frac{\partial^{#3} #1}{\partial #2^{#3}}}

\newcommand{\der}[2]{\frac{d #1}{d #2}}
\newcommand{\dersup}[3]{\frac{d^{#3} #1}{d #2^{#3}}}

\newcommand{\Norm}[2]{\left\Vert #1 \right\Vert_{#2}}

\title[Deterministic control of SDEs with stochastic drift and multiplicative noise]{Deterministic control of SDEs with stochastic drift and multiplicative noise: a variational approach}
\author{Giacomo Ascione$^\ast$}
\address{$^\ast$ Dipartimento di Matematica e Applicazioni ``Renato Caccioppoli'', Università degli Studi di Napoli Federico II, 80126 Napoli, Italy}
\author{Giuseppe D'Onofrio$^\dagger$}
\address{$^\dagger$ Dipartimento di Matematica ``G. Peano'', Università degli Studi di Torino, Via Carlo Alberto 10, 10123 Torino, Italy}
\email{giacomo.ascione@unina.it \\
 giuseppe.donofrio@unito.it
	}
\begin{document}
\maketitle
\begin{abstract}
We consider a linear stochastic differential equation with stochastic drift and multiplicative noise. We study the problem of approximating its solution with the process that solves the equation where the possibly stochastic drift is replaced by a deterministic function. To do this, we use a combination of deterministic Pontryagin’s maximum principle approach and direct methods of calculus of variations. We find necessary and sufficient conditions for a function $u \in L^1(0, T)$ to be a minimizer of a certain cost functional. To overcome the problem of the existence of such minimizer, we also consider suitable families of penalized coercive cost functionals. Finally, we consider the important example of the quadratic cost functional, showing that the expected value of the drift component is not always the best choice in the mean squared error approximation.
\end{abstract}
\section{Introduction}

Optimal control of dynamical systems consists in the optimization, via a suitable control, of certain measures of performance of the system.
Precisely, assuming that the state of the system is described by a differential equation, we want to {\it  modifiy} the equation with a function (called control) belonging to a suitable class
in order to minimize a certain functional depending on both the controlled state of the system and the control itself.
In the context of stochastic calculus 
this problem extends naturally to the case in which the system is described through a controlled  stochastic differential equation (SDE).
Historically, the latter is addressed by two main theoretical approaches that have been developed starting from Bellman's and Pontryagin's optimality principles
(see for instance the comprehensive survey by Pham \cite{pham_survey}).
The first one is  called the dynamic
programming principle, based on Bellman's optimality principle \cite{bellman}: it consists in defining a dynamic value function by using the cost functional and then trying to describe it via partial differential equations (PDEs).
This method relies on a class of nonlinear PDEs called Hamilton-Jacobi-Bellmann equations \cite{Kalman,lions1983}.
Let us emphasize that one can also adapt the latter to more {\it complex} situation (e.g. \cite{bayraktar2018randomized}).
The second approach, instead,  is based  on a stochastic generalization of Pontryagin's maximum principle \cite{peng}. 
While the deterministic version can be expressed, in some suitable cases, via a forward-backward differential system, the stochastic one led to the definition of  backward stochastic differential equations (BSDEs) \cite{pardouxpeng}.
Let us also stress that the stochastic maximum principle usually works with second variations (while the deterministic one only with first) due to the presence of the white noise. 
This branch of control theory considerably developed over the last
years \cite{Andersson,Bonnans2012FirstAS,dipersio,Fuhrman2013StochasticMP,
Fuhrman2016StochasticMP,MenoukeuPamen2021MaximumPF,orrieri,
STECHA20114714}.
Here we want to address an approximation problem concerning a linear SDE.  Indeed, the tools coming from optimal control theory have been already used to approach some approximation problems. This is done, for instance, in \cite{davis} where a stochastic control problem is approximated by a sequence of deterministic control problems, obtaining a Wong-Zakai like (\cite{wongzakai}) convergence result. Actually we are interested in approximating an SDE admitting a stochastic drift with another one in which such drift is replaced by a deterministic one. 

More precisely, in this paper we consider the following type of linear SDEs
\begin{equation}\label{eq:linSDE_intro}
	\begin{cases}
		dX(t)=[a(t)X(t)+z(t)]dt+X(t)dW(t), & t \in [0,T]\\
		X(0)=X_0,
	\end{cases}
\end{equation}
with multiplicative noise and where $z(t)$, appearing in the drift term, is a suitable stochastic process.
This kind of equations arises in many applications ranging from
finance \cite{oksendal2000} to
neuronal modeling \cite{tanre,delarue} or quickest detection \cite{JOHNSON2021}.
Moreover, if $z(t)$ is itself the solution of an SDE, Eq.\eqref{eq:linSDE_intro} plays a role in many systems of equations used in epidemiology, climate models, game theory and others \cite{arnold1998random,bhattacharya2007random,flandoli}.

We are interested in finding the {\it best} approximation for a solution of
Eq.\eqref{eq:linSDE_intro} obtained by substituting the possibly stochastic drift with a deterministic function. 
A measure of goodness of the approximation is expressed via the cost functional 
\begin{equation}
	J: u \in L^1(0,T) \to \E\left[\int_0^T F(t,|X(t)-X_u(t)|)dt\right],
\end{equation}
where $F$ is a suitably regular function depending on the distance between $X(t)$ and $X_u(t)$, that is the solution of Eq.\eqref{eq:linSDE_intro} where we replace $z$ by $u$.
Let us underline that the  Lagrangian function $F$ does not depend directly on $u$. Usually, this could lead to a trivial solution of a control problem. Triviality is avoided since we are constraining $u$ to  be deterministic.
Our aim is to find, if it exists, a function $\overline{u} \in L^1(0,T)$ that minimizes $J$. In the  literature, to the best of our knowledge, few contributions on purely deterministic controls of stochastic equations are available \cite{ascione2020optimal,stannat}.

In \cite{ascione2020optimal} we considered the problem of
approximating the solution of an SDE with stochastic drift and additive noise through an Ornstein–Uhlenbeck type process, by using direct
methods of calculus of variations. Conditions 
for existence and uniqueness of the approximation and 
bounds on the goodness of the corresponding approximations are given for some examples. 
However, in that work, the presence of just additive noise allowed us to reformulate the problem on the class of absolutely continuous functions and led to a purely deterministic treatment. The multiplicative noise, on the other hand,  requires a different approach.

Here we find necessary and sufficient conditions for a function $u$ in $L^1(0,T)$ to be a minimizer of $J$, while we are not able to prove the existence of such a solution in a general setting. To overcome this problem we consider suitable families of penalized cost functionals and we prove that  they always admit minimizers. With this property in mind we are able to exploit a sufficient (and necessary) condition for the existence of a solution of the original problem. If the latter condition is not clearly satisfied, then, in any case, the original cost functional evaluated in the solutions of the penalized problems converges towards its infimum as the penalization constant goes to zero. On the other hand, if the condition is satisfied,  we can guarantee only weak $L^1$ convergence of the penalized solution towards the actual solution, but under further regularity assumptions we still have convergence in distribution of the corresponding approximated processes.
In the overall, the method we present here can be considered as a combination of deterministic Pontryagin's maximum principle approach and direct methods of calculus of variations.

The paper is structured as follows: in Section 2 we first show some basic properties of the solution of Eq.\eqref{eq:linSDE_intro} and then we introduce the approximation problem. Section 3 is devoted to obtaining the Euler-Lagrange equation of the functional; i.e we give necessary conditions for a function $u$ to be a minimizer of $J$. In Section 4 we prove that, under suitable convexity assumptions, the aforementioned Euler-Lagrange equation is also a sufficient condition. In Section 5 we study the penalized problems and we address the problem of existence of a solution and convergence of the penalized solutions to the actual one. Finally, in Section 6, we consider the important example of the quadratic cost functional. While on one hand we are able to show that if   $z$ is independent of $W$ a solution exists and it is trivially the expected value of $z$, on the other hand we also provide an example in which it is {\bf not} a minimizer for the quadratic cost functional. This result can be reformulated saying that, in the multiplicative noise case, the expected value of $z$ is not always the best choice in the mean squared error approximation.
Due to the non-trivial nature of the Euler-Lagrange equation, all the examples provided in the section have been obtained by using numerical methods for solution of integral equations via MATLAB R2021a \cite{matlab}.

\section{The linear equation with multiplicative noise and the approximation problem}
\subsection{The linear equation}
Let us consider a filtered probability space $(\Omega, \cF, \bP, \cF_t)$ and a $\cF_t$-Brownian motion $\{W(t), t \ge 0 \}$. Fix $T>0$ and consider $\{z(t), t \ge 0\}$ a $\cF_t$-adapted process such that
\begin{itemize}
	\item[(\emph{H1})] There exists $p \ge 2$ such that for any fixed $t \in [0,T]$, $z(t)\in L^p(\Omega,\bP)$ and 
	\begin{equation*}
		\int_0^T\E[|z(t)|^p]^{\frac{2}{p}}dt<+\infty.
	\end{equation*}
\end{itemize}
Let us stress out that last condition implies, by H\"older inequality, that
\begin{equation*}
	\int_0^T\E[|z(t)|^2]dt\le\int_0^T\E[|z(t)|^p]^{\frac{2}{p}}dt<+\infty.
\end{equation*}
On the other hand, let us observe that, denoting $U=\{t \in [0,T]: \ \E[|z(t)|^p]\ge 1\}$ and $U^c=[0,T]\setminus U$, 
\begin{align*}
	\begin{split}
	\int_0^T\E[|z(t)|^p]^{\frac{1}{p}}dt&=\int_{U}\E[|z(t)|^p]^{\frac{1}{p}}dt+\int_{U^c}\E[|z(t)|^p]^{\frac{1}{p}}dt \\&\le \int_{U}\E[|z(t)|^p]^{\frac{2}{p}}dt+|U^c|\le \int_{0}^T\E[|z(t)|^p]^{\frac{2}{p}}dt+T<+\infty.
\end{split}
\end{align*}
Finally, let us observe that
\begin{equation*}\label{contr}
	\E\left[\int_0^T|z(t)|dt\right]=\int_0^T \E[|z(t)|]dt\le \int_0^T \E[|z(t)|^p]^{\frac{1}{p}}dt<+\infty,
\end{equation*}
hence $\int_0^T|z(t)|dt$ is $\bP$-almost surely finite and $z \in L^1(0,T)$ $\bP$-almost surely. Let us denote by $\cL^2_p(\Omega, \bP; [0,T])$ the space of $\cF_t$-adapted processes $\{z(t), \ t \ge 0\}$ satisfying (\emph{H1}). The notation is justified by the fact that (\emph{H1}) can be also written as
\begin{equation*}
	\Norm{\Norm{z(\cdot)}{L^p(\Omega,\bP)}}{L^2(0,T)}<+\infty.
\end{equation*}
Let us also consider a function $a:[0,T]\to \R$ in $L^\infty(0,T)$.
We focus on the linear SDE
\begin{equation}\label{eq:linSDE}
	\begin{cases}
		dX(t)=[a(t)X(t)+z(t)]dt+X(t)dW(t), & t \in [0,T]\\
		X(0)=X_0,
	\end{cases}
\end{equation}
where $X_0 \in L^2(\Omega, \bP)$. In particular the following result holds.
\begin{prop}
	Let $\cL^2_1([0,T]; \Omega, \bP)$ be the space of $\cF_t$-adapted processes $\{z(t), \ t \ge 0\}$ such that $z(\cdot)\in L^1(0,T)$ $\bP$-almost surely and
	\begin{equation*}
		\E\left[\left(\int_0^T|z(t)|dt\right)^2\right]<+\infty.
	\end{equation*}
	Then the map $\cS_{X_0}: \cL^2_1([0,T]; \Omega, \bP) \mapsto \cL^2_2(\Omega, \bP; [0,T])$, such that for any $z \in \cL^2_1([0,T]; \Omega, \bP)$ the process $\cS_{X_0} z$ is solution of \eqref{eq:linSDE}, is well-defined and it holds
	\begin{equation}\label{eq:solutionmap}
		\cS_{X_0} z(t)=G(t)e^{A(t)}\left(X_0+\int_0^t \frac{e^{-A(s)}}{G(s)} z(s)ds\right), \quad \forall z \in \cL^2_1([0,T]; \Omega, \bP),
	\end{equation}
where $A(t)=\int_0^t a(s)ds$ and $G(t)$ is the geometric Brownian motion associated to $W(t)$, i.e.
\begin{equation}\label{eq:geomB}
	G(t)=e^{W(t)-\frac{1}{2}t}
\end{equation}
\end{prop}
\begin{proof}
	By a simple adaptation of the proof of \cite[Theorem $5.2.1$]{oksendal2013stochastic}, the SDE \eqref{eq:linSDE} admits a unique strong solution in $\cL^2_2(\Omega, \bP; [0,T])$ whenever $z \in \cL^2_1([0,T]; \Omega, \bP)$.\\
	Let us prove Equation \eqref{eq:solutionmap}. To do this, let us consider the linear SDE
	\begin{equation*}
		\begin{cases}
			dY(t)=(1-a(t))Y(t)dt-Y(t)dW(t), & t \in [0,T]\\
			Y(0)=1
		\end{cases}
	\end{equation*}
	and define $Z(t)=\log(Y(t))$. By It\^{o}'s formula we have
	\begin{equation*}
		\begin{cases}
			dZ(t)=\left(\frac{1}{2}-a(t)\right)dt-dW(t), & t \in [0,T]\\
			Z(0)=0
		\end{cases}
	\end{equation*}
	and then, integrating
	\begin{equation*}
		Z(t)=\frac{1}{2}t-W(t)-\int_0^ta(s)ds, \ t \in [0,T].
	\end{equation*}
	Recalling the definition of $Z(t)$, we have
	\begin{equation}\label{eq:procY}
		Y(t)=\frac{e^{-A(t)}}{G(t)}, \ t \in [0,T].
	\end{equation}
	Let $X(t)=\cS_{X_0} z(t)$. By It\^o's formula we have
	\begin{equation*}
		d(X(t)Y(t))=Y(t)dX(t)+X(t)dY(t)-X(t)Y(t)dt=z(t)Y(t)dt.
	\end{equation*}
	Integrating the previous relation we have
	\begin{equation*}
		X(t)Y(t)=X_0+\int_0^tY(s)z(s)ds.
	\end{equation*}
		Equation \eqref{eq:procY} concludes the proof.
\end{proof}
\begin{rmk}
	Since Equation \eqref{eq:linSDE} is linear, one could explicitly write the solution in terms of $z(t)$ apparently just supposing that $z \in L^1(0,T)$ $\bP$-almost surely. However, this does not guarantee that $X(t)$ is regular enough to admit an It\^o integral, which is instead needed to express the equation itself. \\
	Let us also observe that $\cL^2_p(\Omega,\bP;[0,T]) \cup L^1(0,T) \subset \cL^2_1([0,T];\Omega, \bP)$, where with $L^1(0,T)$ we denote the space of absolutely integrable deterministic functions, considered as degenerate stochastic processes. 
\end{rmk}
As a direct consequence of the previous result we obtain the following.
\begin{cor}\label{linear}
	The solution map $\cS_{X_0}$ is affine, i.e. for any $n \in \N$, $(a_1,\dots,a_n)\in \R^n$ such that $\sum_{i=1}^{n}a_i=1$ and $z_1,\dots,z_n \in \cL_1^2([0,T];\Omega,\bP)$ it holds
	\begin{equation*}
		\cS_{X_0}\left(\sum_{i=1}^{n}a_iz_i\right)=\sum_{i=1}^{n}a_i\cS_{X_0} z_i.
	\end{equation*}
Moreover, for any $X_1,X_2 \in L^2(\Omega, \bP)$ and $z_1,z_2 \in \cL_1^2([0,T];\Omega, \bP)$ it holds
\begin{equation*}
	\cS_{X_1}z_1-\cS_{X_2}z_2=\cS_{X_1-X_2}(z_1-z_2).
\end{equation*}
Finally, $\cS_0$ is linear.
\end{cor}
\begin{proof}
	Let us just observe that
	\begin{align*}
		\cS_{X_0}\left(\sum_{i=1}^{n}a_iz_i\right)(t)&=G(t)e^{A(t)}\left(X_0+\int_0^t \frac{e^{-A(s)}}{G(s)} \left(\sum_{i=1}^{n}a_iz_i(s)\right)ds\right)\\
		&=G(t)e^{A(t)}\left(\sum_{i=1}^{n}a_iX_0+\sum_{i=1}^{n}a_i\int_0^t \frac{e^{-A(s)}}{G(s)}z_i(s)ds\right)\\
		&=\sum_{i=1}^{n}a_iG(t)e^{A(t)}\left(X_0+\int_0^t \frac{e^{-A(s)}}{G(s)}z_i(s)ds\right)=\sum_{i=1}^{n}a_i\cS_{X_0} z_i(t).
	\end{align*}
	The second and third statements can be proved in an analogous way.
\end{proof}
Next, we want to underline some properties of the moments of $\cS_{X_0} z(t)$ when $z$ belongs to a certain Banach space. To do this, let us introduce the Banach space $\cL_p^1(\Omega,\bP;[0,T])$ of the $\cF_t$-adapted processes $\{z(t), \ t \ge 0\}$ such that
\begin{equation*}
	\int_0^T \E[|z(t)|^p]^{\frac{1}{p}}dt<+\infty.
\end{equation*}
Clearly, we have $\cL_p^2(\Omega, \bP;[0,T])\cup L^1(0,T)\subset \cL_p^1(\Omega,\bP;[0,T]) \cap \cL_1^2([0,T];\Omega,\bP)$.
On the other hand, let us also recall the following moment estimate for linear SDEs (see \cite[Chapter $3$, Lemma $4.2$]{yong1999stochastic}).
\begin{lem}\label{lem:momest}
	Consider $\kappa \ge 1$ and let $Y(t)$ be a strong solution of
	\begin{equation*}
		\begin{cases}
			dY(t)=[a_1(t)Y(t)+a_2(t)]dt+[b_1(t)Y(t)+b_2(t)]dW(t), &t \in [0,T]\\
			Y(0)=Y_0 
		\end{cases}
	\end{equation*}
	where $Y_0 \in L^{2\kappa}(\Omega, \bP)$, $a_1,b_1:[0,T]\to \R$ are functions in $L^\infty(0,T)$ with $M \ge \max\{\Norm{a_1}{L^\infty(0,T)},\Norm{b_1}{L^\infty(0,T)}\}$ and
	\begin{equation*}
		\int_0^T \E[|a_2(t)|^{2\kappa}]^{\frac{1}{2\kappa}}dt+\int_0^T \E[|b_2(t)|^{2\kappa}]^{\frac{1}{\kappa}}dt<+\infty.
	\end{equation*}
	Then there exists a constant $K(\kappa,M,T)>0$ such that
	\begin{equation*}
		\sup_{t \in [0,T]}\E[|Y(t)|^{2\kappa}]\le K(\kappa,M,T)\left(\E[|Y_0|^{2\kappa}]+\left(\int_0^T \E[|a_2(t)|^{2\kappa}]^{\frac{1}{2\kappa}}dt\right)^{2\kappa}+\left(\int_0^T \E[|b_2(t)|^{2\kappa}]^{\frac{1}{\kappa}}dt\right)^{\kappa}\right).
	\end{equation*}
	Moreover, for fixed $\kappa \ge 1$ and $M>0$, the function $T>0 \mapsto K(\kappa,M,T)$ is increasing.
\end{lem}
\begin{rmk}
	Actually, the last statement of the Lemma is a direct consequence of the constructive proof presented in \cite[Chapter $3$, Lemma $4.2$]{yong1999stochastic}.
\end{rmk}
By using the previous Lemma, we have the following result.
\begin{lem}\label{lem:momest2}
	Let $z \in \cL_p^1(\Omega,\bP;[0,T]) \cap \cL_1^2([0,T];\Omega, \bP)$ and $X_0 \in L^p(\Omega,\bP)$ for some $p \ge 2$. Then it holds
	\begin{align*}
		\sup_{t \in [0,T]}\E[|\cS_{X_0} z(t)|^p]<+\infty.
	\end{align*}
Moreover, if $X_0=0$ almost surely and $u \in L^1(0,T)$, then
\begin{align*}
	\sup_{t \in [0,T]}\E[|\cS_{0} u(t)|^p]\le K\left(\frac{p}{2},M,T\right) \Norm{u}{L^1(0,T)}^p,
\end{align*}
where $M=\Norm{a}{L^\infty(0,T)}$ and $K$ is defined in Lemma \ref{lem:momest}.
\end{lem}
\begin{proof}
	Being $p\ge 2$, we can consider $\kappa=\frac{p}{2}\ge 1$. By using Lemma \ref{lem:momest} we have
	\begin{align*}
		\sup_{t \in [0,T]}\E[|\cS_{X_0} z(t)|^p]&\le K\left(\frac{p}{2},M,T\right)\left(\E[|X_0|^p]+\left(\int_0^T \E[|z(t)|^{p}]^{\frac{1}{p}}dt\right)^{p}\right)<+\infty.
	\end{align*}
	The second part of the statement easily follows by the fact that $\E[|X_0|^p]=0$ and $\E[|u(t)|^p]^{\frac{1}{p}}=|u(t)|$.
\end{proof}
\begin{rmk}
	The arguments in the paper can be carried on without the hypothesis (\emph{H1}), but just considering $z \in \cL_p^1(\Omega,\bP;[0,T]) \cap \cL_1^2([0,T];\Omega, \bP)$. Here, for the ease of the reader, we will directly consider $z \in \cL_p^2(\Omega,\bP;[0,T])$.
\end{rmk}
\subsection{Some properties of the Geometric Brownian Motion}
As we have seen in the previous subsection, the Geometric Brownian Motion $G(t)$ defined in Equation \eqref{eq:geomB} will play a major role. Let us first recall that, it being a Doleans-Dade exponential (see \cite[Chapter $1$]{kazamaki2006continuous}) with $G(0)=1$, it is a $\cF_t$-martingale.
On the other hand, we can consider the process
\begin{equation*}\label{eq:conjgeom}
	G'(t)=\frac{e^{-t}}{G(t)}=e^{-W(t)-\frac{1}{2}t}.
\end{equation*}
It is not difficult to check that $G'(t)$ is still a Geometric Brownian motion (by the fact that $-W(t)$ is still a Brownian motion) and it is given by the Doleans-Dade exponential of $-W(t)$. Thus, in particular, also $G'(t)$ is a $\cF_t$-martingale.\\
Concerning the distribution of $G(t)$, let us call back that it is a log-normal process such that, for fixed $t>0$, $\log(G(t))\sim \cN\left(-\frac{1}{2}t,t\right)$. By using the formula of the moment generating function of a Gaussian random variable, it is easy to show that, for any $q \ge 0$,
\begin{equation}\label{eq:momgeom}
	\E[G(t)^q]=e^{\frac{q(q-1)}{2}t}, \qquad  t \ge 0.
\end{equation}
The same relation holds for $G'(t)$. Combining Equation \eqref{eq:momgeom} and Doob's maximal inequality (see \cite[Theorem II.1.7]{revuz2013continuous}) we get the following bound on the supremum of $G$ and $G'$.
\begin{lem}\label{lem:supest}
	Let $p_1,p_2 \ge 0$ and $T>0$. Then there exists a constant $C(p_1,p_2,T)$ such that
	\begin{equation*}
		\E\left[\left(\sup_{t \in [0,T]}G(t)\right)^{p_1}\left(\sup_{t \in [0,T]}G'(t)\right)^{p_2}\right]\le C(p_1,p_2,T).
	\end{equation*}
\end{lem}
\begin{proof}
	By the Cauchy-Schwartz inequality, we have
	\begin{equation*}
		\E\left[\left(\sup_{t \in [0,T]}G(t)\right)^{p_1}\left(\sup_{t \in [0,T]}G'(t)\right)^{p_2}\right]\le \E\left[\left(\sup_{t \in [0,T]}G(t)\right)^{2p_1}\right]^{\frac{1}{2}}\E\left[\left(\sup_{t \in [0,T]}G'(t)\right)^{2p_2}\right]^{\frac{1}{2}}.
	\end{equation*}
Since $2p_1 \ge 0$, we can use Doob's maximal inequality in $L^p$ form to achieve 
\begin{equation*}
	\E\left[\left(\sup_{t \in [0,T]}G(t)\right)^{2p_1}\right]=\E\left[\sup_{t \in [0,T]}G^{2p_1}(t)\right]\le \left(\frac{2p_1}{2p_1-1}\right)^{2p_1}\sup_{t \in [0,T]}\E[G^{2p_1}(t)]=\left(\frac{2p_1}{2p_1-1}\right)^{2p_1}e^{p_1(2p_1-1)T},
\end{equation*}
where we also used equation \eqref{eq:momgeom}. In the same way we have
\begin{equation*}
	\E\left[\left(\sup_{t \in [0,T]}G'(t)\right)^{2p_2}\right]\le \left(\frac{2p_2}{2p_2-1}\right)^{2p_2}e^{p_2(2p_2-1)T}.
\end{equation*}
Setting
\begin{equation*}
	C(p_1,p_2,T)=\left(\frac{2p_1}{2p_1-1}\right)^{p_1}\left(\frac{2p_2}{2p_2-1}\right)^{p_2}e^{\frac{p_1(2p_1-1)+p_2(2p_2-1)}{2}T}
\end{equation*}
we conclude the proof.
\end{proof}
From now on, we will use the symbol $C$ to denote a generic positive constant whose value is not important in our arguments. Whenever we need to underline the dependence of $C$ on some parameters $p_1,\dots,p_n$ we will denote it as $C(p_1,\dots,p_n)$. The only exception is Theorem \ref{thm:strongconv}, in which the constants are indexed to keep track of the dependence on the involved parameters.
\subsection{The approximation problem}
We are interested in finding the \textit{best approximation} for a solution of Equation \eqref{eq:linSDE}, obtained by substituting the possibly stochastic drift with a deterministic function. To do this, let us first introduce a cost functional
\begin{equation*}
	J: u \in L^1(0,T) \to \E\left[\int_0^T F(t,\xi_u(t))dt\right],
\end{equation*}
where $F$ is a suitable function and $\xi_u=\cS_0(z-u)$. Let us consider the following assumptions on $F$:
\begin{itemize}
	\item[(\emph{H2})] It holds $F(t,\xi)\ge 0$ for any $t \in [0,T]$ and $\xi \in \R$;
	\item[(\emph{H3})] $F(t,\xi)$ is twice continuously differentiable in the $\xi$ variable and $\pd{F}{\xi}(t,\xi)$ and $\pdsup{F}{\xi}{2}(t,\xi)$ are continuous functions of both variables;
	\item[(\emph{H4})] There exist $\alpha \in (0,p)$ and a non-negative function $L \in L^1(0,T)$ such that
	\begin{equation*}
		|F(t,\xi)|+\left|\pd{F}{\xi}(t,\xi)\right|+\left|\pdsup{F}{\xi}{2}(t,\xi)\right|\le L(t)(1+|\xi|^\alpha), \ t \in [0,T], \ \xi \in \R.
	\end{equation*}
\end{itemize}
Our aim is to find, if it exists, a function $\overline{u} \in L^1(0,T)$ such that 
\begin{equation*}
	J[\overline{u}]=\min_{u \in L^1(0,T)}J[u].
\end{equation*}
We can consider the functional $J[u]$ to be a cost functional for an approximation problem. Indeed, we want to find a deterministic function $u(t)$ that we can substitute to the process $z(t)$ in $X(t)=\cS_{X_0} z(t)$ to obtain \textit{the best possible approximation} under the cost $J$. For this reason we expect the cost functional to depend in some sense on the gap between $X(t)$ and the approximating process $X_u(t)=\cS_{X_0}u(t)$. By affinity of the solution map, we have that $\xi_u(t):=X(t)-X_u(t)=\cS_0(z-u)(t)$. With this idea in mind, the function $F$ can be seen as a running cost.\\
Hypothesis (\emph{H2}) is natural as we want to consider $J[u]$ as a cost functional for an approximation problem, while (\emph{H3}) is just a regularity assumption. Hypothesis (\emph{H4}) implies some form of controlled growth for both the running cost $F$ and its first and second derivatives with respect to the gap process $\xi_u$. The growth assumption on $F$ can be justified by means of the following non-triviality result.
\begin{lem}
	For any $u \in L^1(0,T)$ it holds $J[u]<+\infty$.
\end{lem}
\begin{proof}
	We have
	\begin{equation*}
		J[u]=\E\left[\int_0^T F(t,\xi_u(t))dt\right]\le \int_0^TL(t)(1+\E[|\xi_u(t)|^\alpha])dt.
	\end{equation*}
	Now let us estimate $\E[|\xi_u(t)|^{\alpha}]$. To do this, let us consider $\widetilde{p}=\frac{p}{\alpha}>1$ and let us apply H\"older inequality to achieve
	\begin{equation*}
		\E[|\xi_u(t)|^{\alpha}]\le \E[|\xi_u(t)|^p]^{\frac{\alpha}{p}}.
	\end{equation*}
		Let $C=\sup_{t \in [0,T]}\E[|\xi_u(t)|^p]$, that is finite by Lemma \ref{lem:momest2}. Then we have
		\begin{equation*}
			\E[|\xi_u(t)|^{\alpha}]\le C^{\frac{\alpha}{p}}.
		\end{equation*}
	Thus we have
	\begin{equation*}
		J[u]\le (1+K^{\frac{\alpha}{p}})\int_0^TL(t)dt<+\infty,
	\end{equation*}
	being $L \in L^1(0,T)$.
\end{proof}
The previous result and hypothesis (\emph{H2}) guarantee that
\begin{equation*}
	\inf_{u \in L^1(0,T)}J[u]\in [0,+\infty)
\end{equation*} 
thus it makes sense to search for a minimizer (if it exists) of $J[u]$. Next section will clarify the role of the first and second derivatives in hypothesis (\emph{H4}).
\section{Necessary optimality conditions}
Now let us focus on necessary optimality conditions, i.e. conditions that a global minimizer $\overline{u} \in L^1(0,T)$ of the functional $J[u]$ has to satisfy. Let us stress out that, in order to discuss necessary optimality conditions, we assume that we already have a minimizer $\overline{u}$. In particular, necessary optimality conditions are needed to find at least a set of candidate minimizers.\\
Before going into details, let us introduce some notation. Let $f \in L^1(0,T)$. We denote the set of its Lebesgue points as $E_f$ (see \cite[Section $1.7$]{evans2015measure}). From now on, since $f \in L^1(0,T)$ is almost everywhere finite, we will always consider a version that is everywhere finite, so that, for each $t \in E_f$, it holds
\begin{equation*}
	\lim_{\varepsilon \to 0^+}\frac{1}{\varepsilon}\int_{t-\frac{\varepsilon}{2}}^{t+\frac{\varepsilon}{2}}f(\tau)d\tau=f(t)
\end{equation*}
and we set, for each $t \not \in E_f$, $f(t)=0$. Such version of $f(t)$ is called precise representative of $f$. To obtain necessary optimality conditions we need the following property:

\begin{prop} \label{prop_leb}
Let $f \in L^1(0,T)$ and $g:[0,T]\to \R$ be a continuous function. Define $h(t)=g(t)f(t)$ for any $t \in [0,T]$. Then $h \in L^1(0,T)$ and $E_f \subseteq E_h$.
\end{prop}

The previous statement is classical, but, for completeness, we add its proof in Appendix \ref{AppendixLeb}.\\
Now we are ready to prove the main result of this section.
\begin{thm}\label{thm:EL}
	Suppose hypotheses (H1) to (H4) are satisfied. Let $\overline{u}\in L^1(0,T)$ be a global minimum of the functional $J$ over $L^1(0,T)$.	Then it holds
	\begin{equation}\label{eq:EL}
		\int_{t_0}^T\E\left[\pd{F}{\xi}(t,\xi_u(t))\frac{G(t)}{G(t_0)}e^{(A(t)-A(t_0))}\right]dt=0, \ t_0 \in [0,T].	
	\end{equation}
\end{thm}
\begin{proof}
	Let us consider $E_{\overline{u}}$ the set of Lebesgue points of $\overline{u}$ in $(0,T)$, $E_L$ the set of Lebesgue points of $L$ in $(0,T)$, $E=E_{\overline{u}} \cup E_L$ and let $t_0 \in E$. Fix any real number $u \in \R$ and $\varepsilon_0>0$ small enough to have $\left(t_0-\frac{\varepsilon_0}{2},t_0+\frac{\varepsilon_0}{2}\right)\subset (0,T)$. Now let us define, for any $\varepsilon \in (0,\varepsilon_0)$, $I_\varepsilon:=\left(t_0-\frac{\varepsilon}{2},t_0+\frac{\varepsilon}{2}\right)$, the following needle variation
	\begin{equation*}
		u_\varepsilon(t)=\begin{cases}
			u & t \in I_\varepsilon,\\
			\overline{u}(t) & \mbox{ otherwise.}
		\end{cases}
	\end{equation*}
	and, denoting $\xi_\varepsilon(t):=\xi_{u_\varepsilon}(t)$, the value of the cost corresponding to a certain choice of $\varepsilon$
	\begin{equation*}
		g(\varepsilon)=J[u_\varepsilon]=\E\left[\int_0^T F(t,\xi_\varepsilon(t))dt\right].
	\end{equation*}
	We want to show that $g$ is  right differentiable in $0$.\\
	Let us consider the auxiliary function
	\begin{equation*}
		h: \theta \in (0,1)\mapsto F(t,\theta \xi_{\overline{u}}(t)+(1-\theta)\xi_\varepsilon(t))
	\end{equation*}
	for fixed $t \in [0,T]$ and $\varepsilon \in (-\varepsilon_0,\varepsilon_0)$. Let us observe that, by hypothesis (\emph{H3}), $h$ is twice differentiable and
	\begin{align*}
		h'(\theta)&=\pd{F}{\xi}(t,\theta \xi_{\overline{u}}(t)+(1-\theta)\xi_\varepsilon(t))(\xi_{\overline{u}}(t)-\xi_\varepsilon(t))\\
		h''(\theta)&=\pdsup{F}{\xi}{2}(t,\theta \xi_{\overline{u}}(t)+(1-\theta)\xi_\varepsilon(t))(\xi_{\overline{u}}(t)-\xi_\varepsilon(t))^2.
	\end{align*}
	By using the Fundamental Theorem of Calculus and then integrating by parts, we have
	\begin{equation*}
		h(0)-h(1)=-\int_0^1h'(\theta)d\theta=-h'(1)+\int_0^1\theta h''(\theta)d\theta
	\end{equation*}
	that is to say
	\begin{align*}
		F(t,\xi_\varepsilon(t))-F(t,\xi_{\overline{u}}(t))&=\pd{F}{\xi}(t,\xi_{\overline{u}}(t))(\xi_\varepsilon(t)-\xi_{\overline{u}}(t))\\&\qquad+\int_0^1\theta \pdsup{F}{\xi}{2}(t,\theta \xi_{\overline{u}}(t)+(1-\theta)\xi_\varepsilon(t))(\xi_{\overline{u}}(t)-\xi_\varepsilon(t))^2d\theta.
	\end{align*}
	We can use the previous representation of $F(t,\xi_\varepsilon(t))$ to rewrite the incremental ratio of $g$ in $0$. Indeed, denoting $\eta_\varepsilon(t):=\xi_{\overline{u}}(t)-\xi_\varepsilon(t)$, we get
	\begin{align}
		\label{eq:passg1}
		\begin{split}
			\frac{g(\varepsilon)-g(0)}{\varepsilon}&=\frac{1}{\varepsilon}\E\left[\int_0^T\left( -\pd{F}{\xi}(t,\xi_{\overline{u}}(t))\eta_\varepsilon(t)\right.\right.\\&\qquad\left.\left.+\int_0^1\theta \pdsup{F}{\xi}{2}(t,\theta \xi_{\overline{u}}(t)+(1-\theta)\xi_\varepsilon(t))\eta_\varepsilon^2(t)d\theta \right)dt\right].
		\end{split}
	\end{align}
By the properties of the solution map $\cS_{0}$ given in Corollary \ref{linear}, we have that
	\begin{equation*}
		\eta_\varepsilon(t)=\cS_0(u_\varepsilon-\bar{u})(t)=\begin{cases}
			0 & t \le t_0-\frac{\varepsilon}{2}\\
			G(t)e^{A(t)}\int_{t_0-\frac{\varepsilon}{2}}^t\frac{e^{-A(s)}}{G(s)}[u-\overline{u}(s)]ds & t \in I_\varepsilon \\
			G(t)e^{A(t)}\int_{I_\varepsilon}\frac{e^{-A(s)}}{G(s)}[u-\overline{u}(s)]ds & t\ge t_0+\frac{\varepsilon}{2}
		\end{cases}
	\end{equation*}
	On the other hand, let us observe that
	\begin{equation*}
		\Norm{u_\varepsilon-\bar{u}}{L^1(0,T)}=\int_{I_\varepsilon}|u-\bar{u}(\tau)|d\tau
	\end{equation*}
	and then, by the second part of Lemma \ref{lem:momest2} we have, for any exponent $\kappa \ge 2$,
	\begin{equation}\label{eq:mometa}
		\sup_{t \in [0,T]}\E[|\eta_\varepsilon(t)|^{\kappa}]\le C\left(\int_{I_\varepsilon}|u-\bar{u}(\tau)|d\tau\right)^{\kappa}.
	\end{equation}
	Going back to Equation \eqref{eq:passg1}, let us split the integral as
	\begin{align}
		\label{eq:passg2}
		\begin{split}
			\frac{g(\varepsilon)-g(0)}{\varepsilon}&=-\frac{1}{\varepsilon}\E\left[\int_{I_\varepsilon} \pd{F}{\xi}(t,\xi_{\overline{u}}(t))\eta_\varepsilon(t)dt\right]\\
			&\qquad -\frac{1}{\varepsilon}\E\left[\int_{t_0+\frac{\varepsilon}{2}}^{T} \pd{F}{\xi}(t,\xi_{\overline{u}}(t))\eta_\varepsilon(t)dt\right]\\
			&\qquad+\frac{1}{\varepsilon}\E\left[\int_0^T\int_0^1\theta \pdsup{F}{\xi}{2}(t,\theta \xi_{\overline{u}}(t)+(1-\theta)\xi_\varepsilon(t))\eta_\varepsilon^2(t)d\theta dt\right]\\
			&=:I_1(\varepsilon)+I_2(\varepsilon)+I_3(\varepsilon),
		\end{split}
	\end{align}
	where we used the fact that $\eta_\varepsilon(t)=0$ as $t \le t_0-\frac{\varepsilon}{2}$. Now we want to take the limit as $\varepsilon \to 0$.\\
	First of all, let us show that $\lim_{\varepsilon \to 0}I_1(\varepsilon)=0$. To do this, let us observe that
	\begin{align*}
		|I_1(\varepsilon)|&\le \frac{1}{\varepsilon}\E\left[\int_{I_\varepsilon}\left|\pd{F}{\xi}(t,\xi_{\overline{u}}(t))\eta_\varepsilon(t)\right|dt\right]\\
		&=\frac{1}{\varepsilon}\int_{I_\varepsilon}\E\left[\left|\pd{F}{\xi}(t,\xi_{\overline{u}}(t))\eta_\varepsilon(t)\right|\right]dt.
	\end{align*}
	Bringing back the exponent $\alpha \in (0,p)$ in hypothesis (\emph{H4}), let us consider any $\widetilde{p} \in \left(1, \min\left\{2,\frac{p}{\alpha}\right\} \right)$. Let $\widetilde{q}$ be its conjugate exponent, i.e. such that $\frac{1}{\widetilde{p}}+\frac{1}{\widetilde{q}}=1$. Being $\widetilde{p}<2$ we have $\widetilde{q}>2$ . By H\"{o}lder's inequality it holds
	\begin{align}\label{eq:I1est0}
		\begin{split}
			|I_1(\varepsilon)|&\le \frac{1}{\varepsilon}\int_{I_\varepsilon}\left(\E\left[\left|\pd{F}{\xi}(t,\xi_{\overline{u}}(t))\right|^{\widetilde{p}}\right]\right)^{\frac{1}{\widetilde{p}}}\left(\E\left[\left|\eta_\varepsilon(t)\right|^{\widetilde{q}}\right]\right)^{\frac{1}{\widetilde{q}}}dt.
		\end{split}
	\end{align}
	Using the growth hypothesis (\emph{H4}) we have
	\begin{align*}
		\E\left[\left|\pd{F}{\xi}(t,\xi_{\overline{u}}(t))\right|^{\widetilde{p}}\right]&\le L^{\widetilde{p}}(t)\E\left[(1+|\xi_{\overline{u}}(t)|^\alpha)^{\widetilde{p}}\right]\\
		&\le L^{\widetilde{p}}(t)2^{\widetilde{p}-1}(1+\E[|\xi_{\overline{u}}(t)|^{\alpha\widetilde{p}}])\\
		&\le L^{\widetilde{p}}(t)2^{\widetilde{p}-1}(1+(\E[|\xi_{\overline{u}}(t)|^{p}])^{\frac{\alpha \widetilde{p}}{p}}),
	\end{align*}
	where we used the convexity of the function $x \mapsto x^{\widetilde{p}}$ and applied H\"older's inequality a second time with the exponent $\frac{p}{\alpha \widetilde{p}}>1$. Taking the supremum as $t \in [0,T]$ on the right-hand side and using Lemma \ref{lem:momest2} we finally achieve
	\begin{equation}\label{eq:I1est1}
		\E\left[\left|\pd{F}{\xi}(t,\xi_{\overline{u}}(t))\right|^{\widetilde{p}}\right]\le CL^{\widetilde{p}}(t)
	\end{equation}
	where $C$ is a positive constant. On the other hand, by Equation \eqref{eq:mometa}, we get
	\begin{equation}\label{eq:I1est2}
		\left(\E\left[\left|\eta_\varepsilon(t)\right|^{\widetilde{q}}\right]\right)^{\frac{1}{\widetilde{q}}}\le C\left(\int_{I_\varepsilon}|u-\bar{u}(\tau)|d\tau\right).
	\end{equation}
	Combining equations \eqref{eq:I1est0}, \eqref{eq:I1est1} and \eqref{eq:I1est2} we get
	\begin{equation*}
		|I_1(\varepsilon)|\le C\varepsilon\left(\frac{1}{\varepsilon}\int_{I_\varepsilon}L(t)dt\right)\left(\frac{1}{\varepsilon}\int_{I_\varepsilon}|u-\bar{u}(t)|dt\right).
	\end{equation*}
	It is not difficult to see that if $t_0 \in E_{\bar{u}}$, then it is also a Lebesgue point for $|u-\bar{u}(t)|$. Thus, being $t_0 \in E$, we conclude that $\lim_{\varepsilon \to 0}I_1(\varepsilon)=0$.\\
	Now let us show that $\lim_{\varepsilon \to 0}I_3(\varepsilon)=0$. Arguing as before we have
	\begin{align}\label{eq:I3est0}
		\begin{split}
			|I_3(\varepsilon)|&\le \frac{1}{\varepsilon}\int_0^T\int_0^1\theta \E\left[\left|\pdsup{F}{\xi}{2}(t,\theta \xi_{\overline{u}}(t)+(1-\theta)\xi_\varepsilon(t))\eta_\varepsilon^2(t)\right|\right]d\theta dt\\
			&\le \frac{1}{\varepsilon}\int_0^T\int_0^1\theta \left(\E\left[\left|\pdsup{F}{\xi}{2}(t,\theta \xi_{\overline{u}}(t)+(1-\theta)\xi_\varepsilon(t))\right|^{\widetilde{p}}\right]\right)^{\frac{1}{\widetilde{p}}}\left(\E\left[\left|\eta_\varepsilon(t)\right|^{2\widetilde{q}}\right]\right)^{\frac{1}{\widetilde{q}}}d\theta dt,
		\end{split}
	\end{align}
	where we used H\"older's inequality with the exponent $\widetilde{p}$. Again, using Hypothesis (\emph{H4}), 
	\begin{align*}
		\E\left[\left|\pdsup{F}{\xi}{2}(t,\theta \xi_{\overline{u}}(t)+(1-\theta)\xi_\varepsilon(t))\right|^{\widetilde{p}}\right]&\le L^{\widetilde{p}}(t)\E[(1+|\theta \xi_{\overline{u}}(t)+(1-\theta)\xi_\varepsilon(t)|^\alpha)^{\widetilde{p}}]\\
		&\le L^{\widetilde{p}}(t)2^{\widetilde{p}-1}(1+\E[|\theta \xi_{\overline{u}}(t)+(1-\theta)\xi_\varepsilon(t)|^{\alpha\widetilde{p}}])\\
		&\le L^{\widetilde{p}}(t)2^{\widetilde{p}-1}(1+\E[|\theta \xi_{\overline{u}}(t)+(1-\theta)\xi_\varepsilon(t)|^{p}]^{\frac{\alpha \widetilde{p}}{p}}),
	\end{align*}
	where we also used the convexity of the function $x \mapsto |x|^{\widetilde{p}}$ and H\"older's inequality with the exponent $\frac{p}{\alpha \widetilde{p}}>1$. Moreover, by the properties of the solution map $\cS_0$ as in Corollary \ref{linear},
	\begin{equation*}
		\theta \xi_{\overline{u}}(t)+(1-\theta)\xi_\varepsilon(t)=\cS_0(\theta \overline{u}+(1-\theta)u_\varepsilon)(t)
	\end{equation*}
	and then, by Lemma \ref{lem:momest2}, we conclude
	\begin{align}\label{eq:I3est2}
		\begin{split}
		\E\left[\left|\pdsup{F}{\xi}{2}(t,\theta \xi_{\overline{u}}(t)+(1-\theta)\xi_\varepsilon(t))\right|^{\widetilde{p}}\right]&\le L^{\widetilde{p}}(t)\E[(1+|\theta \xi_{\overline{u}}(t)+(1-\theta)\xi_\varepsilon(t)|^\alpha)^{\widetilde{p}}]\\
		&\le L^{\widetilde{p}}(t)2^{\widetilde{p}-1}(1+\E[|\theta \xi_{\overline{u}}(t)+(1-\theta)\xi_\varepsilon(t)|^{\alpha\widetilde{p}}])\\
		&\le CL^{\widetilde{p}}(t).
		\end{split}
	\end{align}
	On the other hand, by Equation \eqref{eq:mometa} we know that
	\begin{equation}\label{eq:I3est1}
		\left(\E\left[\left|\eta_\varepsilon(t)\right|^{2\widetilde{q}}\right]\right)^{\frac{1}{\widetilde{q}}}\le K\left(\int_{I_\varepsilon}|u-\bar{u}(\tau)|d\tau\right)^2.
	\end{equation}
	Combining Equations \eqref{eq:I3est0}, \eqref{eq:I3est1} and \eqref{eq:I3est2} we obtain
	\begin{equation*}
		|I_3(\varepsilon)|\le \frac{C}{\varepsilon} \left(\int_{I_\varepsilon}|u-\bar{u}(\tau)|d\tau\right)^2\int_0^T\int_0^1\theta L(t)dt \le C\varepsilon \Norm{L}{L^1(0,T)}\left(\frac{1}{\varepsilon}\int_{I_\varepsilon}|u-\bar{u}(\tau)|d\tau\right)^2.
	\end{equation*}
	Taking the limit as $\varepsilon \to 0$ we conclude that $\lim_{\varepsilon \to 0}I_3(\varepsilon)=0$.\\
	Finally, we need to evaluate $\lim_{\varepsilon \to 0}I_2(\varepsilon)$. To do this, let us first show that we can use Fubini's theorem to exchange the order of expectation and Lebesgue integral. Indeed we have
	\begin{align*}
		\int_{t_0}^{T}\E\left[ \left|\pd{F}{\xi}(t,\xi_{\overline{u}}(t))\eta_\varepsilon(t)\right|\right]\mathbf{1}_{\left[t_0+\frac{\varepsilon}{2},T\right]}(t)dt&\le \int_{t_0}^{T}\left(\E\left[ \left|\pd{F}{\xi}(t,\xi_{\overline{u}}(t))\right|^{\widetilde{p}}\right]\right)^{\frac{1}{\widetilde{p}}}\E\left[\left|\eta_\varepsilon(t)\right|^{\widetilde{q}}\right]^{\frac{1}{\widetilde{q}}}dt\\
		&\le C\left(\int_{I_\varepsilon}|u-\bar{u}(\tau)|d\tau\right) \Norm{L}{L^1(0,T)}\\
		&\le C \Norm{u-\bar{u}}{L^1(0,T)}\Norm{L}{L^1(0,T)},
	\end{align*}
	where, for any $B\subseteq [0,T]$, $\mathbf{1}_{B}$ is the indicator function of the set B and we used again Equations \eqref{eq:I1est1} and \eqref{eq:I1est2}. Hence, by Fubini's theorem, we have
	\begin{align}\label{eq:I2est0}
		\begin{split}
			I_2(\varepsilon)&=-\frac{1}{\varepsilon}\int_{t_0}^{T}\E\left[ \pd{F}{\xi}(t,\xi_{\overline{u}}(t))\eta_\varepsilon(t)\right]\mathbf{1}_{\left[t_0+\frac{\varepsilon}{2},T\right]}(t)dt\\
			&=-\int_{t_0}^{T}\E\left[ \mathbf{1}_{\left[t_0+\frac{\varepsilon}{2},T\right]}(t)\pd{F}{\xi}(t,\xi_{\overline{u}}(t))G(t)e^{A(t)}\frac{1}{\varepsilon}\int_{I_\varepsilon}\frac{e^{-A(s)}}{G(s)}[u-\overline{u}(s)]ds\right]dt,
		\end{split}
	\end{align}
	where we explicitly wrote $\eta_\varepsilon(t)$.
	Now let us show that we can take the limit inside both the integral and the expectation sign. To do this, we want to use dominated convergence theorem. Let us observe that
	\begin{align*}\label{eq:I2est10}
		\begin{split}
			&\left|\mathbf{1}_{\left[t_0+\frac{\varepsilon}{2},T\right]}(t)\pd{F}{\xi}(t,\xi_{\overline{u}}(t))G(t)e^{A(t)}\frac{1}{\varepsilon}\int_{I_\varepsilon}\frac{e^{-A(s)}}{G(s)}[u-\overline{u}(s)]ds\right|\\
			&\qquad \qquad \le e^T\left(\sup_{t \in [0,T]}e^{2A(t)}\right)\left(\sup_{t \in [0,T]}G(t)\right)\left(\sup_{t \in [0,T]}G'(t)\right)\left(\frac{1}{\varepsilon}\int_{I_\varepsilon}|u-\bar{u}(s)|ds\right)\left|\pd{F}{\xi}(t,\xi_{\overline{u}}(t))\right|\\
			&\qquad \qquad =C\left(\sup_{t \in [0,T]}G(t)\right)\left(\sup_{t \in [0,T]}G'(t)\right)\left(\frac{1}{\varepsilon}\int_{I_\varepsilon}|u-\bar{u}(s)|ds\right)\left|\pd{F}{\xi}(t,\xi_{\overline{u}}(t))\right|.
		\end{split}
	\end{align*}
	In particular we have
	\begin{equation*}
		\lim_{\varepsilon \to 0}\left(\frac{1}{\varepsilon}\int_{I_\varepsilon}|u-\bar{u}(s)|ds\right)=|u-\bar{u}(t_0)|,
	\end{equation*}
	hence we can suppose $\varepsilon$ is small enough to have
	\begin{equation*}
		\left(\frac{1}{\varepsilon}\int_{I_\varepsilon}|u-\bar{u}(s)|ds\right)\le 2|u-\bar{u}(t_0)|.
	\end{equation*}
	This implies
	\begin{align}\label{eq:I2est1}
		\begin{split}
			&\left|\mathbf{1}_{\left[t_0+\frac{\varepsilon}{2},T\right]}(t)\pd{F}{\xi}(t,\xi_{\overline{u}}(t))G(t)e^{A(t)}\frac{1}{\varepsilon}\int_{I_\varepsilon}\frac{e^{-A(s)}}{G(s)}[u-\overline{u}(s)]ds\right|\\
			&\qquad \qquad \le C\left(\sup_{t \in [0,T]}G(t)\right)\left(\sup_{t \in [0,T]}G'(t)\right)\left|\pd{F}{\xi}(t,\xi_{\overline{u}}(t))\right|.
		\end{split}
	\end{align} 
	Now let us show that the stochastic process on the right-hand side is integrable. Observe that
	\begin{align*}
		&\E\left[\left(\sup_{t \in [0,T]}G(t)\right)\left(\sup_{t \in [0,T]}G'(t)\right)\left|\pd{F}{\xi}(t,\xi_{\overline{u}}(t))\right|\right] \\&\qquad \qquad \le \E\left[\left(\sup_{t \in [0,T]}G(t)\right)^{\widetilde{q}}\left(\sup_{t \in [0,T]}G'(t)\right)^{\widetilde{q}}\right]^{\frac{1}{\widetilde{q}}}\E\left[\left|\pd{F}{\xi}(t,\xi_{\overline{u}}(t))\right|^{\widetilde{p}}\right]^{\frac{1}{\widetilde{p}}}\le CL(t),
	\end{align*}
	where we used Equation \eqref{eq:I1est1} and Lemma \ref{lem:supest}.
Integrating on $[t_0,T]$ we conclude that
	\begin{equation*}
		\int_{t_0}^{T}\E\left[\left(\sup_{t \in [0,T]}G(t)\right)\left(\sup_{t \in [0,T]}G'(t)\right)\left|\pd{F}{\xi}(t,\xi_{\overline{u}}(t))\right|\right]\le C\Norm{L}{L^1(0,T)}.
	\end{equation*}
	We only need to show that the integrand in Equation \eqref{eq:I2est0} converges almost everywhere. Recalling that $t \mapsto G(t)$ is almost surely continuous, fix $\omega \in \Omega$ such that $G(\cdot, \omega)$ is a continuous function, then $t_0 \in E$ is a Lebesgue point for $\frac{e^{-A(t)}}{G(t,\omega)}[u-\overline{u}(t)]$ by Proposition \ref{prop_leb}. Hence we have
	\begin{multline*}
		\lim_{\varepsilon \to 0}\mathbf{1}_{\left[t_0+\frac{\varepsilon}{2},T\right]}(t)\pd{F}{\xi}(t,\xi_{\overline{u}}(t))G(t)e^{A(t)}\frac{1}{\varepsilon}\int_{I_\varepsilon}\frac{e^{-A(s)}}{G(s)}[u-\overline{u}(s)]ds\\=  \pd{F}{\xi}(t,\xi_{\overline{u}}(t))G(t)e^{A(t)}\frac{e^{-A(t_0)}}{G(t_0)}[u-\overline{u}(t_0)] \mbox{ a.s.}
	\end{multline*}
	Thus, by Dominated Convergence Theorem, we get
	\begin{equation*}
		\lim_{\varepsilon \to 0}I_2(\varepsilon)=-(u-\overline{u}(t_0))\int_{t_0}^T \E\left[\pd{F}{\xi}(t,\xi_{\overline{u}}(t))e^{A(t)-A(t_0)}\frac{G(t)}{G(t_0)}\right]dt.
	\end{equation*}
	In conclusion, from Equation \eqref{eq:passg2} we have
	\begin{equation}\label{eq:derg}
		\lim_{\varepsilon \to 0}\frac{g(\varepsilon)-g(0)}{\varepsilon}=-(u-\overline{u}(t_0))\int_{t_0}^T \E\left[\pd{F}{\xi}(t,\xi_{\overline{u}}(t))e^{A(t)-A(t_0)}\frac{G(t)}{G(t_0)}\right]dt.
	\end{equation}
However, by definition of $\overline{u}$, we know that $0$ is a minimum point for $g$, and then 
\begin{equation*}
\lim_{\varepsilon \rightarrow 0} \frac{g(\varepsilon)-g(0)}{\varepsilon}\geq 0,
\end{equation*}
that implies 
\begin{equation*}
u\int_{t_0}^T \E\left[\pd{F}{\xi}(t,\xi_{\overline{u}}(t))e^{A(t)-A(t_0)}\frac{G(t)}{G(t_0)}\right]dt \leq \overline{u}(t_0)\int_{t_0}^T \E\left[\pd{F}{\xi}(t,\xi_{\overline{u}}(t))e^{A(t)-A(t_0)}\frac{G(t)}{G(t_0)}\right]dt.
\end{equation*}
Define $H(u):=u\int_{t_0}^T \E\left[\pd{F}{\xi}(t,\xi_{\overline{u}}(t))e^{A(t)-A(t_0)}\frac{G(t)}{G(t_0)}\right]dt$ and observe that $\overline{u}(t_0)$ is a maximum point for $H(u)$. Hence,  by Fermat's theorem, we get $H^\prime(\overline{u}(t_0))=0$, that is to say,
being $t_0 \in E$ arbitrary,
	\begin{equation}\label{Eulp1}
		\int_{t_0}^T \E\left[\pd{F}{\xi}(t,\xi_{\overline{u}}(t))e^{A(t)-A(t_0)}\frac{G(t)}{G(t_0)}\right]dt=0, \ \forall t_0 \in E.
	\end{equation}
Now we want to extend $E$ to the whole interval $[0,T]$. Let us show that
\begin{equation}\label{intfun}
	t_0 \in [0,T] \to \int_{t_0}^T \E\left[\pd{F}{\xi}(t,\xi_{\overline{u}}(t))e^{A(t)-A(t_0)}\frac{G(t)}{G(t_0)}\right]dt
\end{equation}
is continuous. To do this, consider $t_1 \in [0,T]$ and $t_2=t_1+\delta$ for some $\delta$ small enough to have $t_2 \in [0,T]$. To fix the ideas, let us suppose  $\delta>0$, since the arguments for $\delta<0$ are the same. We have
\begin{align*}
	&\left|\int_{t_1}^T \E\left[\pd{F}{\xi}(t,\xi_{\overline{u}}(t))e^{A(t)-A(t_1)}\frac{G(t)}{G(t_1)}\right]dt-\int_{t_2}^T \E\left[\pd{F}{\xi}(t,\xi_{\overline{u}}(t))e^{A(t)-A(t_2)}\frac{G(t)}{G(t_2)}\right]dt\right|\\
	&\qquad \le \int_{t_1}^{t_2} \E\left[\left|\pd{F}{\xi}(t,\xi_{\overline{u}}(t))e^{A(t)-A(t_1)}\frac{G(t)}{G(t_1)}\right|\right]dt\\
	&\qquad \qquad+\int_{t_2}^T \E\left[\left|\pd{F}{\xi}(t,\xi_{\overline{u}}(t))e^{A(t)-A(t_1)}\frac{G(t)}{G(t_1)}-\pd{F}{\xi}(t,\xi_{\overline{u}}(t))e^{A(t)-A(t_2)}\frac{G(t)}{G(t_2)}\right|\right]dt\\
	&\qquad :=I_4(\delta)+I_5(\delta).
\end{align*}
Let us consider $I_4(\delta)$. The exact same argument we considered for $I_2(\varepsilon)$ leads to
\begin{align*}
	\E\left[\left|\pd{F}{\xi}(t,\xi_{\overline{u}}(t))e^{A(t)-A(t_1)}\frac{G(t)}{G(t_1)}\right|\right]\le CL(t)
\end{align*}
and then
\begin{equation*}
I_4(\delta)\le C\int_{t_1}^{t_2}L(t)dt.
\end{equation*}
By absolute continuity of the Lebesgue integral, we have $\lim_{\delta \to 0}I_4(\delta)=0$.\\
Concerning $I_5(\delta)$, we have to use dominated convergence theorem. To do this, let us just observe that, as in Equation \eqref{eq:I2est1},
\begin{equation*}
	\left|\pd{F}{\xi}(t,\xi_{\overline{u}}(t))e^{A(t)-A(t_1)}\frac{G(t)}{G(t_1)}\right|\le C\left|\pd{F}{\xi}(t,\xi_{\overline{u}}(t))\right|\left(\sup_{t \in [0,T]}G(t)\right)\left(\sup_{t \in [0,T]}G'(t)\right)
\end{equation*}
where the right-hand side is independent of $t_1$ and integrable. The same can be done with $t_2$ and then, by triangular inequality, we have that the integrand in $I_5(\delta)$ is dominated. By dominated convergence theorem, since $G(t)$ is almost surely continuous, we have $\lim_{\delta \to 0}I_5(\delta)=0$.\\
Hence, the function in Equation \eqref{intfun} is continuous and, since $|[0,T]\setminus E|=0$ and thus $E$ is dense in $[0,T]$, we can extend Equation \eqref{Eulp1} to the whole interval $[0,T]$, concluding the proof. 
\end{proof}
\begin{rmk}\label{Gateaux}
	Let us observe that, with the same arguments, we can actually show that $J$ is Gateaux differentiable (see \cite{berger1977nonlinearity}) in any $u \in L^1(0,T)$ with Gateaux derivative given by
	\begin{equation*}
		\partial_u J[v]=-\int_0^T v(s)\int_s^T\E\left[\pd{F}{\xi}(t,\xi_{u}(t))e^{A(t)-A(s)}\frac{G(t)}{G(s)}\right]dtds, \ v \in L^1(0,T)
	\end{equation*}
and then Equation \eqref{eq:EL} can be restated as
\begin{equation*}
	\partial_{\bar{u}} J[v]=0, \ \forall v \in L^1(0,T).
\end{equation*}
From this point of view, Equation \eqref{eq:EL} is a consequence of Fermat's theorem applied directly on $J$. For this reason, we can refer to Equation \eqref{eq:EL} as the Euler-Lagrange equation for the functional $J$. In the same fashion, we can recognize $H$ as the Hamiltonian function of the cost functional $J$. \\
Moreover, let us stress out that, by the absolute continuity of Lebesgue's integral, for any $\delta>0$ there exists $\varepsilon_0>0$ such that for any $\varepsilon<\varepsilon_0$ it holds
\begin{equation*}
	\Norm{u_\varepsilon-\bar{u}}{L^1}=\int_{I_\varepsilon}|u-\bar{u}(\tau)|d\tau<\delta,
\end{equation*}
being $|I_\varepsilon|<\varepsilon_0$. Thus, defining the ball
\begin{equation*}
	B_\delta(\bar{u})=\{u \in L^1(0,T): \ \Norm{u-\bar{u}}{L^1}<\delta\}
\end{equation*}
we know that, for any fixed $\delta>0$, there exists $\varepsilon_0>0$ such that $u_\varepsilon \in B_\delta(\bar{u})$ for any $\varepsilon<\varepsilon_0$. Last observation leads to the fact that Theorem $3.1$ holds also for local minimizers of $J$, i.e. for functions $\bar{u}$ for which there exists $\delta>0$ such that
\begin{equation*}
	J[u]\ge J[\bar{u}], \ \forall u \in B_\delta(\bar{u}).
\end{equation*}
\end{rmk}
\section{Sufficient optimality conditions}
In the previous section we obtained a necessary optimality condition given in terms of Equation \eqref{eq:EL}. Now we want to investigate whether such condition is also sufficient, i.e. any solution of Equation \eqref{eq:EL} is actually a minimizer of the functional $J$ on $L^1(0,T)$. To do this, we need some additional hypotheses:
\begin{itemize}
	\item[(\emph{H5})] For any fixed $t \in [0,T]$, the function $x \mapsto F(t,x)$ is convex.
	\item[(\emph{H5+})] For any fixed $t \in [0,T]$, the function $x \mapsto F(t,x)$ is strictly convex.
\end{itemize}
\begin{thm}\label{thm:suffcond}
	Suppose Hypotheses (H1) to (H5) are satisfied. Let $\bar{u} \in L^1(0,T)$ be a solution of Equation \eqref{eq:EL}. Then $\bar{u}$ is a global minimizer of $J$.
\end{thm}
\begin{proof}
	Let us consider $\bar{u} \in L^1(0,T)$ solution of Equation \eqref{eq:EL} and let $u \in L^1(0,T)$ be any other function. Then we have
	\begin{equation*}
		J[u]-J[\bar{u}]=\int_0^T\E[F(t,\xi_u(t))-F(t,\xi_{\bar{u}}(t))]dt
	\end{equation*}
	where we already used Fubini's theorem, by means of hypotheses (\emph{H2}). By hypothesis (\emph{H3}) and (\emph{H5}) we have
	\begin{align*}
		F(t,\xi_u(t))-F(t,\xi_{\bar{u}}(t))&\ge \pd{F}{\xi}(t,\xi_{\bar{u}}(t))(\xi_u(t)-\xi_{\bar{u}}(t))\\
	&=\pd{F}{\xi}(t,\xi_{\bar{u}}(t))\cS_0(\bar{u}-u)\\
	&=\pd{F}{\xi}(t,\xi_{\bar{u}}(t))G(t)e^{A(t)}\int_0^t\frac{e^{-A(s)}}{G(s)}(\bar{u}(s)-u(s))ds\\
	&=\int_0^t\pd{F}{\xi}(t,\xi_{\bar{u}}(t))e^{A(t)-A(s)}\frac{G(t)}{G(s)}(\bar{u}(s)-u(s))ds
	\end{align*}
and then
\begin{equation}\label{eq:beforeex}
	J[u]-J[\bar{u}]\ge \int_0^T\E\left[\int_0^t\pd{F}{\xi}(t,\xi_{\bar{u}}(t))e^{A(t)-A(s)}\frac{G(t)}{G(s)}(\bar{u}(s)-u(s))ds\right]dt.
\end{equation}
Now we want to exchange the order of the integrals. To do this, observe that
\begin{align*}
	\begin{split}
	&\left|\pd{F}{\xi}(t,\xi_{\bar{u}}(t))e^{A(t)-A(s)}\frac{G(t)}{G(s)}(\bar{u}(s)-u(s))\right|\\
	&\qquad \le \left|\pd{F}{\xi}(t,\xi_{\bar{u}}(t))\right|\left(\sup_{\tau_1,\tau_2 \in [0,T]}e^{A(\tau_1)-A(\tau_2)}\right)e^T\left(\sup_{\tau \in (0,T)}G(\tau)\right)\left(\sup_{\tau \in (0,T)}G'(\tau)\right)|\bar{u}(s)-u(s)|\\
	&\qquad =C\left|\pd{F}{\xi}(t,\xi_{\bar{u}}(t))\right|\left(\sup_{\tau \in (0,T)}G(\tau)\right)\left(\sup_{\tau \in (0,T)}G'(\tau)\right)|\bar{u}(s)-u(s)|.
	\end{split}
\end{align*}
Let us consider the process on the right-hand side of the previous inequality. Integrating with respect to $s$ and applying the expectation operator we have
\begin{align}\label{eq:Fubinipass2}
	\begin{split}
	&\E\left[\int_0^T\left|\pd{F}{\xi}(t,\xi_{\bar{u}}(t))\right|\left(\sup_{\tau \in (0,T)}G(\tau)\right)\left(\sup_{\tau \in (0,T)}G'(\tau)\right)|\bar{u}(s)-u(s)|ds\right]\\
	&\qquad =\Norm{\bar{u}-u}{L^1}\E\left[\left|\pd{F}{\xi}(t,\xi_{\bar{u}}(t))\right|\left(\sup_{\tau \in (0,T)}G(\tau)\right)\left(\sup_{\tau \in (0,T)}G'(\tau)\right)\right].
	\end{split}
\end{align}
Now let us fix $\widetilde{p} \in \left(1,\min\left\{2,\frac{p}{\alpha}\right\}\right)$, where $\alpha$ is defined in hypothesis (\emph{H4}), and $\widetilde{q}$ such that $\frac{1}{\widetilde{p}}+\frac{1}{\widetilde{q}}=1$. By H\"older's inequality we have
\begin{align*}
	&\E\left[\left|\pd{F}{\xi}(t,\xi_{\bar{u}}(t))\right|\left(\sup_{\tau \in (0,T)}G(\tau)\right)\left(\sup_{\tau \in (0,T)}G'(\tau)\right)\right]\\
	&\quad \qquad \le\E\left[\left|\pd{F}{\xi}(t,\xi_{\bar{u}}(t))\right|^{\widetilde{p}}\right]^{\frac{1}{\widetilde{p}}}\E\left[\left(\sup_{\tau \in (0,T)}G(\tau)\right)^{\widetilde{q}}\left(\sup_{\tau \in (0,T)}G'(\tau)\right)^{\widetilde{q}}\right]^{\frac{1}{\widetilde{q}}}\\
	&\quad \qquad \le CL(t)\E\left[(1+|\xi_{\bar{u}}(t)|^\alpha)^{\widetilde{p}}\right]^{\frac{1}{\widetilde{p}}}\\
	&\quad \qquad \le CL(t)2^{1-\frac{1}{\widetilde{p}}}(1+\E[|\xi_{\bar{u}}(t)|^{\alpha\widetilde{p}}]^{\frac{1}{\widetilde{p}}})\\
	&\quad \qquad \le CL(t)2^{1-\frac{1}{\widetilde{p}}}(1+\E[|\xi_{\bar{u}}(t)|^{p}]^{\frac{\alpha}{p}}),
\end{align*}
where we used hypothesis (\emph{H4}), Lemma \ref{lem:supest} and H\"older's inequality a second time with exponent $\frac{p}{\alpha \widetilde{p}}>1$. By Lemma \ref{lem:momest2} we conclude that
\begin{align*}
	\E\left[\left|\pd{F}{\xi}(t,\xi_{\bar{u}}(t))\right|\left(\sup_{\tau \in (0,T)}G(\tau)\right)\left(\sup_{\tau \in (0,T)}G'(\tau)\right)\right]
	 \le CL(t).
\end{align*}
Going back to Equation \eqref{eq:Fubinipass2} we have
\begin{equation*}
	\E\left[\int_0^T\left|\pd{F}{\xi}(t,\xi_{\bar{u}}(t))e^{A(t)-A(s)}\frac{G(t)}{G(s)}(\bar{u}(s)-u(s))\right|ds\right] \le C \Norm{\bar{u}-u}{L^1} L(t)
\end{equation*}
and then, integrating with respect to $t$,
\begin{equation*}
	\int_0^T\E\left[\int_0^T\left|\pd{F}{\xi}(t,\xi_{\bar{u}}(t))e^{A(t)-A(s)}\frac{G(t)}{G(s)}(\bar{u}(s)-u(s))\right|ds\right]dt \le C \Norm{\bar{u}-u}{L^1}\Norm{L}{L^1}.
\end{equation*}
Hence, we can use Fubini's theorem in Equation \eqref{eq:beforeex} to achieve
\begin{equation*}
	J[u]-J[\bar{u}]\ge \int_0^T (\bar{u}(s)-u(s))\int_s^T\E\left[\pd{F}{\xi}(t,\xi_{\bar{u}}(t))e^{A(t)-A(s)}\frac{G(t)}{G(s)}\right]dtds=0,
\end{equation*}
$\bar{u}$ being a solution of \eqref{eq:EL}. The fact that $u \in L^1$ is arbitrary concludes the proof.
\end{proof}
The previous result is strictly linked with the convexity hypothesis (\emph{H5}). Indeed, such hypothesis actually implies the convexity of the operator $J$.
\begin{prop}\label{prop:convex}
	Suppose Hypotheses (H1) to (H5) are satisfied. Then $J$ is convex. Moreover, if (H5+) is satisfied, $J$ is strictly convex.
\end{prop}
\begin{proof}
	Let us consider $u_1,u_2 \in L^1(0,T)$ and $\theta \in [0,1]$. Then we have
	\begin{align*}
		J[\theta u_1+(1-\theta)u_2]&=\E\left[\int_0^T F(t,\cS_0(z-(\theta u_1+(1-\theta)u_2))(t))dt\right]
		\\&=\E\left[\int_0^T F(t,\theta\cS_0(z-u_1)(t)+(1-\theta)\cS_0(z-u_2)(t))dt\right]\\
		&\le \theta\E\left[\int_0^T F(t,\cS_0(z-u_1)(t))dt\right]+(1-\theta)\E\left[\int_0^T F(t,\cS_0(z-u_2)(t))dt\right]\\
		&=\theta J[u_1]+(1-\theta)J[u_2],
	\end{align*}
where we used hypothesis (\emph{H5}) and the third statement of Corollary \ref{linear}. This proves that $J$ is convex.
	Now let us suppose $u_1 \not = u_2$ (that is to say there exists a set $E \subseteq (0,T)$ with $|E|>0$ and $u_1 \not = u_2$ on $E$), $\theta \in (0,1)$ and (\emph{H5+}) holds. By definition of the solution map, there exists $\overline{\Omega}$ such that
	\begin{itemize}
		\item[I] $\bP(\overline{\Omega})>0$;
		\item[II] for any $\omega \in \overline{\Omega}$, $\cS_0(z-u_i)(\cdot, \omega)$ is continuous in $[0,T]$ for $i=1,2$;
		\item[III] for any $\omega \in \overline{\Omega}$, there exists $t(\omega)$ such that $\cS_0(z-u_1)(t(\omega),\omega)\not = \cS_0(z-u_2)(t(\omega),\omega)$.
	\end{itemize} 
	In particular, combining II and III we have that for any $\omega \in \overline{\Omega}$ there exists an interval $I(\omega)$ such that $\cS_0(z-u_1)(t,\omega)\not = \cS_0(z-u_2)(t,\omega)$ for any $t \in I(\omega)$. By using the third statement of Corollary \ref{linear} we get
	\begin{align*}
		J[\theta u_1+(1-\theta)u_2]&=\E\left[\int_0^T F(t,\cS_0(z-(\theta u_1+(1-\theta)u_2))(t))dt\right]
		\\&=\E\left[\int_0^T F(t,\theta\cS_0(z-u_1)(t)+(1-\theta)\cS_0(z-u_2)(t))dt\right]\\
		&=\E\left[\int_{I(\cdot)} F(t,\theta\cS_0(z-u_1)(t)+(1-\theta)\cS_0(z-u_2)(t))dt; \overline{\Omega}\right]\\
		&+\E\left[\int_{[0,T]\setminus I(\cdot)} F(t,\theta\cS_0(z-u_1)(t)+(1-\theta)\cS_0(z-u_2)(t))dt;\overline{\Omega}\right]\\
		&+\E\left[\int_0^T F(t,\theta\cS_0(z-u_1)(t)+(1-\theta)\cS_0(z-u_2)(t))dt;\Omega \setminus \overline{\Omega}\right],
	\end{align*}
where for any random variable $Z$ and any event $B \in \mathcal{F}$ we use the notation $\E[Z;B]=\E[Z\mathbf{1}_B]$. Now let us observe that by Hypothesis (\emph{H5+}) we have
\begin{align*}
	&F(t,\theta\cS_0(z-u_1)(t,\omega)+(1-\theta)\cS_0(z-u_2)(t,\omega))\\
	&\qquad \qquad\le \theta F(t,\cS_0(z-u_1)(t,\omega))+(1-\theta)F(t,\cS_0(z-u_2)(t,\omega)), \ \forall \omega \in \Omega, \ \forall t \in [0,T],
\end{align*}
and the inequality is strict for any $\omega \in \overline{\Omega}$ and $t \in I(\omega)$. Thus we get
\begin{align*}
	J[\theta u_1&+(1-\theta)u_2]<\theta\E\left[\int_{I(\cdot)} F(t,\cS_0(z-u_1)(t))dt; \overline{\Omega}\right]+(1-\theta)\E\left[\int_{I(\cdot)}F(t,\cS_0(z-u_2)(t))dt;\overline{\Omega}\right]\\
	&\qquad+\theta\E\left[\int_{[0,T]\setminus I(\cdot)} F(t,\cS_0(z-u_1)(t))dt; \overline{\Omega}\right]+(1-\theta)\E\left[\int_{[0,T] \setminus I(\cdot)}F(t,\cS_0(z-u_2)(t))dt;\overline{\Omega}\right]\\
	&\qquad+\theta\E\left[\int_{0}^T F(t,\cS_0(z-u_1)(t))dt;\Omega \setminus \overline{\Omega}\right]+(1-\theta)\E\left[\int_{0}^TF(t,\cS_0(z-u_2)(t))dt;\Omega \setminus \overline{\Omega}\right]\\
	&=\theta J[u_1]+(1-\theta)J[u_2],
\end{align*}
concluding the proof.
\end{proof}
\begin{rmk}
	Let us observe that Theorem \ref{thm:suffcond} can be seen as a direct consequence of Proposition \ref{prop:convex} and Remark \ref{Gateaux}, by using the inequality
	\begin{equation*}
		J[u]-J[\bar{u}]\ge \partial_{\bar{u}} J[u-\bar{u}],
	\end{equation*}
implied by the convexity of $J[u]$.
\end{rmk}
Another direct consequence of Proposition \ref{prop:convex} is given by the following Corollary.
\begin{cor}
	Suppose Hypotheses ({H1}) to (H5+) are satisfied. Then Equation \eqref{eq:EL} admits at most one solution.
\end{cor}
\begin{proof}
	Let us suppose $\bar{u}_1$ and $\bar{u}_2$ are two solutions of Equation \eqref{eq:EL}. Then, by Theorem \ref{thm:suffcond} we know that both $\bar{u}_1$ and $\bar{u}_2$ are global minimizers of $J$. However, Proposition \ref{prop:convex} tells us that $J$ is strictly convex, hence it admits a unique global minimizer and then $\bar{u}_1=\bar{u}_2$.
\end{proof}
\section{Minimizing families}
Up to now we are not able to show that $J$ is coercive, which should be the main ingredient, together with lower semicontinuity, to prove the existence of a minimizer. This is due to the fact that, since $\xi_u(t)=\cS_0(z-u)(t)$ depends on a sort of primitive function of $u$, classical lower bounds such as $F(t,\xi)\ge L(1+|\xi|^p)$ are not enough to guarantee coercivity. For this reason, we focus instead on exploiting some minimizing families for $J$, i.e. a family of functions $\{u_{\delta}\}_{\delta>0}$ with the property that, for any $\varepsilon>0$, there exists $\delta_0>0$ such that if $\delta \in (0,\delta_0)$ it holds
\begin{equation*}
	m\le J[u_\delta]\le m+\varepsilon
\end{equation*}
where $m=\inf_{u \in L^1(0,T)}J[u]$.\\
First of all, we observe that $J$ is a continuous functional on $L^1(0,T)$.
\begin{prop}\label{prop:continuity}
	Let hypotheses (H1) to (H4) hold. Then $J:L^1(0,T) \to \R$ is continuous, i.e. for any fixed $u_1 \in L^1(0,T)$, for any $\varepsilon>0$ there exists $r>0$ (possibly depending on $u_1$) such that
	\begin{equation*}
		\forall u_2 \in B_r(u_1), \ |J[u_2]-J[u_1]|<\varepsilon,
	\end{equation*}
where $B_r(u_1)=\{u_2 \in L^1(0,T): \ \Norm{u_2-u_1}{L^1(0,T)}<r\}$.
\end{prop}
\begin{proof}
	Fix $\varepsilon>0$ and let $u_1,u_2 \in L^1(0,T)$ with $\Norm{u_1-u_2}{L^1(0,T)}<r$, where $r$ will be defined in what follows. Then
	\begin{align}\label{eq:contpass1}
		\begin{split}
		|J[u_2]-J[u_1]|&\le \E\left[\int_0^T|F(t,\xi_{u_2}(t))-F(t,\xi_{u_1}(t))|dt\right]\\
		&\le \E\left[\int_0^T\int_0^1 \left|\pd{F}{\xi}(t,\theta\xi_{u_2}(t)+(1-\theta)\xi_{u_1}(t))\right||\xi_{u_2}(t)-\xi_{u_1}(t)|dt\right],
		\end{split}
	\end{align}
	where we used hypothesis (\emph{H3}). Let us consider $p$ as in (\emph{H1}) and $\alpha$ as in (\emph{H4}) and, let $\widetilde{p} \in \left(1,  \min\left\{2,\frac{p}{\alpha}\right\}\right)$ and $\widetilde{q}$ such that $\frac{1}{\widetilde{p}}+\frac{1}{\widetilde{q}}=1$. By H\"older's inequality we have
	\begin{multline}\label{eq:contpass2}
		\E\left[\left|\pd{F}{\xi}(t,\theta\xi_{u_2}(t)+(1-\theta)\xi_{u_1}(t))\right||\xi_{u_2}(t)-\xi_{u_1}(t)|\right] \\\le \E\left[\left|\pd{F}{\xi}(t,\theta\xi_{u_2}(t)+(1-\theta)\xi_{u_1}(t))\right|^{\widetilde{p}}\right]^{\frac{1}{\widetilde{p}}}\E[|\xi_{u_2}(t)-\xi_{u_1}(t)|^{\widetilde{q}}]^\frac{1}{\widetilde{q}}.
	\end{multline}
	Concerning the first term, we have, by hypothesis (\emph{H4})
	\begin{align*}
		\E\left[\left|\pd{F}{\xi}(t,\theta\xi_{u_2}(t)+(1-\theta)\xi_{u_1}(t))\right|^{\widetilde{p}}\right]^{\frac{1}{\widetilde{p}}}&\le L(t)\E[(1+|\theta\xi_{u_2}(t)+(1-\theta)\xi_{u_1}(t)|^{\alpha})^{\widetilde{p}}]^{\frac{1}{\widetilde{p}}}\\
		&\le 2^{1-\frac{1}{\widetilde{p}}}L(t)(1+\E[|\theta\xi_{u_2}(t)+(1-\theta)\xi_{u_1}(t)|^{\alpha \widetilde{p}}]^{\frac{1}{\widetilde{p}}})\\
		&\le 2^{1-\frac{1}{\widetilde{p}}}L(t)(1+\E[|\theta\xi_{u_2}(t)+(1-\theta)\xi_{u_1}(t)|^{p}]^{\frac{\alpha}{p}}),
	\end{align*}
where we used again H\"older's inequality with exponent $\frac{p}{\alpha\widetilde{p}}>1$. Next, since $\xi_{u_i}=\cS_0(z-u_i)$, by Corollary \ref{linear} and Lemma \ref{lem:momest} we get
\begin{align*}
	\E&\left[\left|\pd{F}{\xi}(t,\theta\xi_{u_2}(t)+(1-\theta)\xi_{u_1}(t))\right|^{\widetilde{p}}\right]^{\frac{1}{\widetilde{p}}}\le 2^{1-\frac{1}{\widetilde{p}}}L(t)(1+\E[|S_0(z-\theta u_2-(1-\theta)u_1)|^{p}]^{\frac{\alpha}{p}})\\
	&\le C L(t)\left(1+\left(\int_0^T\E[|z(t)-(\theta u_2+(1-\theta)u_1)|^p]^{\frac{1}{p}}\right)^\alpha\right)\\
	&\le C L(t)\left(1+2^{\alpha-\frac{\alpha}{p}}\left(\int_0^T\E[|z(t)|^p]^{\frac{1}{p}}dt+\int_0^T|\theta u_2(t)+(1-\theta)u_1(t)|dt\right)^\alpha\right)\\
	&\le C L(t)\left(1+2^{2\alpha-\frac{\alpha}{p}-1}\left(\left(\int_0^T\E[|z(t)|^p]^{\frac{1}{p}}dt\right)^{\alpha}+\left(\int_0^T|\theta u_2(t)+(1-\theta)u_1(t)|dt\right)^\alpha\right)\right)\\
	&\le CL(t)\left(1+\left(\int_0^T|\theta u_2(t)+(1-\theta)u_1(t)|dt\right)^\alpha\right)\\
	&\le CL(t)\left(1+\theta\Norm{u_2}{L^1(0,T)}^\alpha+(1-\theta)\Norm{u_1}{L^1(0,T)}^\alpha\right)
\end{align*}
where $C$ is independent of $u_1$ and $u_2$, and we used the fact that $\left(\int_0^T\E[|z(t)|^p]^{\frac{1}{p}}dt\right)^\alpha<+\infty$ and that \linebreak $t\ge 0 \mapsto t^\alpha$ is a convex function. 
We can assume, without loss of generality, that $r\leq 1$. Since $\Norm{u_2-u_1}{L^1(0,T)}<r\leq 1$, we achieve
\begin{align}\label{eq:contpass3}
	\begin{split}
	\E&\left[\left|\pd{F}{\xi}(t,\theta\xi_{u_2}(t)+(1-\theta)\xi_{u_1}(t))\right|^{\widetilde{p}}\right]^{\frac{1}{\widetilde{p}}}\\&\le CL(t)\left(1+\theta2^{\alpha-1}\Norm{u_2-u_1}{L^1(0,T)}^\alpha+(1-\theta+2^{\alpha-1})\Norm{u_1}{L^1(0,T)}^\alpha\right)\\
	&< CL(t)\left(1+\Norm{u_1}{L^1(0,T)}^\alpha\right).
\end{split}
\end{align}
On the other hand, still recalling that $\xi_{u_i}=\cS_0(z-u_i)$ and by Corollary \ref{linear} and Lemma \ref{lem:momest2}, we get
\begin{equation}\label{eq:contpass4}
	\E[|\xi_{u_2}(t)-\xi_{u_1}(t)|^{\widetilde{q}}]^{\frac{1}{\widetilde{q}}}=\E[|\cS_0(u_2-u_1)(t)|^{\widetilde{q}}]^{\frac{1}{\widetilde{q}}} \le C\Norm{u_2-u_1}{L^1(0,T)}<Cr. 
\end{equation}
Combining Equations \eqref{eq:contpass3} and \eqref{eq:contpass4} in Equation \eqref{eq:contpass2} we get
\begin{equation*}
	\E\left[\left|\pd{F}{\xi}(t,\theta\xi_{u_2}(t)+(1-\theta)\xi_{u_1}(t))\right||\xi_{u_2}(t)-\xi_{u_1}(t)|\right]<CL(t)\left(1+\Norm{u_1}{L^1(0,T)}^\alpha\right)r.
\end{equation*}
Using the previous inequality in \eqref{eq:contpass1} together with Fubini's theorem, since the integrand is non-negative, we know there exists a constant $\overline{C}>0$, independent of $u_1$ and $u_2$, such that
\begin{equation*}
	|J[u_2]-J[u_1]|< \overline{C}\left(1+\Norm{u_1}{L^1(0,T)}^\alpha\right)r,
\end{equation*}
Setting
\begin{equation*}
	r=\min\left\{1,\frac{\varepsilon}{\overline{C}\left(1+\Norm{u_1}{L^1(0,T)}^\alpha\right)}\right\}
\end{equation*}
we conclude the proof.
\end{proof}
\begin{rmk}
	If, moreover, Hypothesis (\emph{H5}) holds, then $J$ is also weakly lower semicontinuous in $L^1$. This is a direct consequence of Mazur's theorem \cite[Theorem $3.9$]{dacorogna2007direct}.
\end{rmk}
Let us introduce the set of simple functions:
\begin{equation*}
	\fS=\left\{u \in L^1(0,T): \ \exists N \in \N, \{B_i\}_{i \le N}\subseteq \cB([0,T]), \ \{b_i\}_{i \le N}\subseteq \R, \ u=\sum_{i=1}^{N}b_i\mathbf{1}_{B_i}\right\},
\end{equation*}
where $\cB([0,T])$ is the Borel $\sigma$-algebra on the interval $[0,T]$. As a direct consequence of Proposition \ref{prop:continuity} we have the following Corollary.
\begin{cor}\label{cor:simple}
	Let hypotheses (H1) to (H4) hold. Then
	\begin{equation*}
		\inf_{u \in L^1(0,T)}J[u]=\inf_{u \in \fS}J[u].
	\end{equation*}
\end{cor}
\begin{proof}
	It follows from the fact that $J:L^1(0,T) \to \R$ is continuous (by Propostion \ref{prop:continuity}) and the fact that $\fS$ is dense in $L^1(0,T)$ (see, for instance, \cite[Theorem $3.13$]{rudin1987real}).
\end{proof}
Now we want to \textit{penalize} our functional $J$ to obtain a coercive functional. To do this, let us first give the following definition.
\begin{defn}
	A function $\Psi:[0,+\infty) \to [0,+\infty)$ is called a Young function (see \cite{pick2012function}) if there exists a function $\psi:[0,+\infty) \to [0,+\infty)$ such that
	\begin{equation*}
		\Psi(t)=\int_0^t\psi(s)ds, \ t \ge 0
	\end{equation*}
and $\psi$ satisfies the following properties:
\begin{itemize}
	\item $\psi(0)=0$;
	\item $\psi(s)>0$ for any $s>0$;
	\item $\psi$ is right-continuous;
	\item $\psi$ is non-decreasing;
	\item $\lim_{s \to +\infty}\psi(s)=+\infty$.
\end{itemize}
\end{defn}
Young functions satisfy different important properties. Here we recall some of them (see \cite[Lemma $4.2.2$]{pick2012function}).
\begin{lem}
	Any Young function $\Psi$ is continuous, non-negative, strictly increasing and convex. Moreover it holds $\Psi(0)=0$, $\lim_{t \to 0^+}t^{-1}\Psi(t)=0$ and $\lim_{t \to +\infty}t^{-1}\Psi(t)=+\infty$.
\end{lem}
\begin{rmk}\label{remark_u0}
Let us observe that if $\Psi$ is a Young function then the function $u \in \R \mapsto \Psi(|u|)$ is differentiable. Indeed this is clearly true for $u\neq 0$ with derivative
\begin{equation*}
\der{\Psi(|u|)}{u}=\frac{u}{|u|}\psi(|u|), \quad u\neq 0.
\end{equation*} 
Observing that $-1\leq \frac{u}{|u|}\leq 1$ and that, $\psi$ being right-continuous with $\psi(0)=0$, it holds $\lim_{u\rightarrow 0}\der{\Psi(|u|)}{u}=0$
and then $\Psi(|u|)$ is differentiable at $0$ with derivative $0$.
We will use the notation $\frac{u}{|u|}\psi(|u|)$ for any $u \in \R$, implying the $0$ value as $u=0$. 
\end{rmk}

Fix any Young function $\Psi$, $\delta>0$ and define the following functional
\begin{equation*}
	J_{\delta,\Psi}:L^1(0,T)\to J[u]+\delta \cF_{\Psi}[u]
\end{equation*}
where
\begin{equation*}
	\cF_\Psi[u]=\int_0^T \Psi(|u(t)|)dt.
\end{equation*}
\begin{rmk}
	There exist some $u \in L^1(0,T)$ such that $J_{\delta,\Psi}[u]=+\infty$ (see, for instance, \cite[Remark $4.2.4$]{pick2012function}). In particular we can define the Orlicz class
	\begin{equation*}
		\cL^\Psi(0,T)=\left\{u \in L^1(0,T): \ \int_0^T \Psi(|u(t)|)dt<+\infty\right\}
	\end{equation*}
	and observe that $J_{\delta,\Psi}[u]<+\infty$ if and only if $u \in \cL^\Psi(0,T)$. Let us stress out that $\cL^\Psi(0,T)$ is in general not a vector space. Actually, $\cL^\Psi(0,T)$ is a vector space if and only if $\Psi \in \Delta_2$, i.e. there exists a constant $k>0$ such that $\Psi(2t)\le k \Psi(t)$ for any $t\ge 0$ (see \cite[Theorem $4.5.3$]{pick2012function}). In particular, if $\Psi \in \Delta_2$, we have that $\cL^\Psi(0,T)=L^\Psi(0,T)$, where $L^\Psi(0,T)$ is defined as
	\begin{equation*}
		L^\Psi(0,T)=\left\{u \in L^1(0,T): \ \exists \lambda> 0: \ \int_0^T \Psi\left(\frac{|u(t)|}{\lambda}\right)dt<+\infty\right\}
	\end{equation*}
	and is a Banach space when equipped with the norm
	\begin{equation*}
		\Norm{u}{L^\Psi(0,T)}:=\inf\left\{\lambda>0: \int_0^T \Psi\left(\frac{|u(t)|}{\lambda}\right)dt\le 1\right\}.
	\end{equation*}
	See \cite[Chapter $4$]{pick2012function} for further details.
\end{rmk}
First of all, we want to show that for any $\delta>0$ and any Young function $\Psi$, the functional $J_{\delta,\Psi}$ admits a minimum. To do this, we need the following preliminary result.
\begin{prop}\label{prop:semicontF}
	For any Young function $\Psi \in \Delta_2$, the functional $\cF_\Psi:L^1(0,T) \mapsto \R \cup \{+\infty\}$ is weakly lower semicontinuous.
\end{prop}
The previous result relies on classical arguments in Calculus of Variation (see \cite{angrisani2019appunti,de2008semicontinuity,tonelli1921fondamenti}). We provide its proof in Appendix~\ref{App:B} for completeness.\\
Now we are ready to show that $J_{\delta,\Psi}$ admits a minimum.
\begin{thm}\label{thm:minim}
	Let $\delta>0$ and $\Psi \in \Delta_2$ be a Young function. Suppose hypotheses (H1) to (H5) hold. Then there exists a function $\bar{u}_{\delta,\Psi} \in L^1(0,T)$ such that
	\begin{equation*}
		\min_{u \in L^1(0,T)}J_{\delta,\Psi}[u]=J_{\delta,\Psi}[\bar{u}_{\delta,\Psi}].
	\end{equation*}
Moreover, if hypothesis (H5+) holds or if $\Psi$ is strictly convex, then $\bar{u}_{\delta,\Psi} \in L^1(0,T)$ is unique.
\end{thm}
\begin{proof}
	Let us first observe that $J_{\delta,\Psi}$ is weakly lower semicontinuous. To do this, observe that $J_{\delta,\Psi}$ is the sum of two convex functionals $J$ and $\delta \cF_\Psi$, thus it is convex. Moreover, $J$ is continuous by Proposition \ref{prop:continuity} and then, in particular, lower semicontinuous, while $\delta \cF_\Psi$ is lower semicontinuous by Proposition \ref{prop:semicontF}. Thus $J_{\delta,\Psi}$ is convex and lower semicontinuous and then weakly lower semicontinuous by a direct application of Mazur's theorem \cite[Theorem $3.9$]{dacorogna2007direct}.\\
	Now let us consider a minimizing sequence $\{u_n\}_{n \in \N}\subset L^1(0,T)$ of $J_{\delta,\Psi}$, i.e. $\{u_n\}_{n \in \N}$ is  such that \linebreak $J_{\delta,\Psi}[u_n]\downarrow \inf_{u \in L^1(0,T)}J_{\delta,\Psi}[u]$. Let us consider any function $f \in L^\Psi(0,T)$. Then $J_{\delta,\Psi}[f]<+\infty$ by definition of $L^\Psi$. In particular this implies that $\inf_{u \in L^1(0,T)}J_{\delta,\Psi}[u]<+\infty$ and we can suppose $J_{\delta,\Psi}[u_1]<+\infty$. Let us observe that
	\begin{equation*}
		\delta \cF_\Psi[u_n]\le J_{\delta,\Psi}[u_n]\le J_{\delta,\Psi}[u_1], \ \forall n \in \N
	\end{equation*} 
and then there exists a constant $C(\delta)=\frac{J_{\delta,\Psi}[u_1]}{\delta}$ such that
\begin{equation*}
	\cF_\Psi[u_n]\le C(\delta), \ \forall n \in \N.
\end{equation*}
By the de la Vall\'ee-Poussin theorem (see \cite[Theorem T22]{meyer1966probability}), we know that the sequence $u_n$ is uniformly integrable. By the Dunford-Pettis theorem (see \cite[Theorem $4.30$]{brezis2010functional}) we have that the sequence $\{u_n\}$ is weakly relatively compact in $L^1(0,T)$ and then there exists $\overline u_{\delta,\Psi} \in L^1(0,T)$ such that $u_n \rightharpoonup \overline{u}_{\delta, \Psi}$. By weak semicontinuity of $J_{\delta,\Psi}$ we have
\begin{equation*}
	\inf_{u \in L^1(0,T)}J_{\delta,\Psi}[u]=\lim_{n \to +\infty}J_{\delta,\Psi}[u_n]\ge J_{\delta,\Psi}[\overline{u}_{\delta,\Psi}]\ge \inf_{u \in L^1(0,T)}J_{\delta,\Psi}[u],
\end{equation*}
concluding the proof of the first statement.
Concerning the second statement, it follows from the fact that if hypothesis (H5+) holds or if $\Psi$ is strictly convex, then $J_{\delta,\Psi}$ is a strictly convex functional and then the minimum is unique.
\end{proof}
\begin{rmk}
Let us first observe that for any $\delta>0$ it holds $\overline{u}_{\delta,\Psi} \in L^\Psi(0,T)$.\\
Moreover, we can prove that the functional $J_{\delta,\Psi}$ is coercive with respect to the weak topology in $L^1(0,T)$, i.e. for any $M>0$ there exists a weakly compact set $K_M \subset L^1(0,T)$ such that $u \in L^1(0,T)\setminus K_M$ implies $J_{\delta,\Psi}[u]>M$. Precisely, if $J_{\delta,\Psi}[u]\le M$, we have
\begin{equation*}
	\delta \cF_\Psi[u]\le J_{\delta,\Psi}[u]\le M
\end{equation*}
and then $\cF_\Psi[u]\le M/\delta$. The set $U_M=\{u \in L^1(0,T): \ \cF_\Psi[u]\le M/\delta\}$ is uniformly integrable by the de la Vall\'ee-Poussin theorem and then it is weakly relatively compact by the Dunford-Pettis theorem. Let $K_M=\overline{U_M}$, where the closure is taken in the weak topology of $L^1(0,T)$, so that $K_M$ is weakly compact. Then $J_{\delta,\Psi}[u]\le M$ implies $u \in K_M$ and, by contrapositive, we have that $u \in L^1(0,T)\setminus K_M$ implies $J_{\delta,\Psi}[u]>M$.
\end{rmk}
Now we want to show that the penalization procedure (i.e. defining the functional $J_{\delta,\Psi}$ as $J$ plus a penalization term $\delta \cF_\Psi$) generates a minimizing family for $J$.
\begin{thm}\label{thm:minseq}
	Let hypotheses (H1) to (H5) hold and consider a Young function $\Psi \in \Delta_2$. The family of functions $\{\overline{u}_{\delta,\Psi}\}_{\delta>0}$ defined in Theorem \ref{thm:minim} constitute a minimizing family for $J$.
\end{thm}
\begin{proof}
	Set $m=\inf_{u \in L^1(0,T)}J[u]$. Let us first observe that $\fS \subset L^\Psi(0,T)$. Indeed, if $u \in \fS$ then there exist $N \in \N$, $\{B_i\}_{i \le N}\subset \cB([0,T])$ and $\{b_i\}_{i \le N}\subset \R$ such that $u=\sum_{i=1}^{N}b_i\mathbf{1}_{B_i}$. We can suppose, without loss of generality, that $\bigcup_{i \le N}B_i=[0,T]$ and that $B_i \cap B_j= \emptyset$ for any $i \not = j$. Hence we have
	\begin{equation*}
		\int_0^T \Psi(|u(t)|)dt=\sum_{i=1}^{N}\int_{B_i}\Psi(|b_i|)dt= \sum_{i=1}^{N}\Psi(|b_i|)|B_i|\le T\max_{i \le N}\Psi(|b_i|)<+\infty.
	\end{equation*}
	Now let us consider any $u \in \fS$ and observe that $J_{\delta,\Psi}[u]<+\infty$. Moreover, we have 
	\begin{equation*}
		m \le J[\overline{u}_{\delta,\Psi}]\le J_{\delta,\Psi}[\overline{u}_{\delta,\Psi}]\le J_{\delta,\Psi}[u]=J[u]+\delta \cF_\Psi[u].
	\end{equation*}
	Taking the limit superior and inferior as $\delta \to 0$ we get
	\begin{equation*}
		m \le \liminf_{\delta \to 0} J[\overline{u}_{\delta,\Psi}]\le \limsup_{\delta \to 0} J[\overline{u}_{\delta,\Psi}]\le J[u].
	\end{equation*}
	 Being $u \in \fS$ arbitrary, we can take the infimum on $\fS$ and use Corollary \ref{cor:simple} to achieve
	\begin{equation*}
		m \le \liminf_{\delta \to 0} J[\overline{u}_{\delta,\Psi}]\le \limsup_{\delta \to 0} J[\overline{u}_{\delta,\Psi}]\le m,
	\end{equation*}
	obtaining $\lim_{\delta \to 0}J[\overline{u}_{\delta,\Psi}]=m$ and concluding the proof.
\end{proof}
Last theorem provides a theoretical way to construct a minimizing family for the functional $J$. Let us first stress out that, under an additional regularity assumption, the previous approach actually proves the existence of a minimizer for $J$.
\begin{thm}\label{thm:converg}
	Let hypotheses (H1) to (H5) hold and suppose there exist two Young functions $\Psi_i \in \Delta_2$, $i=1,2$, and two constants $C,\delta_0>0$ such that
	\begin{equation}\label{eq:FPsi}
		\cF_{\Psi_2}[\overline{u}_{\delta,\Psi_1}]\le C
	\end{equation}
for any $\delta\in (0,\delta_0)$, where the functions $\overline{u}_{\delta,\Psi_1}$ are defined in Theorem \ref{thm:minim}. Then there exists $\overline{u}\in L^1(0,T)$ such that
\begin{equation*}
	\inf_{u \in L^1(0,T)}J[u]=J[\overline{u}]
\end{equation*}
and $\overline{u}_{\delta,\Psi_1}\rightharpoonup \overline{u}$ as $\delta \to 0$.\\
Moreover, if hypothesis (H5+) holds, then $\overline{u}$ is unique.
\end{thm}
\begin{proof}
	Let us consider any sequence $\delta_n \downarrow 0$ with $\delta_1<\delta_0$. Equation \eqref{eq:FPsi} implies, via the de la Vall\'ee-Poussin theorem, that the sequence $\{\overline{u}_{\delta_n,\Psi_1}\}_{n \in \N}$ is uniformly integrable and then the Dunford-Pettis theorem ensures that it is weakly relatively compact. Hence, there exists $\overline{u}$ such that $\overline{u}_{\delta_n,\Psi_1}\rightharpoonup \overline{u}$. \\
	By Theorem \ref{thm:minseq} we know that $\lim_{n}J[\overline{u}_{\delta_n,\Psi_1}]=\inf_{u \in L^1(0,T)}J[u]=:m$. On the other hand, being $J$ convex and continuous in $L^1(0,T)$, we know that it is weakly lower semicontinuous and then
	\begin{equation*}
		m=\lim_{n}J[\overline{u}_{\delta_n,\Psi_1}]\ge J[\overline{u}]\ge m
	\end{equation*}
thus $J[\overline{u}]=m$, concluding the first part of the proof. The second statement follows from the strict convexity of $J$.
\end{proof}
\begin{rmk}
	The previous theorem can be proved directly by using the fact that $J_{\delta_n, \Psi}$ is a monotone sequence of operators pointwise converging to $J$, that is convex and lower semicontinuous, thus it also $\Gamma$-converges towards $J$ in the weak topology of $L^1(0,T)$ (see \cite[Remark $1.40$]{braides2002gamma}). Let us also recall that, being $L^\infty(0,T)$ not separable, the weak topology of $L^1(0,T)$ is not metrizable on closed balls, hence the more general definition of $\Gamma$-convergence on topological spaces has to be considered.
\end{rmk}
Let us stress out that if a minimizer of $J$ exists \textit{in the right space}, then we are under the hypotheses of the previous theorem.
\begin{cor}\label{cor:weakconv}
	Let hypotheses \textit{(H1)} to \textit{(H5)} hold and consider a Young function $\Psi \in \Delta_2$. Suppose there exists $\overline{u} \in L^\Psi(0,T)$ such that
	\begin{equation*}
	J[\overline{u}]=\min_{u \in L^1(0,T)}J[u].
	\end{equation*}
	Then, the family $\{\overline{u}_{\delta,\Psi}\}_{\delta>0}$ defined in Theorem \ref{thm:minim} satisfies the hypotheses of Theorem \ref{thm:converg} with $\Psi_1=\Psi_2=\Psi$.
	Moreover, if hypothesis \textit{(H5+)} holds, then $\overline{u}_{\delta_n,\Psi}\rightharpoonup \overline{u}$ for some sequence $\delta_n \to 0$.
\end{cor}
\begin{proof}
	We have to show that there exists a constant $C>0$ such that $\cF_\Psi[\overline{u}_{\delta,\Psi}]\le C$ for any $\delta>0$. To do this, let us recall that $\overline{u}_{\delta,\Psi}$ is a global minimizer of $J_{\delta,\Psi}$, thus we have
	\begin{equation*}
	J[\overline{u}_{\delta,\Psi}]+\delta \cF_\Psi[\overline{u}_{\delta,\Psi}]=J_{\delta,\Psi}[\overline{u}_{\delta,\Psi}]\le J_{\delta,\Psi}[\overline{u}]=J[\overline{u}]+\delta \cF_\Psi[\overline{u}].
	\end{equation*}
	On the other hand, being $\overline{u}$ the global minimizer of $J$, we have
	\begin{equation*}
	J[\overline{u}_{\delta,\Psi}]+\delta \cF_\Psi[\overline{u}_{\delta,\Psi}]\le J[\overline{u}]+\delta \cF_\Psi[\overline{u}]\le J[\overline{u}_{\delta,\Psi}]+\delta \cF_\Psi[\overline{u}],
	\end{equation*}
	that is to say
	\begin{equation*}
	\cF_\Psi[\overline{u}_{\delta,\Psi}]\le  \cF_\Psi[\overline{u}].
	\end{equation*}
	Setting $C=\cF_\Psi[\overline{u}]$, since $C<+\infty$ by hypothesis, we conclude the proof.
\end{proof}
\begin{rmk}
	Actually, de la Vall\'ee-Poussin theorem tells us that if we have a global minimizer $\overline{u} \in L^1(0,T)$ for $J$, then there exists a Young function $\Psi$ such that $\overline{u} \in \cL^\Psi(0,T)$. In such case, the previous theorem holds by choosing $\Psi$ as Young function even if $\Psi \not \in \Delta_2$. In conclusion, if $J$ admits a global minimizer $\overline{u} \in L^1(0,T)$, then there exists a Young function $\Psi$ such that $\overline{u}_{\delta,\Psi}$ weakly converge towards a (possibly different) minimizer of $J$ and, if $\overline{u}$ is unique, then $\overline{u}_{\delta,\Psi} \rightharpoonup \overline{u}$.
\end{rmk}
As a consequence of the weak convergence of the minimizers we obtain a form of weak convergence of the approximating processes.
\begin{thm}\label{thm:strongconv}
	Let hypotheses \textit{(H1)} to \textit{(H5+)} hold and suppose $J$ admits a global minimizer $\overline{u} \in L^{\widetilde{p}}(0,T)$ for some $\widetilde{p}>2\sqrt{2}$ and consider $\Psi(t)=t^{\widetilde{p}}$. Let $\{\overline{u}_{\delta,\Psi}\}_{\delta>0}$ be the family defined in Theorem \ref{thm:minim}. Let $X_0 \in L^2(0,T)$. Then there exists a sequence $\delta_n \to 0$ such that $\cS_{X_0}\overline{u}_{\delta_n,\Psi}\Rightarrow \cS_{X_0}\overline{u}$ in $C([0,T])$ in distribution.
\end{thm}
\begin{proof}
	Let $\mathcal{A}=\{\omega \in \Omega: \ G(\cdot,\omega)\in C([0,T])\}$ and recall that $\bP(\mathcal A)=1$. Fix $t>0$, $\omega \in \mathcal{A}$ and observe that, by Corollary \ref{linear} and Equation \eqref{eq:solutionmap},
\begin{align*}
	|\cS_{X_0}\overline{u}_{\delta,\Psi}(t,\omega)-\cS_{X_0}\overline{u}(t,\omega)|=|\cS_0(\overline{u}_{\delta,\Psi}-\overline{u})(t,\omega)|=G(t,\omega)e^{A(t)}\left|\int_0^t\frac{e^{-A(s)}}{G(s,\omega)}\overline{u}_{\delta,\Psi}(s)ds-\int_0^t\frac{e^{-A(s)}}{G(s,\omega)}\overline{u}(s)ds\right|.
\end{align*}
	Consider $\delta_n \to 0$ as in Corollary \ref{cor:weakconv}, so that $\overline{u}_{\delta_n,\Psi}\rightharpoonup \overline{u}$. Since $\frac{e^{-A(\cdot)}}{G(\cdot,\omega)}$ is a continuous function, we have that
	\begin{equation*}
		\int_0^t\frac{e^{-A(s)}}{G(s,\omega)}\overline{u}_{\delta_n,\Psi}(s)ds \to \int_0^t\frac{e^{-A(s)}}{G(s,\omega)}\overline{u}(s)ds
	\end{equation*}
	and then
	\begin{equation}\label{eq:limit}
		\lim_{n \to +\infty}|\cS_{X_0}\overline{u}_{\delta_n,\Psi}(t,\omega)-\cS_{X_0}\overline{u}(t,\omega)|=0, \ \forall \omega \in \mathcal A,
	\end{equation}
	that is to say that, for fixed $t \in [0,T]$, it holds $\cS_{X_0}\overline{u}_{\delta_n,\Psi}(t) \to \cS_{X_0}\overline{u}(t)$ almost surely. Let us observe that this is enough to guarantee the convergence in any finite-dimensional distribution. Indeed, consider $N \in \N$ and $t_1,\dots,t_N \in [0,T]$ and fix $\omega \in \mathcal A$. Since Equation \eqref{eq:limit} holds for any $t_i$, $i=1,\dots,N$, we have that for any fixed $\varepsilon>0$ there exists $\nu_i \in \N$ such that if $n \ge \nu_i$ it holds
	\begin{equation*}
		|\cS_{X_0}\overline{u}_{\delta_n,\Psi}(t_i,\omega)-\cS_{X_0}\overline{u}(t_i,\omega)|<\frac{\varepsilon}{N}.
	\end{equation*}
	Let $\nu=\max\{\nu_1,\dots,\nu_N\}$ and consider $n \ge \nu$. By the triangular inequality we get
	\begin{equation*}
		|\c(S_{X_0}\overline{u}_{\delta_n,\Psi}(t_i,\omega))_{i \le N}-(\cS_{X_0}\overline{u}(t_i,\omega))_{i \le N}| \le \sum_{i=1}^{N}|\cS_{X_0}\overline{u}_{\delta_n,\Psi}(t_i,\omega)-\cS_{X_0}\overline{u}(t_i,\omega))_{i \le n}|<\varepsilon,
	\end{equation*}
so that
\begin{equation*}
	\lim_{n \to +\infty}|(\cS_{X_0}\overline{u}_{\delta_n,\Psi}(t_i,\omega))_{i \le N}-(\cS_{X_0}\overline{u}(t_i,\omega))_{i \le N}|=0, \ \forall \omega \in \mathcal A.
\end{equation*}
This implies that $(\cS_{X_0}\overline{u}_{\delta_n,\Psi}(t_i))_{i \le N} \to (\cS_{X_0}\overline{u}(t_i))_{i \le N}$ almost surely and thus in distribution. To extend the convergence in distribution to the whole paths, we need to show that the sequence $\cS_{X_0}\overline{u}_{\delta,\Psi_n}$ is tight.\\
Let us denote, for simplicity, $X_\delta=\cS_{X_0}\overline{u}_{\delta,\Psi}$. Consider $0 \le t_1 < t_2 \le T$ and observe that, by Equation \eqref{eq:solutionmap}, it holds
	\begin{align*}
	X_\delta(t_2)-X_\delta(t_1)=(G(t_2)e^{A(t_2)}-G(t_1)e^{A(t_1)})\int_0^{t_2} \frac{e^{-A(s)}}{G(s)}\overline{u}_{\delta,\Psi}(s)ds+G(t_2)e^{A(t_1)}\int_{t_1}^{t_2}\frac{e^{-A(s)}}{G(s)}\overline{u}_{\delta,\Psi}(s)ds.
	\end{align*}
	Being $\widetilde{p}>2\sqrt{2}$, we have
	\begin{equation*}
	\frac{4+\widetilde{p}}{3\widetilde{p}-1}<\frac{4+\widetilde{p}}{4}<\frac{4}{\widetilde{p}}<\frac{\widetilde{p}}{2}<\widetilde{p}-1.
\end{equation*}
Hence, we can consider $p_1 \in \left(\frac{4}{\widetilde{p}},\frac{\widetilde{p}}{2}\right)$, so that
\begin{equation*}
	1<\frac{p_1+1}{p_1}<2<\frac{\widetilde{p}}{p_1}<\frac{4\widetilde{p}}{4+\widetilde{p}},
\end{equation*}
and $\overline{p} \in \left(2,\min\left\{\frac{\widetilde{p}}{p_1},4\right\}\right)$. By convexity inequality,
	\begin{align}\label{eq:passdist7}
		\begin{split}
	\E[|X_\delta(t_2)-X_\delta(t_1)|^{\overline{p}}]&\le 2^{\overline{p}-1}\left( \E\left[|G(t_2)e^{A(t_2)}-G(t_1)e^{A(t_1)}|^{\overline{p}}\left|\int_0^{t_2} \frac{e^{-A(s)}}{G(s)}\overline{u}_{\delta,\Psi}(s)ds\right|^{\overline{p}}\right]\right.\\&\qquad \left.+\E\left[G^{\overline{p}}(t_1)e^{\overline{p}A(t_1)}\left|\int_{t_1}^{t_2}\frac{e^{-A(s)}}{G(s)}\overline{u}_{\delta,\Psi}(s)ds\right|^{\overline{p}}\right]\right)\\
	&=2^{\overline{p}-1}(I_1+I_2).
\end{split}	
\end{align}
	Let us first work with $I_2$. Using H\"older's inequality with exponent $p_1$ (and $q_1$ such that $1/p_1+1/q_1=1$) and Jensen's inequality we achieve
	\begin{align}\label{eq:passdist0}
		\begin{split}
		I_2 &\le \E[G^{\overline{p}q_1}(t_1)e^{\overline{p}q_1A(t_1)}]^{\frac{1}{q_1}}(t_2-t_1)^{\overline{p}}\E\left[\left|\frac{1}{(t_2-t_1)}\int_{t_1}^{t_2}\frac{e^{-A(s)}}{G(s)}\overline{u}_{\delta,\Psi}(s)ds\right|^{\overline{p}p_1}\right]^{\frac{1}{p_1}}\\
		&\le \E[G^{\overline{p}q_1}(t_1)e^{\overline{p}q_1A(t_1)}]^{\frac{1}{q_1}}(t_2-t_1)^{\overline{p}-\frac{1}{p_1}}\E\left[\int_{t_1}^{t_2}\frac{e^{-\overline{p}p_1A(s)}}{G^{\overline{p}p_1}(s)}|\overline{u}_{\delta,\Psi}(s)|^{\overline{p}p_1}ds\right]^{\frac{1}{p_1}}.
		\end{split}
	\end{align}
	Set $\gamma_1=\overline{p}-\frac{1}{p_1}$ with $\gamma_1>1$ by the choice of $\overline{p}$ and $p_1$. Now observe that
	\begin{align*}
		\E[G^{\overline{p}q_1}(t_1)e^{\overline{p}q_1A(t_1)}]^{\frac{1}{q_1}}&\le \left(\sup_{t \in [0,T]}e^{\overline{p}A(t)}\right)\E\left[\sup_{t \in [0,T]}|G(t)|^{\overline{p}{q_1}}\right]^{\frac{1}{q_1}}\\
		&\le \left(\sup_{t \in [0,T]}e^{\overline{p}A(t)}\right)\left(\frac{\overline{p}{q_1}}{\overline{p}{q_1}-1}\right)^{\overline{p}}e^{\frac{\overline{p}(\overline{p}{q_1}-1)}{2}T}=:C_1(T,\overline{p},p_1),
	\end{align*}
	where we also used Doob's maximal inequality. On the other hand, we have
	\begin{align}\label{eq:passdist1}
		\begin{split}
		\E\left[\int_{t_1}^{t_2}\frac{e^{-\overline{p}p_1A(s)}}{G^{\overline{p}p_1}(s)}|\overline{u}_{\delta,\Psi}(s)|^{\overline{p}p_1}ds\right]^{\frac{1}{p_1}}&\le \E\left[\int_{0}^{T}e^{\overline{p}p_1(s-A(s))}(G'(s))^{\overline{p}p_1}|\overline{u}_{\delta,\Psi}(s)|^{\overline{p}p_1}ds\right]^{\frac{1}{p_1}}\\
		&\le \left(\sup_{t \in [0,T]}e^{\overline{p}(t-A(t))}\right)\left(\int_0^T|\overline{u}_{\delta,\Psi}|^{\overline{p}p_1}(s)ds\right)^{\frac{1}{p_1}}\E\left[\sup_{t \in [0,T]}|G'(t)|^{\overline{p}p_1}\right]^{\frac{1}{p_1}}\\
		&\le \left(\sup_{t \in [0,T]}e^{\overline{p}(t-A(t))}\right)e^{\frac{\overline{p}(\overline{p}{p_1}-1)}{2}T}\left(\int_0^T|\overline{u}_{\delta,\Psi}|^{\overline{p}p_1}(s)ds\right)^{\frac{1}{p_1}}.
		\end{split}
	\end{align}
	Concerning the last remaining integral in the previous inequality, let us observe that, by definition of $\overline{p}$ and $p_1$, it holds $\frac{\widetilde{p}}{\overline{p}p_1}>1$ hence, we can use it as exponent in H\"older's inequality, obtaining
	\begin{equation*}
		\left(\int_0^T|\overline{u}_{\delta,\Psi}|^{\overline{p}p_1}(s)ds\right)^{\frac{1}{p_1}}\le \left(\int_0^T|\overline{u}_{\delta,\Psi}|^{\widetilde{p}}(s)ds\right)^{\frac{\overline{p}}{\widetilde{p}}}T^\frac{\widetilde{p}-\overline{p}p_1}{\widetilde{p}}.
	\end{equation*}
	Arguing as in Corollary \ref{cor:weakconv}, we know that $\cF_\Psi[\overline{u}_{\delta,\Psi}]\le \cF_\Psi[\overline{u}]$ and then
	\begin{equation*}
		\left(\int_0^T|\overline{u}_{\delta,\Psi}|^{\overline{p}p_1}(s)ds\right)^{\frac{1}{p_1}}\le \left(\int_0^T|\overline{u}|^{\widetilde{p}}(s)ds\right)^{\frac{\overline{p}}{\widetilde{p}}}T^\frac{\widetilde{p}-\overline{p}p_1}{\widetilde{p}}=:C_2(T,\widetilde{p},\overline{p},p_1,\overline{u}).
	\end{equation*}
	Plugging last inequality in Equation \eqref{eq:passdist1} we get
	\begin{align*}
		\begin{split}
			\E\left[\int_{t_1}^{t_2}\frac{e^{-\overline{p}p_1A(s)}}{G^{\overline{p}p_1}(s)}|\overline{u}_{\delta,\Psi}(s)|^{\overline{p}p_1}ds\right]^{\frac{1}{p_1}}&\le \left(\sup_{t \in [0,T]}e^{\overline{p}(t-A(t))}\right)e^{\frac{\overline{p}(\overline{p}{p_1}-1)}{2}T}C_2(T,\widetilde{p},\overline{p},p_1,\overline{u})=:C_3(T,\widetilde{p},\overline{p},p_1,\overline{u}).
		\end{split}
	\end{align*}
	Setting then $C_4(T,\widetilde{p},\overline{p},p_1,\overline{u}):=C_3(T,\widetilde{p},\overline{p},p_1,\overline{u})C_1(T,\overline{p},p_1)$ we obtain, from Equation \eqref{eq:passdist0},
	\begin{equation}\label{eq:passdist6}
		I_2 \le C_4(T,\widetilde{p},\overline{p},p_1,\overline{u})(t_2-t_1)^{\gamma_1}.
	\end{equation}
	Now let us consider $I_1$. By H\"older's inequality with exponent $\frac{4}{\overline{p}}>1$ we get
	\begin{equation}\label{eq:passdist2}
		I_1\le \E\left[|G(t_2)e^{A(t_2)}-G(t_1)e^{A(t_1)}|^4\right]^{\frac{\overline{p}}{4}}\E\left[\left|\int_0^{t_2}\frac{e^{-A(s)}}{G(s)}\overline{u}_{\delta,\Psi}(s)ds\right|^{\frac{4\overline{p}}{4-\overline{p}}}\right]^{\frac{4-\overline{p}}{4}}.
	\end{equation}
	Let us first consider the first factor of $I_1$. Let $Y(t)=G(t)e^{A(t)}$ and observe, by It\^o's formula, that
	\begin{equation*}
		dY(t)=a(t)Y(t)dt+Y(t)dW(t)
	\end{equation*}
	in $[t_1,t_2]$, that is to say 
	\begin{equation*}
	Y(t)=Y(t_1)+\int_{t_1}^{t} a(s)Y(s)ds + \int_{t_1}^t Y(s)dW(s).
	\end{equation*}
	For any $t \in [0,t_2-t_1]$ it holds 
\begin{align*}
Y(t+t_1)&=Y(t_1)+\int_{t_1}^{t+t_1} a(s)Y(s)ds + \int_{t_2}^{t+t_2} Y(s)dW(s) \\
&= Y(t_1)+\int_0^t a(s+t_1)Y(s+t_1)ds +\int_0^t Y(s+t_1)d\widetilde W(s)\\
&= Y(t_1)+\int_0^t a(s+t_1)(Y(s+t_1)-Y(t_1))ds+\int_0^t a(s+t_1)Y(t_1)ds \\
&\qquad +\int_0^t (Y(s+t_1)-Y(t_1))d\widetilde W(s)+ \int_0^t Y(t_1)d\widetilde W(s),
\end{align*}	
where we used the change of variables $s \mapsto s-t_1$ and we set $\widetilde W(t):=W(t+t_1)-W(t)$, that is still a Brownian motion. Hence, the process $\widetilde Y(t)=Y(t+t_1)-Y(t_1)$ solves the SDE 
\begin{equation*}
d\widetilde Y(t)=a(t+t_1)\left(\widetilde Y(t)+Y(t_1) \right)dt+(\widetilde Y(t)+Y(t_1))d\widetilde W(t), \quad \widetilde Y(0)=0
\end{equation*}
with $t\in [0,t_1-t_2]$. We can extend the process by setting $\widetilde Y(t)=Y(t_2)-Y(t_1)$ as $t\in [t_2-t_1,T]$ so that:
\begin{equation*}
d\widetilde Y(t)=\widetilde a(t)\left(\widetilde Y(t)+Y(t_1) \right)dt+(\widetilde Y(t)+Y(t_1))\mathbf{1}_{[0,t_2-t_1]}d\widetilde W(t), \quad \widetilde Y(0)=0,
\end{equation*}
where $\widetilde a(t)=a(t+t_1)$ if $t\in[0,t_2-t_1]$ and $\widetilde a(t)=0$ if $t\in(t_2-t_1,T]$. Now set $M=\Norm{a}{L^\infty(0,T)}\geq \Norm{\widetilde a}{L^\infty(0,T)}$ and observe that by Lemma \ref{lem:momest}, 
	\begin{align*}
		\sup_{t \in [0,t_2-t_1]}\E[|\widetilde{Y}(t)|^4]&\le K(2,M,t_2-t_1)\left(\left(\int_0^{t_2-t_1} \E[|\widetilde{a}(t){Y}(t_1)|^{4}]^{\frac{1}{4}}dt\right)^{4}+\left(\int_0^{t_2-t_1} \E[|{Y}(t_1)|^{4}]^{\frac{1}{2}}dt\right)^{2}\right)\\
		&\le K(2,M,T)\left(\E[|{Y}(t_1)|^{4}]\left(\int_0^{t_2-t_1}|\widetilde{a}(t)| dt\right)^{4}+\E[|{Y}(t_1)|^{4}](t_2-t_1)^2\right)\\
		&\le K(2,M,T)\E[|Y(t_1)|^{4}](t_2-t_1)^2\left(MT^2+1\right),
	\end{align*}
where we used the fact that $T \mapsto K(2,M,T)$ is increasing and $t_2-t_1 \le T$.
	Calling back that
	\begin{equation*}
		\E[|Y(t_1)|^4]=e^{4A(t_1)}\E[G^4(t_1)]\le \left(\sup_{t \in [0,T]}e^{4A(t)}\right)e^{6T}
	\end{equation*}
	by Equation \eqref{eq:momgeom} and setting $C_5(M,T):=\left(MT^2+1\right)K(2,M,T)\left(\sup_{t \in [0,T]}e^{4A(t)}\right)e^{6T}$ we get
	\begin{align*}
		\E[|Y(t_2)-Y(t_1)|^4]=\E[|\widetilde{Y}(t_2-t_1)|^4]\leq \sup_{t \in [0,t_1-t_2]}\E[|\widetilde{Y}(t)|^4]\le C_5(M,T)(t_2-t_1)^2.
	\end{align*}
	Going back to Equation \eqref{eq:passdist2}, setting $\gamma_2=\frac{\overline{p}}{2}$, where $\gamma_2>1$ since $\overline{p}>2$, we have
	\begin{equation}\label{eq:passdist4}
		I_1 \le  (C_5(M,T))^{\frac{\overline{p}}{4}}(t_2-t_1)^{\gamma_2}\E\left[\left|\int_0^{t_2}\frac{e^{-A(s)}}{G(s)}\overline{u}_{\delta,\Psi}(s)ds\right|^{\frac{4\overline{p}}{4-\overline{p}}}\right]^{\frac{4-\overline{p}}{4}}.
	\end{equation}
	Now we have to estimate the second factor. Let us first use Jensen's inequality (observing that $\frac{4\overline{p}}{4-\overline{p}}>1$) to get
\begin{align*}
	\E\left[\left|\int_0^{t_2}\frac{e^{-A(s)}}{G(s)}\overline{u}_{\delta,\Psi}(s)ds\right|^{\frac{4\overline{p}}{4-\overline{p}}}\right]&\le t_2^{\frac{4\overline{p}}{4-\overline{p}}-1}\E\left[\int_0^{t_2}\frac{e^{-\frac{4\overline{p}}{4-\overline{p}}A(s)}}{G^\frac{4\overline{p}}{4-\overline{p}}(s)}|\overline{u}_{\delta,\Psi}(s)|^{\frac{4\overline{p}}{4-\overline{p}}}ds\right]\\
	&\le T^{\frac{4\overline{p}}{4-\overline{p}}-1}\E\left[\int_0^{T}e^{\frac{4\overline{p}}{4-\overline{p}}(s-A(s))}|G^\prime(s)|^\frac{4\overline{p}}{4-\overline{p}}|\overline{u}_{\delta,\Psi}(s)|^{\frac{4\overline{p}}{4-\overline{p}}}ds\right]\\
	&\le T^{\frac{4\overline{p}}{4-\overline{p}}-1}\left(\sup_{t \in [0,T]}e^{\frac{4\overline{p}}{4-\overline{p}}(s-A(s))}\right)\E\left[\sup_{s \in [0,T]}|G^\prime(s)|^\frac{4\overline{p}}{4-\overline{p}}\right]\int_0^{T}|\overline{u}_{\delta,\Psi}(s)|^{\frac{4\overline{p}}{4-\overline{p}}}ds.
\end{align*}
	By Doob's maximal inequality we have
	\begin{equation*}
		\E\left[\sup_{s \in [0,T]}|G^\prime(s)|^\frac{4\overline{p}}{4-\overline{p}}\right]\le \left(\frac{4\overline{p}}{5\overline{p}-4}\right)^{\frac{4\overline{p}}{4-\overline{p}}}e^{\frac{2\overline{p}}{4-\overline{p}}\left(\frac{4\overline{p}}{4-\overline{p}}-1\right)T}.
	\end{equation*}
	On the other hand, let us observe that, being $\overline{p}<\frac{4\widetilde{p}}{4+\widetilde{p}}$, it holds $\frac{\widetilde{p}(4-\overline{p})}{4\overline{p}}>1$, thus we can use it as exponent for H\"older's inequality, obtaining
	\begin{equation*}
		\int_0^{T}|\overline{u}_{\delta,\Psi}(s)|^{\frac{4\overline{p}}{4-\overline{p}}}ds\le \left(\int_0^{T}|\overline{u}_{\delta,\Psi}(s)|^{\widetilde{p}}ds\right)^{\frac{4 \overline{p}}{\widetilde p(4-\overline{p}) }}T^{1-\frac{4 \overline{p}}{\widetilde p(4-\overline{p})}}\le \left(\int_0^{T}|\overline{u}(s)|^{\widetilde{p}}ds\right)^{\frac{4 \overline{p}}{\widetilde p(4-\overline{p})}}T^{1-\frac{4 \overline{p}}{\widetilde p(4-\overline{p})}}.
	\end{equation*}
 Hence we get, from Equation \eqref{eq:passdist4},
\begin{align}\label{eq:passdist5}
	\begin{split}
	I_1 &\le  (C_5(M,T))^{\frac{\overline{p}}{4}}(t_1-t_2)^{\gamma_2}T^{\overline{p}-\frac{\overline{p}}{\widetilde{p}}}\left(\sup_{t \in [0,T]}e^{-\overline{p}A(s)}\right)\left(\frac{4\overline{p}}{5\overline{p}-4}\right)^{\overline{p}}e^{\frac{\overline{p}}{2}\left(\frac{4\overline{p}}{4-\overline{p}}+1\right)T}\left(\int_0^{T}|\overline{u}(s)|^{\widetilde{p}}ds\right)^{\frac{\overline{p}}{\widetilde{p}}}\\
	&=:C_6(M,T,\widetilde{p},\overline{p},p_1,\overline{u})(t_2-t_1)^{\gamma_2}.
\end{split}
\end{align}
Now set $\gamma=\min\{\gamma_1,\gamma_2\}>1$ and combine Equations \eqref{eq:passdist6} and \eqref{eq:passdist5} with \eqref{eq:passdist7} to conclude that there exists $C_7(M,T,\widetilde{p},\overline{p},p_1,\overline{u})>0$ (notice that $C_7$ does not depend on $t_1,t_2$) such that
\begin{equation*}
	\E[|X_\delta(t_2)-X_\delta(t_1)|^{\overline{p}}]\le C_7(M,T,\widetilde{p},\overline{p},p_1,\overline{u})(t_2-t_1)^\gamma.
\end{equation*}
This, together with the fact that $X_\delta(0)=0$ for any $\delta>0$, guarantees that $\{X_\delta\}_{\delta>0}$ is tight (see, for instance, \cite[Theorem $11.6.5$]{whitt2002stochastic}). Thus, by a Corollary of Prohorov's theorem (see, for instance, \cite[Corollary $11.6.2$]{whitt2002stochastic}), we know that $X_{\delta_n} \Rightarrow \cS_{X_0}\overline{u}$ in $C([0,T])$, concluding the proof.
\end{proof}
\begin{rmk}
	The previous theorem clearly holds even if $\Psi(t)=Ct^{\widetilde{p}}$ for some constant $C>0$.
\end{rmk}
Theorem \ref{thm:strongconv} guarantees that even if $\overline{u}_{\delta,\Psi}$ does not converge strongly to $\overline{u}$ (due, for instance, to a highly oscillatory behaviour), it can be still used to approximate the process $\cS_{X_0} \overline{u}$. This comes in handy in the application context, whenever one has to numerically determine some functional properties of $\cS_{X_0} \overline{u}$. Moreover, in the proof of the previous theorem, we have also shown that if $\Psi \in \Delta_2$ is a Young function and $\Norm{\overline{u}_{\delta,\Psi}}{L^{\widetilde{p}}(0,T)}$ is uniformly bounded for some $\widetilde{p}>2\sqrt{2}$, then the family $\{X_{\delta}\}_{\delta>0}$ is tight. This means that, in this case, by Prohorov's theorem, $\{X_{\delta}\}_{\delta>0}$ is relatively compact, i.e. there exists a process $\overline{X}$ with a.s. continuous sample paths and a sequence $\delta_n \to 0$ such that $X_{\delta_n}\Rightarrow \overline{X}$. Combining the latter observation with Theorem \ref{thm:converg} we have that not only in this case $J$ admits a minimizer, but there exists a sequence $\delta_n \to 0$ such that $\overline{u}_{\delta_n,\Psi}\rightharpoonup \overline{u}$, $X_{\delta_n}\Rightarrow \overline{X}$ and $\overline{X}=\cS_{X_0}\overline{u}$.\\
Let us now exhibit a necessary and sufficient condition for a function $\overline{u}_{\delta,\Psi}$ to be a minimizer  of $J_{\delta,\Psi}$, given in terms of an Euler-Lagrange type equation.
\begin{thm}\label{thm:ELpen}
	Let hypotheses (H1) to (H5) hold and $\Psi$ be a Young function with strictly increasing continuous derivative $\psi$. Then $\overline{u}_{\delta,\Psi}$ is the unique solution of
	\begin{equation}\label{eq:ELpen}
		\delta \frac{\overline{u}_{\delta,\Psi}(t_0)}{|\overline{u}_{\delta,\Psi}(t_0)|} \psi(|\overline{u}_{\delta,\Psi}(t_0)|)=\int_{t_0}^T\E\left[\pd{F}{\xi}(t,\xi_{\overline{u}_{\delta,\Psi}}(t))e^{A(t)-A(t_0)}\frac{G(t)}{G(t_0)}\right]dt, \ \forall t_0 \in [0,T].
	\end{equation}
\end{thm}
\begin{proof}
	The proof follows as in Theorem \ref{thm:EL}. Precisely, let $E_{\overline{u}}$ be the set of Lebesgue points of $\overline{u}_{\delta,\Psi}$ in $(0,T)$, $E_L$ the set of Lebesgue points of $L$ in $(0,T)$ and $E=E_{\overline{u}}\cup E_L$. Let $t_0 \in E$, fix a real number $u \in \R$ and $\varepsilon_0>0$ small enough to have $\left(t_0-\frac{\varepsilon_0}{2},t_0+\frac{\varepsilon_0}{2}\right)\subset (0,T)$. Define, for any $\varepsilon \in (0,\varepsilon_0)$, $I_\varepsilon:=\left(t_0-\frac{\varepsilon}{2},t_0+\frac{\varepsilon}{2}\right)$, 
	\begin{equation*}
		u_\varepsilon(t)=\begin{cases} u & t \in I_\varepsilon \\\overline{u}_{\delta,\Psi}(t) &\mbox{otherwise},
		\end{cases}
	\end{equation*}
$g_1(\varepsilon):=J_{\delta,\Psi}[u_\varepsilon]$, $g_2(\varepsilon):=J[u_\varepsilon]$ and $g_3(\varepsilon):=\cF_\Psi[u_\varepsilon]$, so that $g_1=g_2+\delta g_3$ and
\begin{equation*}
	\frac{g_1(\varepsilon)-g_1(0)}{\varepsilon}=\frac{g_2(\varepsilon)-g_2(0)}{\varepsilon}+\delta \frac{g_3(\varepsilon)-g_3(0)}{\varepsilon}.
\end{equation*} 
Let us only study the second incremental ratio. It holds
\begin{equation*}
	\frac{g_3(\varepsilon)-g_3(0)}{\varepsilon}=\frac{1}{\varepsilon}\int_0^T (\Psi(u_\varepsilon(t))-\Psi(\overline{u}_{\delta,\Psi}(t)))dt=\frac{1}{\varepsilon}\int_{I_\varepsilon} (\Psi(u)-\Psi(\overline{u}_{\delta,\Psi}(t)))dt.
\end{equation*}
Being $\Psi \in \Delta_2$,
$\Psi(u)-\Psi(\overline{u}_{\delta,\Psi}(t))$ belongs to $L^1(0,T)$, hence we can consider $E_\Psi$ as the set of Lebesgue points of $\Psi(\overline{u}_{\delta,\Psi}(t))$ and $E^\prime=E \cap E_\Psi$. From now on, let us assume that $t_0 \in E^\prime$. Taking the limit we obtain
\begin{align*}
	\lim_{\varepsilon \to 0}\frac{g_3(\varepsilon)-g_3(0)}{\varepsilon}&=\Psi(u)-\Psi(\overline{u}_{\delta,\Psi}(t_0)).
\end{align*}
 On the other hand, we have, from Equation \eqref{eq:derg},
 \begin{equation*}
 	\lim_{\varepsilon \to 0}\frac{g_2(\varepsilon)-g_2(0)}{\varepsilon}=-(u-\overline{u}_{\delta,\Psi}(t_0))\int_{t_0}^T\E\left[\pd{F}{\xi}(t,\xi_{\overline{u}_{\delta,\Psi}}(t))e^{A(t)-A(t_0)}\frac{G(t)}{G(t_0)}\right]dt,
 \end{equation*}
with the notation introduced in Theorem \ref{thm:EL}. Hence we conclude that
\begin{equation*}
	\lim_{\varepsilon \to 0}\frac{g_1(\varepsilon)-g_1(0)}{\varepsilon}=\delta\Psi(u)-\delta\Psi(\overline{u}_{\delta,\Psi}(t_0))-(u-\overline{u}_{\delta,\Psi}(t_0))\int_{t_0}^T\E\left[\pd{F}{\xi}(t,\xi_{\overline{u}_{\delta,\Psi}}(t))e^{A(t)-A(t_0)}\frac{G(t)}{G(t_0)}\right]dt.
\end{equation*}
However, we know that $0$ is a minimum point of $g_1$, thus it holds
\begin{equation}\label{eq_H1}
	\delta \Psi(|u|)-\delta\Psi(|\overline{u}_{\delta,\Psi}(t_0)|)-(u-\overline{u}_{\delta,\Psi}(t_0))\int_{t_0}^T\E\left[\pd{F}{\xi}(t,\xi_{\overline{u}_{\delta,\Psi}}(t))e^{A(t)-A(t_0)}\frac{G(t)}{G(t_0)}\right]dt\ge 0.
\end{equation}
Setting 
\begin{equation*}
	H(u):=-\delta \Psi(|u|)+u\int_{t_0}^T\E\left[\pd{F}{\xi}(t,\xi_{\overline{u}_{\delta,\Psi}}(t))e^{A(t)-A(t_0)}\frac{G(t)}{G(t_0)}\right]dt,
\end{equation*}
by Equation \ref{eq_H1} it holds $\max_{u \in \R}H(u)=H(\overline{u}_{\delta,\Psi}(t_0))$. 
By Remark \ref{remark_u0} we know that $H$ is differentiable and then by Fermat's theorem $H^\prime(\overline{u}_{\delta,\Psi}(t_0))=0$, that is to say 
\begin{equation*}
	\delta \frac{\overline{u}_{\delta,\Psi}(t_0)}{|\overline{u}_{\delta,\Psi}(t_0)|}\psi(|\overline{u}_{\delta,\Psi}(t_0)|)=\int_{t_0}^T\E\left[\pd{F}{\xi}(t,\xi_{\overline{u}_{\delta,\Psi}}(t))e^{A(t)-A(t_0)}\frac{G(t)}{G(t_0)}\right]dt, \ \forall t_0 \in E^\prime.
\end{equation*}
Finally, the right-hand side being continuous (as we have shown in the proof of Theorem \ref{thm:EL}), we achieve Equation \eqref{eq:ELpen}.\\

Now we have to show that the latter is also a sufficient condition. To do this, let $\overline{u}_{\delta,\Psi} \in L^1(0,T)$ be a solution of its and $u \in L^1(0,T)$ be any other function. If $u \not \in L^\Psi(0,T)$, then $+\infty=J_{\delta,\Psi}[u]\ge J_{\delta,\Psi}[\overline{u}_{\delta,\Psi}]$. Thus, let us consider $u \in L^\Psi(0,T)$. Observe that
\begin{equation*}
	J_{\delta,\Psi}[u]-J_{\delta,\Psi}[\overline{u}_{\delta,\Psi}]=J[u]-J[\overline{u}_{\delta,\Psi}]+\delta (\cF_{\Psi}[u]-\cF_\Psi[\overline{u}_{\delta,\Psi}]).
\end{equation*}
We already know, as it is shown in Theorem \ref{thm:suffcond}, that
\begin{equation}\label{eq:ineqJ}
	J[u]-J[\overline{u}_{\delta,\Psi}]\ge \int_0^T(u(s)-\overline{u}_{\delta,\Psi}(s))\left(-\int_s^T\E\left[\pd{F}{\xi}(t,\xi_{\overline{u}_{\delta,\Psi}}(t))e^{A(t)-A(s)}\frac{G(t)}{G(s)}\right]dt\right)ds.
\end{equation}
Moreover, by the convexity of $\Psi$, we get
\begin{equation*}
	\Psi(|u(s)|)-\Psi(|\overline{u}_{\delta,\Psi}(s)|)\ge \frac{\overline{u}_{\delta,\Psi}(s)}{|\overline{u}_{\delta,\Psi}(s)|}\psi(|\overline{u}_{\delta,\Psi}(s)|)(u(s)-\overline{u}_{\delta,\Psi}(s)), \ \forall s \in (0,T),
\end{equation*}
and then
\begin{equation}\label{eq:ineqF}
	\cF_{\Psi}[u]-\cF_\Psi[\overline{u}_{\delta,\Psi}]\ge \int_0^T \frac{\overline{u}_{\delta,\Psi}(s)}{|\overline{u}_{\delta,\Psi}(s)|}\psi(|\overline{u}_{\delta,\Psi}(s)|)(u(s)-\overline{u}_{\delta,\Psi}(s))ds. 
\end{equation}
Combining Equation \eqref{eq:ineqJ} and \eqref{eq:ineqF} we get
\begin{align*}
	J_{\delta,\Psi}[u]-J_{\delta,\Psi}[\overline{u}_{\delta,\Psi}]&\ge \int_0^T(u(s)-\overline{u}_{\delta,\Psi}(s))\\&\quad \times \left(\frac{\overline{u}_{\delta,\Psi}(s)}{|\overline{u}_{\delta,\Psi}(s)|}\psi(|\overline{u}_{\delta,\Psi}(s)|)-\int_s^T\E\left[\pd{F}{\xi}(t,\xi_{\overline{u}_{\delta,\Psi}}(t))e^{A(t)-A(s)}\frac{G(t)}{G(s)}\right]dt\right)ds=0,
\end{align*}
that is to say that $\overline{u}_{\delta,\Psi}$ is a minimizer for $J_{\delta,\Psi}$ in $L^1(0,T)$.\\
Hence we conclude that any solution of Equation \eqref{eq:ELpen} is a minimizer for $J_{\delta,\Psi}$ in $L^1(0,T)$. However, being $\Psi$ strictly convex, $\overline{u}_{\delta,\Psi}$ is the unique minimizer of $J_{\delta,\Psi}$. Thus $\overline{u}_{\delta,\Psi}$ is the unique solution of equation \eqref{eq:ELpen}.
\end{proof}
\begin{rmk}
	Let us stress out, as we did for $J$, that $\cF_\Psi$ is Gateaux-differentiable in $L^\Psi(0,T)$ with Gateaux derivative
	\begin{equation*}
		\partial_u \cF[v]=\int_0^T \frac{u(s)v(s)}{|u(s)|}\psi(|u(s)|)ds, \ v \in L^\Psi(0,T).
	\end{equation*}
One can show that the previous quantity is finite for any $v \in L^\Psi(0,T)$ by means of  H\"older's inequality for Orlicz spaces \cite[Theorem $4.7.5$]{pick2012function}. As a consequence, we obtain that also $J_{\delta,\Psi}[u]$ is Gateaux-differentiable in $L^\Psi(0,T)$ with Gateaux derivative
\begin{equation*}
	\partial_u J_{\delta,\Psi}[v]=\partial_u J[v]+\delta \partial_u \cF_\Psi[v]
\end{equation*}
and Equation \eqref{eq:ELpen} follows from Fermat's theorem. For this reason, we can refer to \eqref{eq:ELpen} as the Euler-Lagrange equation for $J_{\delta,\Psi}$. Let us emphasize that the function $H(u)$ defined in the proof of Theorem \ref{thm:ELpen} is, in some sense, the Hamiltonian function associated to $J_{\delta,\Psi}$.
\end{rmk}
With a suitable choice of the Young function $\Psi$, we can guarantee better regularity for $\overline{u}_{\delta,\Psi}$.
\begin{cor}
	Let hypotheses (H1) to (H5) hold and set $\Psi(t)=\frac{t^{2n}}{2n}$ for any $n=1,2,3,\dots$. Then $\overline{u}_{\delta,\Psi}$ is continuous.
\end{cor}
\begin{proof}
	With this choice, we have $\psi(t)=t^{2n-1}$ and then Equation \eqref{eq:ELpen} becomes
	\begin{equation*}
		\delta \frac{\overline{u}_{\delta,\Psi}(t_0)}{|\overline{u}_{\delta,\Psi}(t_0)|} |\overline{u}_{\delta,\Psi}(t_0)|^{2n-1}=\int_{t_0}^T\E\left[\pd{F}{\xi}(t,\xi_{\overline{u}_{\delta,\Psi}}(t))e^{A(t)-A(t_0)}\frac{G(t)}{G(t_0)}\right]dt, \ \forall t_0 \in [0,T],
	\end{equation*}
that can be recast as
\begin{equation*}
\overline{u}_{\delta,\Psi}(t_0)=\sqrt[2n-1]{\frac{1}{\delta}\int_{t_0}^T\E\left[\pd{F}{\xi}(t,\xi_{\overline{u}_{\delta,\Psi}}(t))e^{A(t)-A(t_0)}\frac{G(t)}{G(t_0)}\right]dt}, \ \forall t_0 \in [0,T].
\end{equation*}
Being the right-hand side continuous, as proved in Theorem \ref{thm:EL}, we conclude the proof.
\end{proof}
\section{Examples}
In this section we provide some examples,  to highlight on one hand some expected features of the approximation problem while, on the other hand, show some unexpected behaviours even in the easier cases. First we consider the general case of power costs. Then we will focus on the quadratic cost, that is to say the mean squared error approximation functional. In this specific case we are able to restate the Euler-Lagrange equation as a first kind Fredholm equation. The latter property allows us to give some explicit examples via numerical methods.
\subsection{Power cost functionals}
Let us consider $F^{(p)}(\xi)=\frac{|\xi|^p}{p}$ for any $p \ge 2$ and the cost functional $J^{(p)}[u]=\E[\int_0^TF^{(p)}(\xi_u(t))dt]$. For such cost functionals, we are able to prove the following Proposition.
\begin{prop}\label{prop:hyp}
	Fix $p \ge 2$ and let $z \in \cL^2_{\widetilde{p}}(\Omega,\bP;[0,T])$ for some $\widetilde{p}>p$. Then $J^{(p)}$ satisfies hypotheses \textit{(H1)} to \textit{(H5+)}.
\end{prop}
\begin{proof}
Being $z \in \cL^2_{\widetilde{p}}(\Omega,\bP;[0,T])$, hypothesis (\emph{H1}) is satisfied with exponent $\widetilde{p}>p\ge 2$. Clearly, $F^{(p)}(\xi)\ge 0$ and hypothesis (\emph{H2}) is satisfied. Moreover, since $p \ge 2$, $F^{(p)}$ is twice continuously differentiable in $\xi$ with
\begin{align*}
	\der{F^{(p)}}{\xi}(\xi)=|\xi|^{p-2}\xi && \dersup{F^{(p)}}{\xi}{2}(\xi)=(p-1)|\xi|^{p-2},
\end{align*}
thus we get hypothesis (\emph{H3}). Next, there exists a constant $L>0$ such that
\begin{equation*}
	F^{(p)}(\xi)+\left|\der{F^{(p)}}{\xi}(\xi)\right|+\left|\dersup{F^{(p)}}{\xi}{2}(\xi)\right|\le L(1+|\xi|^p).
\end{equation*}
As $p<\widetilde{p}$, we achieve hypothesis (\emph{H4}) with exponent $\alpha=p$. Finally, hypothesis (\emph{H5+}) is satisfied due to the fact that $F^{(p)}(\xi)$ is strictly convex.
\end{proof}
The Euler-Lagrange equation \eqref{eq:EL} for the functional $J^{(p)}$ can be stated as
\begin{equation}\label{eq:ELp}
	\int_t^T \E\left[|\xi_{\overline{u}}(\tau)|^{p-2}\xi_{\overline{u}}(\tau) e^{A(\tau)-A(t)}\frac{G(\tau)}{G(t)}\right]d\tau=0, \ \forall t \in [0,T].
\end{equation}
To introduce a penalization on the functional $J^{(p)}$, let us consider $\Psi^{(2)}(x)=\frac{x^2}{2}$, $\cF^{(2)}:=\cF_{\Psi^{(2)}}$ and \linebreak  $J_{\delta}^{(p,2)}:=J^{(p)}+\delta \cF^{(2)}$. Its Euler-Lagrange equation \eqref{eq:ELpen} can be recast as
\begin{equation}\label{eq:ELppen}
	\overline{u}_{\delta}(t)=\frac{1}{\delta}\int_t^T \E\left[|\xi_{\overline{u}_{\delta}}(\tau)|^{p-2}\xi_{\overline{u}_{\delta}}(\tau) e^{A(\tau)-A(t)}\frac{G(\tau)}{G(t)}\right]d\tau, \ \forall t \in [0,T].
\end{equation}
Due to the nature of such equation, we speculate that, for $p>2$, iteration methods to obtain its solution could be developed. Further investigation on the topic is needed.\\
In the next subsections, we will focus on the case $p=2$, in which, as we said before, we are able to restate both equations \eqref{eq:ELp} and \eqref{eq:ELppen} in a more tractable form.
\subsection{The least mean squared error approximation: reduction to Fredholm equations}
 Indeed, let us mainly focus on the case $F^{(2)}(\xi)=\frac{\xi^2}{2}$, i.e. the least mean squared error approximation. In particular, let us denote by $J^{(2)}$ the cost functional defined as
\begin{equation*}
	J^{(2)}[u]=\E\left[\int_0^T \frac{\xi^2_{u}(t)}{2}dt\right].
\end{equation*}
By Proposition \ref{prop:hyp} we know that, if $z \in \cL^2_{p}(\Omega,\bP;[0,T])$ for any $p>2$, then hypotheses (\textit{H1}) to (\textit{H5+}) are satisfied and the Euler-Lagrange equation is given by \eqref{eq:ELp}. Actually, in this case, we can restate the equation as a Fredholm integral equation of the first kind.
\begin{prop}\label{prop:FE1}
	Let $z \in \cL_p^2(\Omega,\bP;[0,T])$ for some $p>2$. Then $\overline{u} \in L^1(0,T)$ is the unique solution of the minimization problem
	\begin{equation}\label{minprob}
		J^{(2)}[\overline{u}]=\min_{u \in L^1(0,T)}J^{(2)}[u]
	\end{equation}
	if and only if
	\begin{equation}\label{eq:EL23}
		\int_0^T k(t_0,s;a)\overline{u}(s)ds=\cZ(t_0), \ \forall t_0 \in [0,T],
	\end{equation}
where
\begin{equation}\label{eq:kdef}
	k(t,s;a)=e^{-A(t)-A(s)-\max\{t,s\}}\int_{\max\{t,s\}}^{T}e^{2A(\tau)+\tau}d\tau
\end{equation}
and
\begin{equation}\label{eq:cZdef}
	\cZ(t)=e^{-A(t)}\int_t^T \int_0^\tau e^{2A(\tau)-A(s)}\E\left[\frac{G^2(\tau)}{G(t)G(s)}z(s)\right]dsd\tau.
\end{equation}
\end{prop}
\begin{proof}
	We already know that $\overline{u}$ is the unique solution of the minimization problem \eqref{minprob} if and only if it solves
	\begin{equation}\label{eq:EL2}
		\int_{t_0}^{T}\E\left[\frac{\xi_{\overline{u}}(t)G(t)}{G(t_0)}e^{(A(t)-A(t_0))}\right]dt=0, \ \forall t_0 \in [0,T],
	\end{equation}
	that is Equation \eqref{eq:ELp}. By the explicit definition of solution map given in Equation \eqref{eq:solutionmap} we get
	\begin{equation*}
		\xi_{\overline{u}}(t)=\cS_0(z-u)(t)=G(t)e^{A(t)}\int_0^t \frac{e^{-A(s)}}{G(s)} (z(s)-\overline{u}(s))ds
	\end{equation*}
	and then Equation \eqref{eq:EL2} becomes
	\begin{equation}\label{eq:EL21}
		\int_{t_0}^{T}\E\left[\int_0^t \frac{G^2(t)}{G(t_0)G(s)}e^{2A(t)-A(t_0)-A(s)}(z(s)-\overline{u}(s))ds\right]dt=0, \ \forall t_0 \in [0,T].
	\end{equation}
	Now we want to show that we are under the hypotheses of Fubini's theorem, so to exchange the order of the inner integral and the expectation operator. Let us first rewrite
	\begin{align*}
		\int_0^t&\frac{G^2(t)}{G(t_0)G(s)}e^{2A(t)-A(t_0)-A(s)}|z(s)-\overline{u}(s)|ds=\int_0^t\frac{G^2(t)G'(s)}{G(t_0)}e^{2A(t)-A(t_0)-A(s)+s}|z(s)-u(s)|ds\\
		&\le \left(\sup_{\tau_1,\tau_2,\tau_3 \in [0,T]}e^{2A(\tau_1)-A(\tau_2)-A(\tau_3)+\tau_3}\right)\left(\sup_{s \in [0,T]}G'(s)\right)\frac{G^2(t)}{G(t_0)}\left(\int_0^T|z(s)|ds+\Norm{\overline{u}}{L^1(0,T)}\right).
	\end{align*}
	Next, we take the expectation on both sides of the previous inequality to achieve
	\begin{align*}
		&\E\left[\int_0^t\frac{G^2(t)}{G(t_0)G(s)}e^{2A(t)-A(t_0)-A(s)}|z(s)-u(s)|ds\right]\\
		&\qquad \qquad \le C\E\left[\left(\sup_{s \in [0,T]}G'(s)\right)\frac{G^2(t)}{G(t_0)}\left(\int_0^T|z(s)|ds+\Norm{\overline{u}}{L^1(0,T)}\right)\right].
	\end{align*}
	By the Cauchy-Schwartz inequality we get
	\begin{align}\label{eq:ineq21}
		\begin{split}
			&\E\left[\left(\sup_{s \in [0,T]}G'(s)\right)\frac{G^2(t)}{G(t_0)}\left(\int_0^T|z(s)|ds+\Norm{u}{L^1(0,T)}\right)\right]\\  &\qquad \qquad \le \E\left[\left(\sup_{s \in [0,T]}G'(s)\right)^2\frac{G^{4}(t)}{G^2(t_0)}\right]^{\frac{1}{2}}\E\left[\left(\int_0^T|z(s)|ds+\Norm{\overline{u}}{L^1(0,T)}\right)^2\right]^{\frac{1}{2}}.
		\end{split}
	\end{align}
	To argue with the first factor of Equation \eqref{eq:ineq21}, let us apply again the Cauchy-Schwartz inequality to obtain
	\begin{equation*}\label{eq:ineq22}
		\E\left[\left(\sup_{s \in [0,T]}G'(s)\right)^2\frac{G^{4}(t)}{G^2(t_0)}\right]\le \E\left[\left(\sup_{s \in [0,T]}G'(s)\right)^{4}\right]^{\frac{1}{2}}\E\left[\frac{G^{8}(t)}{G^{4}(t_0)}\right]^{\frac{1}{2}} \le C\E\left[\frac{G^{8}(t)}{G^{4}(t_0)}\right]^{\frac{1}{2}},
	\end{equation*}
	where we also used Lemma \ref{lem:supest}. Noticing that
	\begin{equation*}
		\frac{G^{8}(t)}{G^{4}(t_0)}=e^{8W(t)-4W(t_0)-4t+2t_0},
	\end{equation*}
	we recall that $\frac{G^{8}(t)}{G^{4}(t_0)}$ is a lognormal random variable and then $\E\left[\frac{G^{8}(t)}{G^{4}(t_0)}\right]$ is finite. Hence, the first factor of Equation \eqref{eq:ineq21} is finite.\\ 
	Concerning the second factor, it clearly holds
	\begin{equation*}
		\E\left[\left(\int_0^T|z(s)|ds+\Norm{\overline{u}}{L^1(0,T)}\right)^2\right]\le 2\left(\E\left[\left(\int_0^T|z(s)|ds\right)^2\right]+\Norm{\overline{u}}{L^1(0,T)}^2\right)<+\infty,
	\end{equation*}
	since $\cL_p^2(\Omega,\bP;[0,T]) \subset \cL_1^2([0,T];\Omega,\bP)$. Hence, we can use Fubini's theorem to rewrite Equation \eqref{eq:EL21} as
	\begin{equation*}
		\int_{t_0}^{T}\int_0^t e^{2A(t)-A(t_0)-A(s)}\E\left[\frac{G^2(t)}{G(t_0)G(s)}(z(s)-\overline{u}(s))\right]dsdt=0, \ \forall t_0 \in [0,T],
	\end{equation*}
that is equivalent to
	\begin{equation}\label{eq:EL22}
		\int_{t_0}^{T}\int_0^t e^{2A(t)-A(t_0)-A(s)}\E\left[\frac{G^2(t)}{G(t_0)G(s)}\right]\overline{u}(s)dsdt=\int_{t_0}^{T}\int_0^t e^{2A(t)-A(t_0)-A(s)}\E\left[\frac{G^2(t)}{G(t_0)G(s)}z(s)\right]dsdt, \ \forall t_0 \in [0,T].
	\end{equation}
Arguing as before, notice that $\frac{G^2(t)}{G(t_0)G(s)}$ is a lognormal random variable with
\begin{align*}
	\E\left[\log\left(\frac{G^2(t)}{G(t_0)G(s)}\right)\right]&=\E\left[2W(t)-W(t_0)-W(s)-t+\frac{t_0+s}{2}\right]=\frac{t_0+s}{2}-t\\
	{\rm Var}\left[\log\left(\frac{G^2(t)}{G(t_0)G(s)}\right)\right]&=\E\left[(2W(t)-W(t_0)-W(s))^2\right]=4t-3t_0-3s+2\min\{t_0,s\}
\end{align*} 
where for any random variable $X \in L^2(\Omega,\bP)$ we set ${\rm Var}(X)=\E[(X-\E[X])^2]$. Hence we get
\begin{equation*}\label{eq:evalkern}
		\E\left[\frac{G^2(t)}{G(t_0)G(s)}\right]=e^{t-t_0-s+\min\{t_0,s\}}=e^{t-\max\{t_0,s\}}
\end{equation*}
	and then
\begin{equation*}
	\int_{t_0}^{T}\int_0^t e^{2A(t)-A(t_0)-A(s)}\E\left[\frac{G^2(t)}{G(t_0)G(s)}\right]\overline{u}(s)dsdt=\int_{t_0}^{T}\int_0^t e^{2A(t)-A(t_0)-A(s)+t-\max\{t_0,s\}}\overline{u}(s)dsdt.
\end{equation*}
Being $\overline{u} \in L^1(0,T)$, it is clear that we can use Fubini's theorem to achieve
\begin{equation*}
	\int_{t_0}^{T}\int_0^t e^{2A(t)-A(t_0)-A(s)+t-\max\{t_0,s\}}\overline{u}(s)dsdt=\int_{0}^{T}\left(\int_{\max\{t_0,s\}}^{T} e^{2A(t)-A(t_0)-A(s)+t-\max\{t_0,s\}}dt\right)\overline{u}(s)ds.
\end{equation*}
Setting $k(t,s;a)$ and $\cZ(t)$ as in Equations \eqref{eq:kdef} and \eqref{eq:cZdef} we can rewrite Equation \eqref{eq:EL22} as \eqref{eq:EL23}, concluding the proof.
\end{proof}
An analogous result can be shown for Equation \eqref{eq:ELppen}.
\begin{prop}\label{prop:FE2}
	Let $z \in \cL_p^2(\Omega,\bP;[0,T])$ for some $p>2$. Then $\overline{u}_{\delta} \in L^1(0,T)$ is the unique solution of the minimization problem
	\begin{equation}\label{eq:minprobpen}
		J^{(2,2)}_{\delta}[\overline{u}_\delta]=\min_{u \in L^1(0,T)}J^{(2,2)}_\delta[u]
	\end{equation}
	if and only if
	\begin{equation}\label{eq:ELpen11}
		\delta\overline{u}_\delta(t_0)+\int_0^T k(t_0,s;a)\overline{u}_\delta(s)ds=\cZ(t_0), \ \forall t_0 \in [0,T],
	\end{equation}
	where $k$ and $\cZ$ are defined in Equations \eqref{eq:kdef} and \eqref{eq:cZdef}.
\end{prop}
We omit the proof since it is identical to the previous one.\\
Both Propositions \ref{prop:FE1} and \ref{prop:FE2} give us an alternative form of the Euler-Lagrange equation whose usefulness is twofold:
we can use some well-known numerical methods to exploit the solution and it can be also used to determine the existence of the solution (and actually exhibit it) under an additional hypothesis. This is the content of the next subsection.
\subsection{The least mean squared approximation: the independence case}
Now let us prove that if $z$ is independent of $W$, then we can exhibit the solution of the approximation problem.
\begin{prop}\label{prop:indquad}
	Let $z \in \cL_p^2(\Omega,\bP;[0,T])$ for some $p > 2$ be independent of the Brownian motion $W(t)$. Then the minimization problem \eqref{minprob} admits as unique solution $\overline{u}(t)=\E[z(t)]$. Moreover, it holds
	\begin{equation} \label{eq:evalJ2}
		J^{(2)}[\overline{u}]=\frac{1}{2}\int_0^T \int_0^t \int_0^t e^{2A(t)-A(s)-A(\tau)+t-\max\{s,\tau\}}{\rm Cov}(z(s),z(\tau))dsd\tau dt,
	\end{equation}
	where ${\rm Cov}$ is the covariance operator, i.e., for two random variables $X,Y \in L^2(\Omega,\bP)$, ${\rm Cov}(X,Y)=\E[(X-\E[X])(Y-\E[Y])]$.
\end{prop}
\begin{proof}
	By Proposition \ref{prop:FE1} we know that $\overline{u}$ is the unique solution of the minimization problem \eqref{minprob} if and only if it solves equation \eqref{eq:EL23}. Thus, we only have to show that $\overline{u}(t)=\E[z(t)]$ solves that equation. To do this, just observe that, being $z$ independent of $W$, by Equation \eqref{eq:cZdef} we get
	\begin{equation*}
		\cZ(t)=e^{-A(t)}\int_t^T \int_0^\tau e^{2A(\tau)-A(s)+\tau-\max\{t,s\}}\E[z(s)]dsd\tau.
	\end{equation*}
	Since $z \in \cL_p^2(\Omega,\bP; [0,T])$, we know that $\E[z(\cdot)] \in L^1(0,T)$ and then we can use Fubini's theorem in the previous equation to achieve
	\begin{equation*}
		\cZ(t)=\int_0^T \left(\int_{\max{t,s}}^T e^{-A(t)+2A(\tau)-A(s)+\tau-\max\{t,s\}} d\tau \right)\E[z(s)]ds=\int_0^T k(t,s;a)\E[z(s)]ds,
	\end{equation*}
	concluding the first part of the proof.\\
	Now let us evaluate $J^{(2)}[\overline{u}]$. We have, by the definition of $J^{(2)}$ and Equation \eqref{eq:solutionmap},
	\begin{align}\label{eq:beforeFubini}
		\begin{split}
		J^{(2)}[\overline{u}]&=\frac{1}{2}\int_0^T \E\left[G^2(t)e^{2A(t)}\left(\int_0^t \frac{e^{-A(s)}}{G(s)}(z(s)-\overline{u}(s))ds\right)^2\right]dt\\
		&=\frac{1}{2}\int_0^T \E\left[\int_0^t\int_0^t \frac{G^2(t)e^{2A(t)-A(s)-A(\tau)}}{G(s)G(\tau)}(z(s)-\overline{u}(s))(z(\tau)-\overline{u}(\tau))dsd\tau\right]dt.
		\end{split}
	\end{align}
	Now we want to show that we are under the hypotheses of Fubini's theorem so that we can exchange the inner double integral with the expectation operator. To do this, notice that
	\begin{align*}
		\int_0^t\int_0^t &\frac{G^2(t)e^{2A(t)-A(s)-A(\tau)}}{G(s)G(\tau)}|z(s)-\overline{u}(s)||z(\tau)-\overline{u}(\tau)|dsd\tau\\&\le G^2(t)\left(\sup_{t,t_0,s \in [0,T]}e^{2A(t)-A(s)-A(\tau)+s+\tau}\right)\left(\sup_{s \in [0,T]}G'(s)\right)\left(\sup_{s \in [0,T]}G'(\tau)\right)\\&\qquad \times \int_0^t\int_0^t|z(s)-\overline{u}(s)||z(\tau)-\overline{u}(\tau)|dsd\tau.
	\end{align*}
	Taking the expectation on both sides and using the fact that $z$ is independent of $W$ we have
	\begin{align*}
		\E\left[\int_0^t\int_0^t \right.&\left.\frac{G^2(t)e^{2A(t)-A(s)-A(\tau)}}{G(s)G(\tau)}|z(s)-\overline{u}(s)||z(\tau)-\overline{u}(\tau)|dsd\tau\right]\\&\le \E\left[G^2(t)\left(\sup_{t,t_0,s \in [0,T]}e^{2A(t)-A(s)-A(\tau)+s+\tau}\right)\left(\sup_{s \in [0,T]}G'(s)\right)\left(\sup_{s \in [0,T]}G'(\tau)\right)\right.\\&\left.\qquad \times \int_0^t\int_0^t|z(s)-\overline{u}(s)||z(\tau)-\overline{u}(\tau)|dsd\tau\right]\\
		&\le C\E\left[G^2(t)\left(\sup_{s \in [0,T]}G'(s)\right)\left(\sup_{s \in [0,T]}G'(\tau)\right)\right]\E\left[\int_0^t\int_0^t|z(s)-\overline{u}(s)||z(\tau)-\overline{u}(\tau)|dsd\tau\right].
	\end{align*}
	Arguing exactly as in the proof of Proposition \ref{prop:FE1}, we have
	\begin{equation*}
		\E\left[G^2(t)\left(\sup_{s \in [0,T]}G'(s)\right)\left(\sup_{s \in [0,T]}G'(\tau)\right)\right]\le C,
	\end{equation*}
while, on the other hand
	\begin{align*}
		\E\left[\int_0^t\int_0^t|z(s)-\overline{u}(s)||z(\tau)-\overline{u}(\tau)|dsd\tau\right]&=\int_0^t\int_0^t\E[|z(s)-\overline{u}(s)||z(\tau)-\overline{u}(\tau)|]dsd\tau \\&\le \int_0^t\int_0^t\E[|z(s)-\overline{u}(s)|^2]^{\frac{1}{2}}\E[|z(\tau)-\overline{u}(\tau)|^2]^{\frac{1}{2}}dsd\tau\\
		&= \int_0^t\E[|z(s)-\overline{u}(s)|^2]^{\frac{1}{2}}ds\int_0^t\E[|z(\tau)-\overline{u}(\tau)|^2]^{\frac{1}{2}}d\tau\\
		&\le \left(\int_0^T\E[|z(s)-\overline{u}(s)|^2]^{\frac{1}{2}}ds\right)^2\\
		&\le T\int_0^T\E[|z(s)-\overline{u}(s)|^2]ds,
	\end{align*}
	where we used the Cauchy-Schwartz inequality and Jensen's inequality. Now let us observe that 
	\begin{equation*}
		\E[|z(s)-\overline{u}(s)|^2]=\E[|z(s)|^2]-\E[z(s)]^2\le \E[|z(s)|^2],
	\end{equation*}
	and then, being $z \in \cL_p^2(\Omega,\bP; [0,T])$, we know that
	\begin{equation*}
		T\int_0^T\E[|z(s)-\overline{u}(s)|^2]ds\le C.
	\end{equation*} 
Hence we can use Fubini's theorem and the fact that $z$ is independent of $W$ in Equation \eqref{eq:beforeFubini} to conclude the proof.
\end{proof}
\begin{rmk}\label{rmk:exp}
	Last Proposition agrees, in some sense, with the classical idea that the expected value should minimize the mean squared error under the hypothesis that $z$ is independent of $W$. However, we cannot remove this hypothesis, as we will see in the next subsection. 
\end{rmk}
Let us use the previous Proposition to provide an example in which we already know that the minimizer exists (and we know its exact form) and we can numerically solve the penalized equations. Precisely, let us set $z(t)$ as a geometric Brownian motion independent of $W$, so that $\overline{u}(t)=\E[z(t)]\equiv 1$. Let $a(t)\equiv -1$, and then $A(t)=-t$, and $T=1$. In this case we have
\begin{equation}\label{kernel}
	k(t,s;-1)=e^{-|t-s|}-e^{-1+\min\{t,s\}}
\end{equation}
and
\begin{equation}\label{zeta}
	\cZ(t)=e^{-1-t}(-e+e^t-3e^{2t}+2e^{1+t}+e^{2t}t).
\end{equation}
For any $\delta>0$, let $\overline{u}_\delta$ be the solution of Equation \eqref{eq:ELpen11}. To obtain a numerical evaluation of $\overline{u}_\delta$ for some fixed $\delta>0$ we used Nystr\"om method (see \cite{atkinson1997numerical}). Since we expect an highly oscillatory behaviour for small values of $\delta$, we need to determine a big number of nodes. To do this, we adopted a composite Newton-Cotes quadrature formula based on a $7$-th order interpolating polynomial on equispaced nodes. Precisely, we divided the interval $[0,1]$ in $N$ subintervals and, on each interval, we applied the closed Newton-Cotes formula (on equispaced nodes) with weights:
\begin{equation*}
	w=\frac{1}{840N}(41, 236, 27, 272, 27, 236, 41).
\end{equation*}
To be sure to avoid Runge's phenomenon, one could also reduce the order of the interpolation while increasing $N$. Such solutions are visualized in Figure \ref{fig:figure1}.
\begin{figure}[htb]
	\centering
	\includegraphics[width=0.7\linewidth]{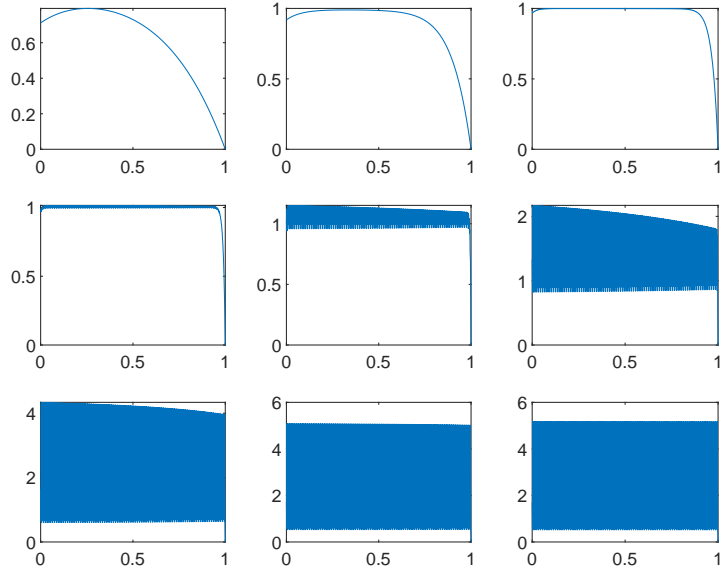}
	\caption{Numerical solutions of Equation \eqref{eq:ELpen11} with $T=1$, $\mathcal{Z}(t)$ given in Equation \eqref{zeta} and $k(t,s;a)$ given in Equation \eqref{kernel}, for different values of $\delta$. Precisely, reading left-to-right top-to-bottom we have $\overline{u}_\delta$ for $\delta=10^{-n}$ with $n=1,\dots,9$. $N$ is fixed to $100$, so that we have $601$ nodes for each $\overline{u}_\delta$.} 
	\label{fig:figure1}
\end{figure}
Evidently, $\overline{u}_\delta$ does not converge to $1$ as $\delta \to 0$. Let us now denote $\overline{u}_n=\overline{u}_{10^{-n}}$. To show  that $J^{(2)}[\overline{u}_n] \to J^{(2)}[\overline{u}]$, let us first evaluate $J^{(2)}[\overline{u}]$. This can be done by observing that
\begin{equation*}
	{\rm Cov}(z(s),z(\tau))=e^{\min\{\tau,s\}}-1
\end{equation*}
and then, by Equation \eqref{eq:evalJ2},
\begin{equation*}
	J^{(2)}[\overline{u}]=\frac{1}{2}\int_0^1\int_0^t\int_0^t e^{-t+\min\{\tau,s\}}(e^{\min\{\tau,s\}}-1)dsd\tau dt=\frac{e^2-7}{4e}\approx 0.0357814.
\end{equation*}
On the other hand, to evaluate $J^{(2)}[\overline{u}_n]$, we adopted a numerical method based on a Monte-Carlo approach. Precisely we simulated a skeleton of $6N+1$ nodes $(\xi_i)_{0 \le i \le 6N}$ for the process $\xi:=\xi_{\overline{u}_n}$. To do this, first we simulated a skeleton $(z_i)_{0 \le i \le 6N}$ of $6N+1$ nodes for $z$ as
\begin{equation*}
	\begin{cases}
		z_0=1,\\
		z_i=z_{i-1}e^{\frac{\widetilde{\zeta}_i}{\sqrt{6N}}-\frac{1}{12N}}, & i=1,\dots,6N,
	\end{cases}
\end{equation*}
where $\widetilde{\zeta}_i \sim \mathcal{N}(0,1)$ with $\widetilde{\zeta}_i$ independent of $\widetilde{\zeta}_j$ for each $i \not = j$. Once this is done, $(\xi_i)_{0 \le i \le 6N}$ can be obtained by using an Euler scheme (see \cite{asmussen2007stochastic}):
\begin{equation*}
	\begin{cases}
		\xi_0=0\\
		\displaystyle \xi_i=\xi_{i-1}+\frac{-\xi_{i-1}+z_{i-1}-\overline{u}_{n}(t_{i-1})}{6N}+\frac{\zeta_i\xi_{i-1}}{\sqrt{6N}}, & i=1,\dots,6N,
	\end{cases}
\end{equation*}
where $t_{i}=\frac{i}{6N}$, $\overline{u}_{n}(t_{i-1})$ has been obtained previously via Nystr\"om method and $\zeta_i \sim \cN(0,1)$ with $\zeta_i$ independent of $\zeta_j$ for $i \not = j$. The value $\frac{1}{2}\int_0^T \xi^2(t)dt$ is then approximated by a quadrature formula and $J^{(2)}[\overline{u}_n]$ by repeating the procedure for a fixed number $N_{\rm traj}$ of trajectories and then taking the average. While, on one hand, the convergence $J^{(2)}[\overline{u}_n] \to J^{(2)}[\overline{u}]$ is justified by Theorem \ref{thm:minseq}, on the other hand the stochastic differential equation could be \textit{stiff} due to the highly oscillatory behaviour of $\overline{u}_n$ and the Euler scheme could fail to catch $\xi$. The estimated values of $J^{(2)}[\overline{u}_n]$ for $n=1,\dots,5$ are given in Table \ref{tbl1}. 
\begin{table}[htb]
	\begin{tabular}{@{}lllll
			>{\columncolor[HTML]{C0C0C0}}l @{}}
		\toprule
		$J^{(2)}[\overline{u}_1]$ & $J^{(2)}[\overline{u}_2]$ & $J^{(2)}[\overline{u}_3]$ & $J^{(2)}[\overline{u}_4]$ & $J^{(2)}[\overline{u}_5]$ & $J^{(2)}[\overline{u}]$ \\ \midrule
		$0.0462$                  & $0.0363$                  & $0.0361$                  & $0.0362$                  & $0.0358$                  & $0.0358$ \\            
		\bottomrule   
	\end{tabular}
	\caption{Numerically estimated values of $J^{(2)}[\overline{u}_n]$ for $n=1,\dots,5$, in comparison with $J^{(2)}[\overline{u}]$. $N$ is fixed to $100$, while $N_{\rm traj}=100000$.}
	\label{tbl1}
\end{table}\\
To show a numerical evidence that $\overline{u}_n \rightharpoonup \overline{u}$, we also numerically evaluated $\int_0^1 t^j\overline{u}_n(t)dt$ for different values of $j$ and $n$ and we compared it with $\int_0^1 t^jdt=1/j$ in Table \ref{tbl2}. 
\begin{table}[htb]
	\begin{tabular}{@{}llllllllll
			>{\columncolor[HTML]{C0C0C0}}l @{}}
		\toprule
		& $n=1$    & $n=2$    & $n=3$    & $n=4$    & $n=5$    & $n=6$    & $n=7$    & $n=8$    & $n=9$    & $\int_0^1 x^jdx$ \\ \midrule
		$j=0$ & $0.6141$ & $0.8878$ & $0.9670$ & $0.9899$ & $0.9968$ & $0.9988$ & $0.9991$ & $0.9992$ & $0.9992$ & $1$              \\
		$j=1$ & $0.2535$ & $0.4067$ & $0.4690$ & $0.4901$ & $0.4968$ & $0.4988$ & $0.4991$ & $0.4992$ & $0.4992$ & $1/2$            \\
		$j=2$ & $0.1431$ & $0.2495$ & $0.3034$ & $0.3235$ & $0.3301$ & $0.3321$ & $0.3325$ & $0.3325$ & $0.3325$ & $1/3$            \\
		$j=3$ & $0.0930$ & $0.1734$ & $0.2211$ & $0.2403$ & $0.2468$ & $0.2488$ & $0.2491$ & $0.2492$ & $0.2492$ & $1/4$            \\ \bottomrule		
	\end{tabular}
	\caption{Numerically estimated values of $\int_0^1 t^j\overline{u}_n(t)dt$ for $n=1,\dots,9$ and $j=0,1,2,3$, in comparison with $\int_0^1 t^jdt=1/j$. $N$ is fixed to $100$ and the values of the integrals are obtained by using the same quadrature formula as applied before to determine $\overline{u}_n$.}
	\label{tbl2}
\end{table}\\
With this example, we want to highlight the fact that even if the solution of the minimizing problem \eqref{minprob} is known and quite regular, the solution of the penalized problem converge towards them only weakly. However, this is a problem only in the case one wants to approximate the actual minimizer $\overline{u}$. Indeed, usually one is interested in properties of the approximating process $\cS_{X_0}\overline{u}$, that, despite the weak convergence of $\overline{u}_\delta$ towards $\overline{u}$, is in the overall approximated well enough by $\cS_{X_0}\overline{u}_\delta$, as shown in Theorem \ref{thm:strongconv}.\\
As already stated in Remark \ref{rmk:exp}, Proposition \ref{prop:indquad} seems to suggest that the expected value should be, in some sense, the minimizer of the mean squared error. However, as we will see in the following example, this is not necessarily true if we suppose that $z$ and $W$ are dependent.
\subsection{The least mean squared error approximation: a dependence case}
Now let us consider a different example. Let $z=G$, $a\equiv 1$, so that $A(t)=t$, and $T=1$. First of all, let us observe that, since $z(t)=f(t,W(t))$ for some function $f$, we cannot use Proposition \ref{prop:indquad}. Thus, let us first determine (at least numerically) the solutions $\overline{u}_\delta$ of the penalized problem \eqref{eq:minprobpen}. According to Equation \eqref{eq:kdef} we have
\begin{equation}\label{eq:kern2}
	k(t,s;1)=\frac{1}{3}\left(e^{3-\min\{t,s\}-2\max\{t,s\}}-e^{|t-s|}\right).
\end{equation}
Concerning $\cZ(t)$, starting from Equation \eqref{eq:cZdef}, it holds
\begin{equation}\label{eq:cZ2}
	\cZ(t)=\int_t^1 \int_0^\tau e^{2\tau-t-s}\E\left[\frac{G^2(\tau)}{G(t)}\right]dsd\tau.
\end{equation}
In particular, $\frac{G^2(\tau)}{G(t)}=e^{2W(\tau)-W(t)-\tau+\frac{t}{2}}$ is a lognormal random variable with
\begin{align*}
	\E\left[\log\left(\frac{G^2(\tau)}{G(t)}\right)\right]&=\E\left[2W(\tau)-W(t)-\tau+\frac{t}{2}\right]=\frac{t}{2}-\tau\\
	{\rm Var}\left[\log\left(\frac{G^2(\tau)}{G(t)}\right)\right]&=\E[(2W(\tau)-W(t))^2]=4\tau-3t
\end{align*}
hence
\begin{equation*}
	\E\left[\frac{G^2(\tau)}{G(t)}\right]=e^{\tau-t}.
\end{equation*}
Thus Equation \eqref{eq:cZ2} becomes
\begin{equation}\label{eq:zeta2}
	\cZ(t)=\int_t^1 \int_0^\tau e^{3\tau-2t-s}dsd\tau=\frac{1}{6}(3-2e^t+e^{2-2t}(2e-3)).
\end{equation}
As before, let us exploit some numerical solutions $\overline{u}_{\delta}$ of Equation \eqref{eq:ELpen11} (with $k$ and $\cZ$ given in Equations \eqref{eq:kern2} and \eqref{eq:zeta2}) by using Nystr\"om's method, as shown in Figure \ref{fig:figure2}. From now on let us denote $\overline{u}_{n}:=\overline{u}_{10^{-n}}$.
\begin{figure}[htb]
	\centering
	\includegraphics[width=0.7\linewidth]{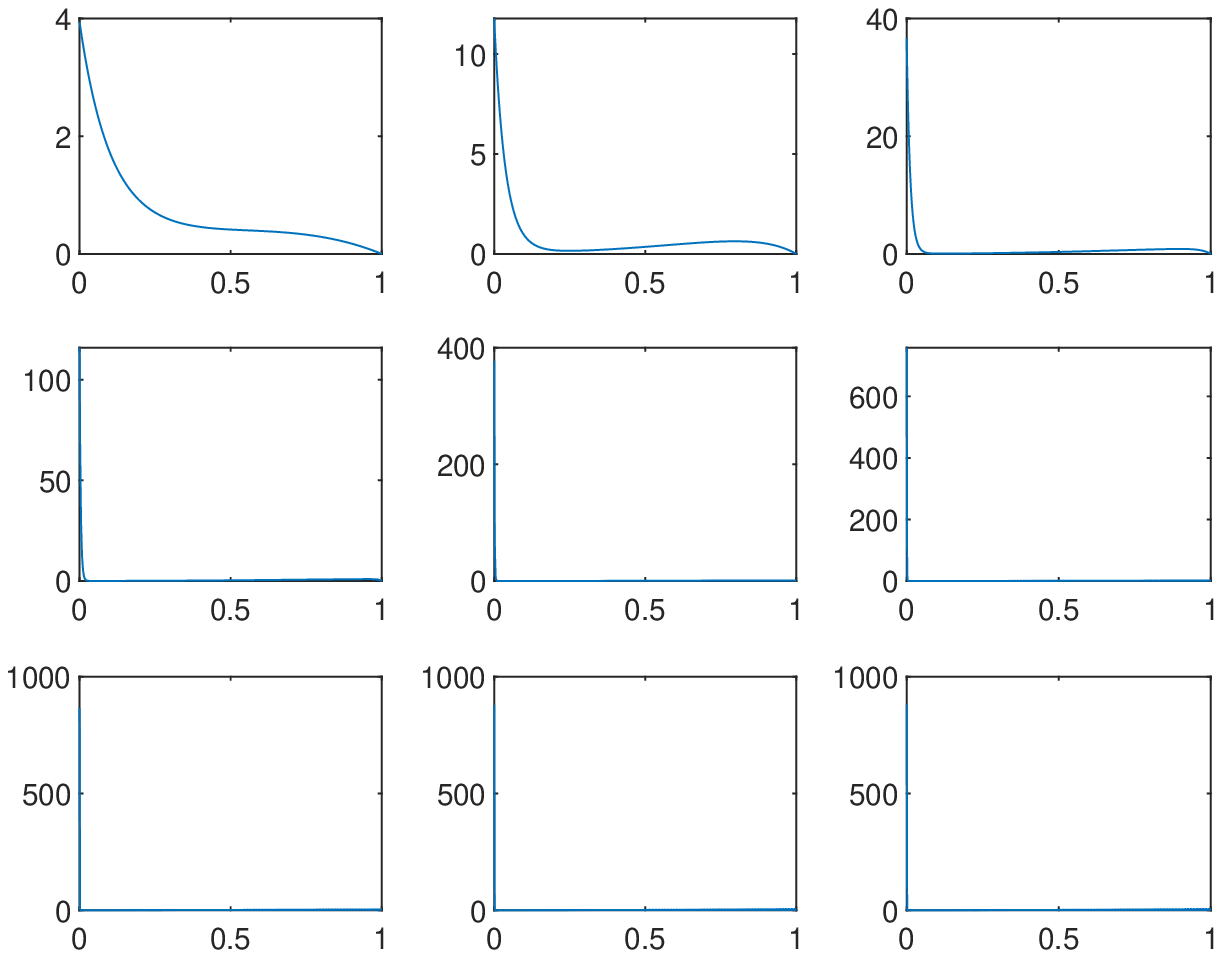}
	\caption{Numerical solutions of Equation \eqref{eq:ELpen11} with $T=1$, $\mathcal{Z}(t)$ given in Equation \eqref{eq:zeta2} and $k(t,s;a)$ given in Equation \eqref{eq:kern2}, for different values of $\delta$. Precisely, reading left-to-right top-to-bottom we have $\overline{u}_\delta$ for $\delta=10^{-n}$ with $n=1,\dots,9$. $N$ is fixed to $100$, so that we have $601$ nodes for each $\overline{u}_\delta$.} 
	\label{fig:figure2}
\end{figure}
We do not know if Equation \eqref{eq:EL23} admits a solution. To have a qualitative idea on whether a solution of Equation \eqref{eq:EL23} exists or not, we could evaluate $\Norm{\overline{u}_n}{L^p(0,T)}$ for some $p>1$, as done in Table \ref{tbl3}.
\begin{table}[htb]
	\begin{tabular}{@{}llllllllll@{}}
		\toprule
		& $n=1$    & $n=2$    & $n=3$    & $n=4$     & $n=5$     & $n=6$      & $n=7$      & $n=8$      & $n=9$      \\ \midrule
		$p=2$    & $1.0492$ & $2.7499$ & $8.1839$ & $24.5822$ & $86.3839$ & $282.1855$ & $367.4504$ & $378.6386$ & $379.7967$ \\
		$p=1.5$  & $0.7900$ & $1.2318$ & $2.0335$ & $3.3589$  & $5.8571$  & $10.7897$  & $12.8737$  & $13.1871$  & $13.2218$  \\
		$p=1.25$ & $0.7262$ & $0.9220$ & $1.1764$ & $1.4681$  & $1.8501$  & $2.4278$   & $2.6948$   & $2.7466$   & $2.7529$   \\
		$p=1.1$  & $0.7057$ & $0.8139$ & $0.9178$ & $1.0014$  & $1.0824$  & $1.1826$   & $1.2347$   & $1.2472$   & $1.2489$   \\
		$p=1.01$ & $0.6996$ & $0.7701$ & $0.8183$ & $0.8402$  & $0.8508$  & $0.8584$   & $0.8622$   & $0.8632$   & $0.8633$   \\ \bottomrule
	\end{tabular}
	\caption{Numerically estimated values of $\Norm{\overline{u}_n(t)}{L^p(0,1)}^p$ for $n=1,\dots,9$ and $p=2, 1.5, 1.25, 1.1, 1.01$. $N$ is fixed to $100$ and the values of the integrals are obtained by using the same quadrature formula as before to determine $\overline{u}_n$.}
	\label{tbl3}
\end{table}
From Table \ref{tbl3}, we expect that $\Norm{\overline{u}_n}{L^p(0,1)}\le C$ for some suitable choice of $p>1$ (a good choice could be $p=1.01$, but, for a big value of $C$, also $p=2$ seems to work). This numerical evidence lets us conjecture that $\overline{u}_n$ is uniformly bounded in $L^p(0,1)$ for some $p>1$ and then, by Theorem \ref{thm:converg}, that a solution $\overline{u}$ of the minimization problem \eqref{minprob} exists. Since we can suppose $\Norm{\overline{u}_n}{L^2(0,1)}\le C$, let us conjecture that $\overline{u} \in L^2(0,1)$.\\
With this idea in mind, let us evaluate numerically the solution of Equation \eqref{eq:EL23}. To do this, we cannot use Nystr\"om's method, as it is well known that for Fredholm integral equations of the first kind the matrix obtained with the quadrature formula is very ill-conditioned. Hence, we have to use a different method. Precisely, we use a Gal\"erkin-type method as follows (see \cite[Section $6.3$]{wing1991primer}). Let $P_n(t)$ be the $n$-th degree Legendre polynomial and define $Q_n(t)=P_n(2t-1)$. Thus $\{Q_n(t)\}_{n \in \N}$ constitute an orthogonal system in $L^2(0,1)$ equipped with the usual scalar product $\langle \cdot, \cdot \rangle$, i.e.
\begin{equation*}
	\langle f,g\rangle=\int_0^1 f(t)g(t)dt, \ \forall f,g \in L^2(0,1).
\end{equation*}
Thus, $\overline{u}=\sum_{i=0}^{+\infty}\langle \overline{u}, Q_i\rangle Q_i$ and $\cZ=\sum_{i=0}^{+\infty}\langle \cZ, Q_i\rangle Q_i$. Moreover, for fixed $s \in [0,1]$, it holds $k(\cdot,s;1)=\sum_{i=0}^{+\infty}\langle k(\cdot,s;1), Q_i \rangle Q_i$. Rewriting Equation \eqref{eq:EL23} by using the series decomposition of $\overline{u}$ we have
\begin{equation*}
	\sum_{i=0}^{+\infty} \langle \overline{u}, Q_i\rangle \int_0^1 k(t,s;1)Q_i(s)ds=\cZ(t).
\end{equation*}
Then, using the decomposition of $\cZ$ and $k(\cdot,s;1)$, we get
\begin{equation*}
	\left(\sum_{i,j=0}^{+\infty} \langle \overline{u}, Q_i\rangle \int_0^1 \langle k(\cdot ,s;1), Q_j\rangle Q_i(s)ds\right)Q_j(t)=\sum_{j=0}^{+\infty}\langle \cZ, Q_j\rangle Q_j(t),
\end{equation*}
that is to say
\begin{equation*}
	\left(\sum_{i,j=0}^{+\infty} \langle \overline{u}, Q_i\rangle \int_0^1 \int_0^1 k(\tau ,s;1) Q_j(\tau) Q_i(s)d \tau ds\right)Q_j(t)=\sum_{j=1}^{+\infty}\langle \cZ, Q_j\rangle Q_j(t).
\end{equation*}
Let us fix $m \in \N$ and consider $\cZ_m=\sum_{i=0}^{m}\langle \cZ, Q_i\rangle Q_i$, that is a finite-dimensional approximation of $\cZ$. Let also $\overline{u}^{(m)}=\sum_{i=0}^{m}\langle \overline{u}, Q_i\rangle Q_i$ be a finite-dimensional approximation of $\overline{u}$ satisfying equation \eqref{eq:EL23} with $\cZ_m$ in place of $\cZ$. Then, if we reduce the problem to finding the finite-dimensional approximation $\overline{u}^{(m)}$, it is equivalent to the problem of solving the following linear system
\begin{equation*}
	\mathbf{K}\mathbf{u}=\mathbf{z},
\end{equation*}
where $\mathbf{K}_{i,j}=\int_0^1 \int_0^1 k(\tau ,s;1) Q_j(\tau) Q_i(s)d \tau ds$, $\mathbf{u}_i=\langle \overline{u}, Q_i\rangle$ and $\mathbf{z}_i=\langle \cZ, Q_i\rangle$, for $i,j=0,\dots,m$. The approximation $\overline{u}^{(m)}$ is shown in Figure \ref{fig:figure3}. From now on we will refer to $\overline{u}^{(m)}$ directly as $\overline{u}$.
\begin{figure}[htb]
	\centering
	\includegraphics[width=0.7\linewidth]{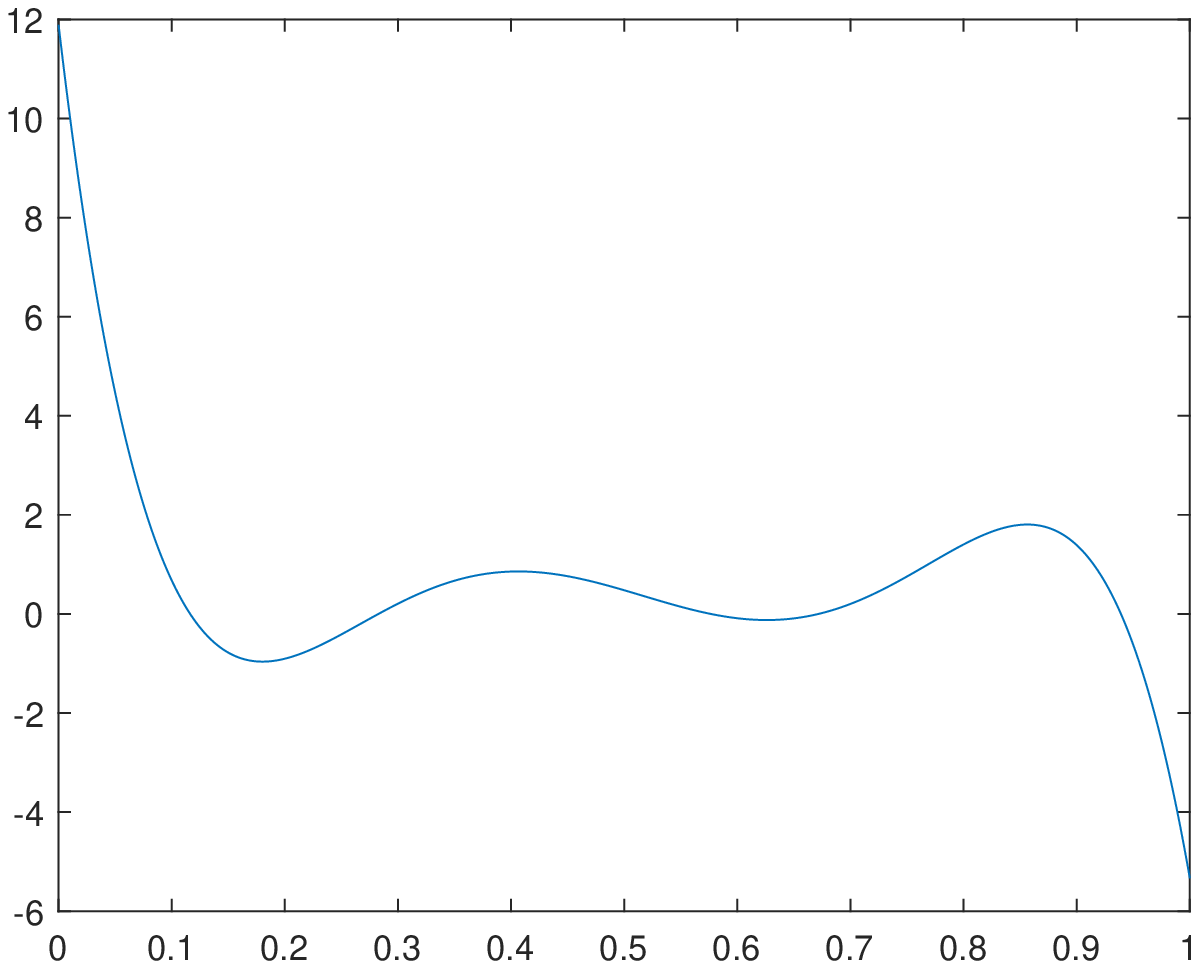}
	\caption{The approximation $\overline{u}^{(m)}$, with $m=5$, of the solution of Equation \eqref{eq:EL23} with $k(t,s;a)$ given in Equation \eqref{eq:kern2} and $\cZ$ given in Equation \eqref{eq:zeta2}.}
	\label{fig:figure3}
\end{figure}
Just looking at the figure, it seems that $\overline{u} \not \equiv 1 \equiv\E[z(\cdot)]$. Indeed, in this case, $\E[z(\cdot)]$ is not a solution of the minimization problem \eqref{minprob}, since
\begin{equation*}
	\int_0^1 k(t,s;1)ds=-\frac{1}{6}e^{-3t}(e^3 - 4 e^{3 t} + 2 e^{4 t} - 2 e^{3 + t} + 
	3 e^{1 + 2 t})\not \equiv \cZ(t),
\end{equation*}
so that $\E[z(\cdot)]$ does not solve Equation \eqref{eq:EL23}. Actually, the expected value seems to be quite far from the optimal approximation. This can be observed by evaluating $J^{(2)}[\overline{u}_n]$, $J^{(2)}[\overline{u}]$ and $J^{(2)}[\E[z(\cdot)]]$. As before, to evaluate $J^{(2)}[\overline{u}_n]$ we use the Monte-Carlo approach presented in the previous subsection, with the following set of nodes:
\begin{equation*}
	\begin{cases}
		z_0=1,\\
		z_i=z_{i-1}e^{\frac{\zeta_i}{\sqrt{6N}}-\frac{1}{12N}}, & i=1,\dots,6N,
	\end{cases}
\end{equation*}
and
\begin{equation*}
	\begin{cases}
		\xi_0=0\\
		\displaystyle \xi_i=\xi_{i-1}+\frac{\xi_{i-1}+z_{i-1}-\overline{u}_{n}(t_{i-1})}{6N}+\frac{\zeta_i\xi_{i-1}}{\sqrt{6N}}, & i=1,\dots,6N,
	\end{cases}
\end{equation*}
where $t_{i}=\frac{i}{6N}$, $\overline{u}_{n}(t_{i-1})$ has been obtained previously via Nystr\"om's method and $\zeta_i \sim \cN(0,1)$ with $\zeta_i$ independent of $\zeta_j$ for $i \not = j$. To evaluate $J^{(2)}[\overline{u}]$ and $J^{(2)}[\E[z(\cdot)]]$ we can use the same exact scheme substituting respectively $\overline{u}$ (obtained by Gal\"erkin's method) and $1$ in place of $\overline{u}_n$. Again, the evaluations of $J^{(2)}[\overline{u}_n]$ for big values of $n$ are not reliable due to the stiffness of the underlying problem. The results are exposed in Table \ref{tbl4}: here it is evident that the expected value is not the optimal approximation.
\begin{table}[htb]
	\begin{tabular}{@{}lll
			>{\columncolor[HTML]{C0C0C0}}l l @{}}
		\toprule
		$J^{(2)}[\overline{u}_1]$ & $J^{(2)}[\overline{u}_2]$ & $J^{(2)}[\overline{u}_3]$ & $J^{(2)}[\overline{u}]$ & $J^{(2)}[\E[z(\cdot)]]$\\ \midrule
		0.0836                    & 0.0572                    & 0.0510                    & 0.0577                  & 0.1505\\ \bottomrule
	\end{tabular}
	\caption{Numerically estimated values of $J^{(2)}[\overline{u}_n]$ for $n=1,\dots,3$, in comparison with $J^{(2)}[\overline{u}]$ and $J^{(2)}[\E[z(\cdot)]]$. $N$ is fixed to $100$, while $N_{\rm traj}=100000$. Consider that, since we are using a Monte-Carlo method, all the values in the table are subject to fluctuations, hence results that are \textit{near} to the best error $J^{(2)}[\overline{u}]$ are still admissible, despite being inferior to it. In some sense, this phenomenon, that is expected due to the stochastic approach used, also evidence the speed of convergence of $J^{(2)}[\overline{u}_n]$ to the best error.}
	\label{tbl4}
\end{table}
\begin{rmk}
In the case of the additive noise, in \cite{ascione2020optimal} it has been shown that the expected value is always the optimal approximation with respect to the quadratic cost (even if $z$ and $W$ are not independent). Clearly, the presence of the multiplicative noise has a crucial effect in this sense.
\end{rmk}
Again, to have another numerical evidence of the fact that $\overline{u}_n \rightharpoonup \overline{u}$, we compare $\int_0^1 t^j \overline{u}_n(t)dt$ with $\int_0^1 t^j \overline{u}(t)dt$ for $j=0,1,2,3$ in Table \ref{tbl5}.
\begin{table}[htb]
	\begin{tabular}{@{}llllllllll
			>{\columncolor[HTML]{C0C0C0}}l @{}}
		\toprule
		& $n=1$    & $n=2$    & $n=3$    & $n=4$    & $n=5$    & $n=6$    & $n=7$    & $n=8$    & $n=9$    & $\int_0^1 t^j\overline{u}(t)dt$ \\ \midrule
		$j=0$ & $0.6992$ & $0.7661$ & $0.8093$ & $0.8262$ & $0.8318$ & $0.8333$ & $0.8336$ & $0.8336$ & $0.8336$ & $0.8336$                        \\
		$j=1$ & $0.1838$ & $0.2256$ & $0.2657$ & $0.2822$ & $0.2877$ & $0.2893$ & $0.2896$ & $0.2896$ & $0.2896$ & $0.2896$                        \\
		$j=2$ & $0.0942$ & $0.1510$ & $0.1981$ & $0.2175$ & $0.2240$ & $0.2260$ & $0.2263$ & $0.2264$ & $0.2264$ & $0.2264$                        \\
		$j=3$ & $0.0606$ & $0.1139$ & $0.1580$ & $0.1768$ & $0.1833$ & $0.1852$ & $0.1856$ & $0.1856$ & $0.1856$ & $0.1856$                        \\ \bottomrule
	\end{tabular}
	\caption{Numerically estimated values of $\int_0^1 t^j\overline{u}_n(t)dt$ for $n=1,\dots,9$ and $j=0,1,2,3$, in comparison with $\int_0^1 t^j\overline{u}(t)dt$. $N$ is fixed to $100$ and the values of the integrals are obtained by using the same quadrature formula as applied before to determine $\overline{u}_n$.}
	\label{tbl5}
\end{table}
\begin{rmk}
	Let us emphasize that one must pay attention to the choice of the numerical method to solve Equation \eqref{eq:ELpen11}. Indeed, one cannot exclude \textit{a priori} an highly oscillatory behaviour of the solution of \eqref{eq:ELpen11}, as show by our first example. Thus, if a Gal\"erkin-type method is adopted, then the family of independent functions on $[0,T]$ should be chosen according to the expected behaviour of the solutions.
\end{rmk}
\appendix
\section{Lebesgue points}\label{AppendixLeb}
Let us recall the definition of Lebesgue point for a function $f \in L^1(0,T)$.
\begin{defn}
	We say that $t \in (0,T)$ is a Lebesgue point for $f$ if
	\begin{equation*}
		\lim_{\varepsilon \to 0^+}\frac{1}{\varepsilon}\int_{t-\frac{\varepsilon}{2}}^{t+\frac{\varepsilon}{2}}f(\tau)d\tau=f(t).
	\end{equation*}
We denote by $E_f$ the set of Lebesgue points of $f$.
\end{defn} 
By Lebesgue's differentiation theorem (see \cite[Section $1.7$, Theorem $1$]{evans2015measure}) it is well known that \linebreak $|[0,T]\setminus E_f|=0$. Let us recall, in particular, the following convergence result (see \cite[Section~$1.7$, Corollary~$2$]{evans2015measure}).
\begin{prop}
	Let $f \in L^p(0,T)$ for some $1 \le p <\infty$ and $t$ be a Lebesgue point for $f$. Let $\cI(t)$ be the family of all closed intervals in $[0,T]$ containing $t$. Then
	\begin{equation*}
		\lim_{\substack{{\rm diam}(I) \to 0 \\ I \in \cI(t)}}\frac{1}{|I|}\int_{I}|f(\tau)-f(t)|^pd\tau=0.
	\end{equation*}
\end{prop}
We can use last statement to prove the following result.
\begin{prop}\label{prop:lebp}
	Consider $1 \le p \le \infty$ and let $f \in L^p(0,T)$ and $g \in L^q(0,T)$ where $\frac{1}{p}+\frac{1}{q}=1$ and define $h=fg \in L^1(0,T)$. Consider versions of $f$ and $g$ that are everywhere finite. Then $E_f \cap E_g \subseteq E_h$. Moreover, if $p=1$ and $g \in C([0,T])$, then $E_f\subseteq E_h$.
\end{prop}
\begin{proof}
	Consider $t \in E_f \cap E_g$ and observe that, by H\"older's inequality, it holds, for any closed interval $I$ containing $t$,
	\begin{align}\label{Lebest}
		\begin{split}
		\left|\frac{1}{|I|}\int_I h(\tau)d\tau-h(t)\right|&\le \frac{1}{|I|}\int_I|h(\tau)-h(t)|d\tau\\
		&\le \frac{1}{|I|}\int_I|f(\tau)||g(\tau)-g(t)|d\tau+\frac{1}{|I|}\int_T|g(t)||f(\tau)-f(t)|d\tau\\
		&\le \left(\frac{1}{|I|}\int_I |f(\tau)|^pd\tau\right)^{\frac{1}{p}}\left(\frac{1}{|I|}\int_I|g(\tau)-g(t)|^qd\tau\right)^\frac{1}{q} \\&\qquad +|g(t)|\left(\frac{1}{|I|}\int_I|f(\tau)-f(t)|^pd\tau\right)^{\frac{1}{p}}.
	\end{split}
	\end{align}
	Being $|g(t)|<\infty$ and $|f(t)|<\infty$, taking the limit as ${\rm diam}(I)\to 0$ in Equation \eqref{Lebest}, it holds 
	\begin{align*}
		\begin{split}
			\lim_{\substack{\rm{diam}(I) \to 0 \\ I \in \cI(t)}}\left|\frac{1}{|I|}\int_I h(\tau)d\tau-h(t)\right|=0,
		\end{split}
	\end{align*}
so that $t \in E_h$.\\
Concerning the second part of the statement, just observe that if $g$ is continuous, $E_g=[0,T]$ by the integral mean value theorem.
\end{proof}
\section{Lower semicontinuity of the functional $\cF$: Proof of Proposition \ref{prop:semicontF}}\label{App:B}
\begin{proof}
First, let us show that $\cF_\Psi$ is lower semicontinuous in any $y \in L^1(0,T)$. Hence, let us consider $y_n \to y$ in $L^1$: we want to show that
\begin{equation*}
	\liminf_{n \to +\infty}\cF_\Psi[y_n]\ge \cF_\Psi[y].
\end{equation*}
Without loss of generality, we can consider a non-relabelled subsequence $y_n$ that realizes the limit inferior. Let us first consider the case $\cF_\Psi[y]<+\infty$. Let $y_{n_k}$ be a subsequence of $y_n$ that converges almost everywhere to $y$. By Egorov's theorem (see \cite[Theorem $1.2.3$]{evans2015measure}) we know that for any $\delta>0$ there exists a compact set $H$ such that $y_{n_k} \to y$ uniformly on $H$ and $|[0,T]\setminus H|<\delta$. Moreover, let us define the measure $\mu$ on $[0,T]$ such that for any Lebesgue-measurable set $A \subseteq [0,T]$ it holds
\begin{equation*}
	\mu(A)=\int_{A}\Psi(y(t))dt,
\end{equation*}
that is to say the measure $\mu$ is defined via $\frac{d \mu}{dt}=\Psi(y(t))$.  In particular, $\mu$ is absolutely continuous with respect to the Lebesgue measure. Fix $\varepsilon>0$. By absolute continuity there exists $\delta>0$ such that for any measurable set $A \subseteq [0,T]$, $|A|<\delta$ implies $\mu(A)<\varepsilon$. Let us consider the compact set $H$ obtained by Egorov's theorem such that $|[0,T]\setminus H|<\delta$. In particular, we get
\begin{equation*}
\cF_\Psi[y]=\mu(H)+\mu([0,T]\setminus H)<\mu(H)+\varepsilon.
\end{equation*}
Being $\cF_\Psi[y]<+\infty$ we have that
\begin{equation*}
	\mu(H)>\cF_\Psi[y]-\varepsilon.
\end{equation*}
On the other hand, it also holds
\begin{equation*}
	\cF_\Psi[y_{n_k}]=\int_{0}^{T}\Psi(y_n(t))dt\ge \int_{H} \Psi(y_{n_k}(t))dt.
\end{equation*}
Since $y_{n_k}\to y$ uniformly on $H$ and $\Psi$ is continuous, we also have $\Psi(y_{n_k})\to \Psi(y_{n})$ uniformly on $H$ and, taking the limit, we achieve
\begin{equation*}
	\liminf_{n\to +\infty}\cF_\Psi[y_{n}]=\lim_{k \to +\infty}\cF_\Psi[y_{n_k}]\ge \int_{H} \Psi(y(t))dt=\mu(H)>\cF_\Psi[y]-\varepsilon.
\end{equation*}
Being $\varepsilon>0$ arbitrary, we conclude the proof in the case $\cF_\Psi[y]<+\infty$.\\
If $\cF_\Psi[y]=+\infty$, let us consider the sequence of measurable sets $S_m=\{t \in [0,T]: \ \Psi(y(t))\le m\}$ for $m \in \N$, so that
\begin{equation*}
	\int_{S_m}\Psi(y(t))dt\le mT<+\infty.
\end{equation*}
However, by monotone convergence theorem
\begin{equation*}
	\lim_{m\to +\infty}\int_{S_m}\Psi(y(t))dt=\cF_\Psi[y]=+\infty.
\end{equation*}
Thus, for any $M>0$ there exists $m \in \N$ such that
\begin{equation*}
	M<\int_{S_m}\Psi(y(t))dt<+\infty.
\end{equation*}
For any Lebesgue-measurable set $A \subset [0,T]$ define $\sigma_m(A)=|A \cap S_m|$ and
\begin{equation*}
	\mu_m(A)=\int_{A \cap S_m}\Psi(y(t))dt
\end{equation*}
that are two positive measures with $\mu_m \ll \sigma_m$ and $\sigma_m([0,T])\le T<+\infty$. We can argue as before, applying Egorov's theorem to $\sigma_m$, to achieve
\begin{equation*}
	\liminf_{n \to +\infty}\int_{S_m}\Psi(y_n(t))dt\ge \int_{S_m}\Psi(y(t))dt-\varepsilon>M-\varepsilon
\end{equation*}
for any $\varepsilon>0$. On the other hand
\begin{equation*}
	\cF_\Psi[y_n]\ge \int_{S_m}\Psi(y_n(t))dt
\end{equation*}
and then
\begin{equation*}
	\liminf_{n \to +\infty}\cF_\Psi[y_n]>M-\varepsilon
\end{equation*}
for any $\varepsilon>0$. Being $\varepsilon>0$ arbitrary, we get
\begin{equation*}
	\liminf_{n \to +\infty}\cF_\Psi[y_n]\ge M.
\end{equation*}
that leads to
\begin{equation*}
	\liminf_{n \to +\infty}\cF_\Psi[y_n]=+\infty=\cF[y].
\end{equation*}
Finally, it is well-known that convex lower semicontinuous functions are also weakly lower semicontinuous, as a consequence of Mazur's theorem (see \cite{dacorogna2007direct}).
\end{proof}
\section*{Acknowledgements}
\vspace*{-0.3cm}
This research is partially supported by MIUR - PRIN 2017, project Stochastic Models for Complex Systems, no. 2017JFFHSH and by Gruppo Nazionale per l'Analisi Matematica, la Probabilit\`a e le loro Applicazioni
(GNAMPA-INdAM).

\bibliographystyle{abbrv}
\bibliography{biblio}

\begin{thebibliography}{10}

\bibitem{Andersson}
D.~Andersson and B.~Djehiche.
\newblock A maximum principle for {SDE}s of mean-field type.
\newblock {\em Applied Mathematics \& Optimization}, 63:341--356, 2011.

\bibitem{angrisani2019appunti}
F.~Angrisani, G.~Ascione, C.~Leone, and C.~Mantegazza.
\newblock {\em Appunti di Calcolo delle Variazioni}.
\newblock Amazon, 2019.
\newblock Lecture notes, Dipartimento di Matematica e Applicazioni “Renato
  Caccioppoli” dell’Università Federico II di Napoli.

\bibitem{arnold1998random}
L.~Arnold.
\newblock {\em Random Dynamical Systems}.
\newblock Monographs in Mathematics. Springer, 1998.

\bibitem{ascione2020optimal}
G.~Ascione, G.~D’Onofrio, L.~Kostal, and E.~Pirozzi.
\newblock An optimal {G}auss--{M}arkov approximation for a process with
  stochastic drift and applications.
\newblock {\em Stochastic Processes and their Applications},
  130(11):6481--6514, 2020.

\bibitem{asmussen2007stochastic}
S.~Asmussen and P.~W. Glynn.
\newblock {\em Stochastic simulation: algorithms and analysis}, volume~57.
\newblock Springer Science \& Business Media, 2007.

\bibitem{atkinson1997numerical}
K.~E. Atkinson.
\newblock {\em The Numerical Solution of Integral Equations of the Second
  Kind}, volume~4.
\newblock Cambridge University Press, 1997.

\bibitem{bayraktar2018randomized}
E.~Bayraktar, A.~Cosso, and H.~Pham.
\newblock Randomized dynamic programming principle and {F}eynman-{K}ac
  representation for optimal control of {M}c{K}ean-{V}lasov dynamics.
\newblock {\em Transactions of the American Mathematical Society},
  370(3):2115--2160, 2018.

\bibitem{bellman}
R.~Bellman.
\newblock Dynamic programming and stochastic control processes.
\newblock {\em Information and Control}, 1(3):228--239, 1958.

\bibitem{berger1977nonlinearity}
M.~S. Berger.
\newblock {\em Nonlinearity and functional analysis: lectures on nonlinear
  problems in mathematical analysis}, volume~74.
\newblock Academic press, 1977.

\bibitem{bhattacharya2007random}
R.~Bhattacharya and M.~Majumdar.
\newblock {\em Random Dynamical Systems: Theory and Applications}.
\newblock Cambridge University Press, 2007.

\bibitem{Bonnans2012FirstAS}
J.~F. Bonnans and F.~J. Silva.
\newblock First and second order necessary conditions for stochastic optimal
  control problems.
\newblock {\em Applied Mathematics \& Optimization}, 65:403--439, 2012.

\bibitem{braides2002gamma}
A.~Braides.
\newblock {\em Gamma-convergence for Beginners}, volume~22.
\newblock Clarendon Press, 2002.

\bibitem{brezis2010functional}
H.~Brezis.
\newblock {\em Functional analysis, Sobolev spaces and partial differential
  equations}.
\newblock Springer Science \& Business Media, 2010.

\bibitem{dipersio}
F.~Cordoni and L.~Di~Persio.
\newblock A maximum principle for a stochastic control problem with multiple
  random terminal times.
\newblock {\em Mathematics in Engineering}, 2:557, 2020.

\bibitem{dacorogna2007direct}
B.~Dacorogna.
\newblock {\em Direct methods in the calculus of variations}, volume~78.
\newblock Springer Science \& Business Media, 2007.

\bibitem{de2008semicontinuity}
E.~De~Giorgi.
\newblock {\em Semicontinuity theorems in the calculus of variations}.
\newblock Quaderni dell'Accademia Pontaniana. Accademia Pontaniana, 2008.

\bibitem{evans2015measure}
L.~C. Evans and R.~F. Gariepy.
\newblock {\em Measure theory and fine properties of functions}.
\newblock CRC press, 2015.

\bibitem{tanre}
O.~Faugeras, E.~Soret, and E.~Tanr\'e.
\newblock Asymptotic behaviour of a network of neurons with random linear
  interactions.
\newblock {\em arXiv: Probability}, 2019.

\bibitem{flandoli}
F.~Flandoli and E.~Tonello.
\newblock An introduction to random dynamical systems for climate.
\newblock
  \url{https://courseclimath19.sciencesconf.org/data/pages/RDS_4_Climate_Flandoli_v2.pdf},
  2019.
\newblock Online, last accessed on 07/10/2021.

\bibitem{Fuhrman2013StochasticMP}
M.~Fuhrman, Y.~Hu, and G.~Tessitore.
\newblock Stochastic maximum principle for optimal control of {SPDE}s.
\newblock {\em ArXiv}, abs/1302.0286, 2013.

\bibitem{Fuhrman2016StochasticMP}
M.~Fuhrman and C.~Orrieri.
\newblock Stochastic maximum principle for optimal control of a class of
  nonlinear spdes with dissipative drift.
\newblock {\em SIAM J. Control. Optim.}, 54:341--371, 2016.

\bibitem{delarue}
{Grazieschi, P.}, {Leocata, M.}, {Mascart, C.}, {Chevallier, J.}, {Delarue,
  F.}, and {Tanr\'{e}, E.}
\newblock Network of interacting neurons with random synaptic weights.
\newblock {\em ESAIM: ProcS}, 65:445--475, 2019.

\bibitem{davis}
M.~H.A.Davis and G.~Burstein.
\newblock A deterministic approach to stochastic optimal control with
  application to anticipative control.
\newblock {\em Stochastics and Stochastic Reports}, 40(3-4):203--256, 1992.

\bibitem{JOHNSON2021}
P.~Johnson, J.~Pedersen, G.~Peskir, and C.~Zucca.
\newblock Detecting the presence of a random drift in brownian motion.
\newblock {\em Stochastic Processes and their Applications}, 2021.

\bibitem{Kalman}
R.~E. Kalman.
\newblock The theory of optimal control and the calculus of variations.
\newblock In R.~Bellman, editor, {\em Mathematical Optimization Techniques},
  pages 309--332. University of California Press, 2021.

\bibitem{kazamaki2006continuous}
N.~Kazamaki.
\newblock {\em Continuous exponential martingales and {BMO}}.
\newblock Springer, 2006.

\bibitem{lions1983}
P.~L. Lions.
\newblock Optimal control of diffusion processes and
  {H}amilton–{J}acobi–{B}ellman equations part 2 : viscosity solutions and
  uniqueness.
\newblock {\em Communications in Partial Differential Equations},
  8(11):1229--1276, 1983.

\bibitem{matlab}
The Mathworks, Inc., Natick, Massachusetts.
\newblock {\em {MATLAB version 9.10.0.1710957 (R2021a) Update 4}}, 2021.

\bibitem{MenoukeuPamen2021MaximumPF}
O.~Menoukeu-Pamen and L.~Tangpi.
\newblock Maximum principle for stochastic control of {SDE}s with measurable
  drifts.
\newblock {\em arXiv preprint arXiv:2101.06205}, 2021.

\bibitem{meyer1966probability}
P.~A. Meyer.
\newblock {\em Probability and potentials}, volume 1318.
\newblock Blaisdell Publishing Company, 1966.

\bibitem{oksendal2013stochastic}
B.~{\O}ksendal.
\newblock {\em Stochastic differential equations: an introduction with
  applications}.
\newblock Springer Science \& Business Media, 2013.

\bibitem{oksendal2000}
B.~{\O}ksendal and A.~Sulem.
\newblock A maximum principle for optimal control of stochastic systems with
  delay, with applications to finance.
\newblock In {\em Optimal control and partial differential equations (Paris, 4
  December 2000)}, pages 64--79, 2001.

\bibitem{orrieri}
C.~Orrieri.
\newblock A stochastic maximum principle with dissipativity conditions.
\newblock {\em Discrete \& Continuous Dynamical Systems}, 35(11):5499--5519,
  2015.

\bibitem{pardouxpeng}
E.~Pardoux and S.~Peng.
\newblock Adapted solution of a backward stochastic differential equation.
\newblock {\em Systems \& Control Letters}, 14(1):55--61, 1990.

\bibitem{peng}
S.~Peng.
\newblock A general stochastic maximum principle for optimal control problems.
\newblock {\em Siam Journal on Control and Optimization}, 28:966--979, 1990.

\bibitem{pham_survey}
H.~Pham.
\newblock {On some recent aspects of stochastic control and their
  applications}.
\newblock {\em Probability Surveys}, 2:506 -- 549, 2005.

\bibitem{pick2012function}
L.~Pick, A.~Kufner, O.~John, and S.~Fuc{\'\i}k.
\newblock {\em Function Spaces}, volume~1.
\newblock Walter de Gruyter, 2012.

\bibitem{revuz2013continuous}
D.~Revuz and M.~Yor.
\newblock {\em Continuous martingales and Brownian motion}, volume 293.
\newblock Springer Science \& Business Media, 2013.

\bibitem{rudin1987real}
W.~Rudin.
\newblock {\em Real and Complex Analysis, 3rd Ed.}
\newblock McGraw-Hill, Inc., USA, 1987.

\bibitem{stannat}
W.~Stannat and L.~Wessels.
\newblock Deterministic control of stochastic reaction-diffusion equations.
\newblock {\em Evolution Equations \& Control Theory}, 2020.

\bibitem{STECHA20114714}
J.~Štecha and J.~Rathouský.
\newblock Stochastic maximum principle.
\newblock {\em IFAC Proceedings Volumes}, 44(1):4714--4720, 2011.

\bibitem{tonelli1921fondamenti}
L.~Tonelli.
\newblock {\em Fondamenti di Calcolo delle Variazioni I.}
\newblock Zanichelli, 1921.

\bibitem{whitt2002stochastic}
W.~Whitt.
\newblock {\em Stochastic-process limits: an introduction to stochastic-process
  limits and their application to queues}.
\newblock Springer Science \& Business Media, 2002.

\bibitem{wing1991primer}
G.~M. Wing.
\newblock {\em A primer on integral equations of the first kind: the problem of
  deconvolution and unfolding}.
\newblock SIAM, 1991.

\bibitem{wongzakai}
E.~Wong and M.~Zakai.
\newblock On the relation between ordinary and stochastic differential
  equations.
\newblock {\em International Journal of Engineering Science}, 3(2):213--229,
  1965.

\bibitem{yong1999stochastic}
J.~Yong and X.~Y. Zhou.
\newblock {\em Stochastic controls: Hamiltonian systems and HJB equations},
  volume~43.
\newblock Springer Science \& Business Media, 1999.

\end{thebibliography}
\end{document}